\documentclass[10pt]{article}

\usepackage{titlesec}
\usepackage{amssymb,amsmath,amsfonts,amsthm,enumerate,color,mathrsfs,amssymb,extarrows,empheq,cite}
\usepackage[shortlabels]{enumitem}
\usepackage[toc,title,titletoc,header]{appendix}

\usepackage{caption}
\usepackage{etoolbox}
\patchcmd{\thebibliography}
  {\settowidth}
  {\setlength{\itemsep}{0pt plus 0.1pt}\settowidth}
  {}{}
\apptocmd{\thebibliography}
  {\small}
  {}{}

\usepackage[usenames,dvipsnames]{xcolor}

\usepackage{esint}

\usepackage[hidelinks]{hyperref}
\hypersetup{colorlinks = true,
linkcolor = MidnightBlue, anchorcolor =red,
citecolor = ForestGreen,
filecolor = red,urlcolor = red,
            pdfauthor=author}
\usepackage{cleveref}

\usepackage{graphicx}
\usepackage{indentfirst}
\usepackage{multicol}
\usepackage{booktabs} 
\usepackage{float} 
\usepackage{subfigure}

\theoremstyle{plain}

\newtheorem{theorem}{Theorem}[section]
\newtheorem{lemma}[theorem]{Lemma}
\newtheorem{proposition}[theorem]{Proposition}

\numberwithin{equation}{section}

\theoremstyle{definition}
\newtheorem{definition}[theorem]{Definition}
\newtheorem{remark}[theorem]{Remark}

\newcommand{\norm}[1]{\lVert #1 \rVert}
\newcommand{\normm}[1]{\left\lVert #1 \right\rVert}

\newcommand{\mb}[1]{{\color{black} #1}}

\def\q {\quad}

\def \l{\langle}
\def \r{\rangle}

\def\bb{\begin{equation}
  \left\{\ 
   \begin{aligned}}
\def\ee{   \end{aligned}
  \right.
  \end{equation}}

\def\mm{ \left[
 \begin{matrix}}
\def\nn{\end{matrix} \right] } 

\def\p{\partial}
\def \dd{\cdot}
\def \t{\times}

\def\n{\nu}
\def \w {\widetilde}
\def \h{\widehat}
\def \ol{\overline}

\def\d{\delta}
\def \vp{\varphi}
\def \si{\sigma}
\def \na{\nabla}
\def \ep{\varepsilon}
\def \lad{\lambda}
\def \Lad{\Lambda}

\def \ddiv {{\rm div}}

\def \sdiv {\sdiv_S}

\def \sd {\mathsf{D}_\si}

\def \NN{\mathbb{N}}
\def \R{\mathbb{R}}

\def \S{\mathbb{S}}
\def \AA{\mathbb{A}}

\def \P{\mathbb{P}}
\def \xx {\mathbb{X}}
\def \Z{\mathbb{Z}}

\def \T{\mathbb{T}}
\def \MM{\mathbb{M}}

\def \TT{\mathcal{T}}

\def \ss{\mathcal{S}}
\def \M{\mathcal{M}}

\def \L{\mathcal{L}}

\def \ce{\mathcal{CE}}

\def \rr{\mathcal{R}}

\def \ff{{\rm F}}

\def \ww{\omega}

\def \V{\mathcal{V}}

\def \hh{\Delta}
\def \jj{\mathcal{J}}
\def \ts{{\tau,\sigma}}

\DeclareMathOperator{\dom}{dom}
\DeclareMathOperator{\tr}{Tr}

\def \mc {\mathcal}

\def \ran {{\rm Ran}}
\def \ker {{\rm Ker}}

\def \sym {{\rm sym}}

\def \ms{\mathsf}
\def \rd {{\rm d}}

\def \Lip{{\rm Lip}}

\title{On the convergence of discrete dynamic unbalanced transport models}

\begin{document}
\author{}
\author{
Bowen Li\footnote{Department of Mathematics, Duke University, Durham, NC 27708, USA. (bowen.li200@duke.edu).}
\and Jun Zou\footnote{Department of Mathematics, The Chinese University of Hong Kong, Shatin, N.T., Hong Kong. 
This author was substantially supported by Hong Kong RGC General Research Fund (projects 14308322 and 14306921) 
and NSFC/Hong Kong RGC Joint Research Scheme 2022/23 (project N\_CUHK465/22). 
(zou@math.cuhk.edu.hk).}
}
\date{}
\maketitle

\begin{abstract}
A generalized unbalanced optimal transport distance ${\rm WB}_{\Lad}$ on 
 matrix-valued measures $\mc{M}(\Omega,\S_+^n)$ was defined in \cite{li2020general1} \`{a} la Benamou-Brenier, which extends the Kantorovich-Bures and the Wasserstein-Fisher-Rao distances. In this work, we investigate the convergence properties of the discrete transport problems associated with ${\rm WB}_{\Lad}$. We first present a convergence framework for abstract discretization. Then, we propose a specific discretization scheme that aligns with this framework, whose convergence relies on the assumption that the initial and final distributions are absolutely continuous with respect to the Lebesgue measure. \mb{Further, in the case of the Wasserstein-Fisher-Rao distance, thanks to the static formulation, we show that such an assumption can be removed.}
\end{abstract}

\section{Introduction}

The optimal transport (OT) problem was initially proposed by Monge in 1781 \cite{monge1781memoire}, and later Kantorovich relaxed it as linear programming \cite{kantorovich1942translocation}, leading to the static formulation of the $2$-Wasserstein distance $\mc{W}_2$. An equivalent dynamic formulation was derived in the seminal work by Benamou and Brenier \cite{benamou2000computational}. Nowadays, OT has become a very active research topic in both theoretical and applied mathematics; see monographs \cite{villani2003topics, villani2008optimal} for a recent overview. While the OT theory has gained popularity in the science community, a known restriction is that \mb{it is only well-defined for} the probabilities (i.e., measures of equal mass). There have been many efforts in the last decades to extend the theory to the case of general positive measures, called the unbalanced optimal transport (UOT). Figalli and Gigli \cite{figalli2010new} extended the Kantorovich formulation of the Wasserstein distance and introduced a UOT model by permitting mass transfer to or from the boundary of the domain. This approach is closely related to the optimal partial transport studied in \cite{figalli2010optimal,caffarelli2010free} by Figalli et al. Another research line for defining unbalanced transport distances is based on the  Benamou-Brenier dynamic formulation, where a source term for the mass change is added to the continuity equation \cite{lombardi2015eulerian,maas2015generalized,benamou2003numerical,chizat2018interpolating,piccoli2016properties}. Piccoli and Rossi \cite{piccoli2014generalized,piccoli2016properties}  introduced a generalized Wasserstein distance in both static and dynamic formulations, which is equivalent to the optimal partial transport under certain conditions \cite{chizat2018interpolating}. More recently, Chizat et al. \cite{chizat2018interpolating} defined the so-called Wasserstein-Fisher-Rao (WFR) distance motivated by the applications in imaging. The same model was introduced
and investigated almost simultaneously by Liero et al. \cite{liero2016optimal,
kondratyev2016new} via different techniques. Later, a unified theory of UOT was established in \cite{chizat2018unbalanced,liero2018optimal} with the concepts of semi-couplings and optimal entropy-transport, which generalizes the optimal partial transport and the WFR distance thoroughly. \mb{We also refer the readers to \cite{fu2023high,li2022computational,gao2022master,lee2021generalized,gangbo2019unnormalized} for some other generalized Wasserstein distances using reaction-diffusion equations and their fast computations with applications in gradient flow and mean-field control problems.}

In addition, there has been a growing interest in extending the theory of OT to the noncommutative realm, with two principal research directions.  The first one is inspired by the ergodic quantum Markov dynamics \cite{gross1975hypercontractivity,kastoryano2013quantum}. Carlen and Maas \cite{carlen2014analog,carlen2017gradient} proposed a quantum analog of Wasserstein distance in the spirit of Benamou-Brenier such that the detailed-balanced quantum Markov semigroup is the gradient flow of the quantum log-relative entropy; see \cite{li2022interpolation} for the extension to a general family of quantum convex relative entropies. 
Meanwhile, a Kantorovich-type quantum transport model was suggested by Golse et al. for quantifying the 
mean-field and classical limits of the Schr\"{o}dinger equation \cite{golse2016mean,golse2017schrodinger,golse2018wave}. The second line of research is motivated by diffusion magnetic resonance imaging, where a tensor is calculated for each voxel \cite{le2014diffusion,wandell2016clarifying}. To compare two tensor fields, Chen et al. \cite{chen2017matrix, chen2020matrix, chen2019interpolation} gave various extensions of the Wasserstein distance to matrix-valued measures in a dynamic formulation, with the help of the Lindblad equation in quantum mechanics. Moreover, an entropy-regularized static quantum optimal transport was proposed by Peyr\'{e} et al. \cite{peyre2019quantum}, where a scaling algorithm generalized from \cite{chizat2018scaling} and applications in diffusion tensor imaging were also investigated; see also \cite{ryu2018vector} for the matrix transport model of order one. We finally emphasize the recent work by Brenier and Vorotnikov \cite{brenier2020optimal}, which is the initial motivation of the percent work, where the authors proposed the so-called Kantorovich-Bures metric on unbalanced matrix-valued measures, with the inspiration drawn from the fact that the incompressible Euler equation can be formulated as a concave maximization problem \cite{brenier2018initial,vorotnikov2022partial}.

In the companion work \cite{li2020general1}, a general unbalanced transport model ${\rm WB}_{\Lad}$ on the space $\mc{M}(\Omega,\S_+^n)$ of positive semi-definite matrix-valued Radon measures, called the weighted Wasserstein-Bures distance, was proposed via an abstract matricial continuity equation \cite[Definition 3.4]{li2020general1}:
$$
\p_t \ms{G} + \ms{D} \ms{q} = \ms{R}^{{\sym}}\,,
$$
and a weighted infinitesimal cost \cite[Proposition 3.1]{li2020general1}: 
$$ 
J_\Lad(G_t, q_t, R_t) = 
\frac{1}{2} (q_t \Lad_1^\dag) \dd G_t^{\dag} (q_t \Lad_1^\dag) + \frac{1}{2} (R_t \Lad_2^\dag) \dd G^{\dag}_t (R_t \Lad_2^\dag)\,.
$$
A similar matrix-valued optimal ballistic transport problem has been considered in \cite{vorotnikov2022partial}; see \cite[Section 7]{li2020general1} for the detailed connections between models in \cite{vorotnikov2022partial,li2020general1}. A comprehensive study of the topological and metric
properties of the model ${\rm WB}_\Lad$ has been given in \cite{li2020general1}, where the authors proved that $(\mc{M}(\Omega, \S_+^n), {\rm WB}_\Lad)$ is a complete geodesic space with a conic structure. These theoretical results justify the reasonability of the weighted distance ${\rm WB}_{\Lad}$. In particular, with specially chosen parameters, the Kantorovich-Bures distance \cite{brenier2020optimal}, the WFR distance \cite{chizat2018interpolating,liero2016optimal,kondratyev2016new}, and the matricial interpolation distance \cite{chen2019interpolation} can fit into the general model ${\rm WB}_\Lad$; see \cite[Section 7]{li2020general1} and also Remark \ref{rem:connect} below. 

This work is devoted to the numerical analysis of the matrix-valued transport model ${\rm WB}_\Lad$. For the classical OT model $\mc{W}_2$, an augmented Lagrangian algorithm was proposed in the seminal work \cite{benamou2000computational}. It has been investigated systematically in the general framework of proximal splitting methods \cite{papadakis2014optimal} and extended to other related problems, such as mean field games \cite{benamou2015augmented} and unbalanced optimal transport \cite{chizat2018interpolating}. In addition, the convergence analysis for this algorithm in the Hilbert space setting was carried out in \cite{guittet2003time,hug2020convergence}.

In this work, our focus lies in studying the convergence of the discretization scheme for the OT model, as it constitutes the fundamental step towards developing a fast algorithm for solving the resulting discrete convex optimization problem. The first result in this direction may be attributed to Gigli and Maas \cite{gigli2013gromov}. They considered a discrete transport distance previously introduced by Maas \cite{maas2011gradient} (which enables the interpretation of finite Markov chains as gradient flows) and demonstrated that the discrete transport metric associated with the random walk on the discrete torus converges to the continuous 2-Wasserstein distance in the sense of Gromov-Hausdorff. It turns out that such a discrete distance coincides with an isotropic finite-volume discretization of the dynamic $\mc{W}_2$ on the torus. These results were later extended by Gladbach et al. \cite{gladbach2020scaling}, where a two-point flux finite volume scheme on a bounded convex domain was considered. Note that by Gromov-Hausdorff convergence, we also have the convergence of 
discrete geodesics (minimizers). Erbar et al. \cite{erbar2020computation} further studied the time discretization based on the semi-discrete approximation in \cite{gladbach2020scaling,gigli2013gromov} and proved a Gamma-convergence when the temporal mesh size goes to zero. It is possible to combine the results in \cite{gladbach2020scaling,erbar2020computation} to obtain a convergent fully discretization scheme for dynamic $\mc{W}_2$, which, however, would involve a restriction on the ratio between the spatial and temporal mesh sizes. 
Carrillo et al. \cite{carrillo2022primal} computed the Wasserstein gradient flow with a finite difference scheme and proved the Gamma-convergence of the fully discrete JKO scheme, which also applies to dynamic $\mc{W}_2$ but relies on a strong regularity assumption on the minimizer. Recently,  inspired by the results in \cite{carrillo2022primal,gladbach2020scaling,erbar2020computation}, Lavenant \cite{lavenant2019unconditional} proposed a general convergence framework for dynamic $\mc{W}_2$ without any assumptions on the regularity and the mesh size, which has been applied to the finite element discretization on the surface \cite{lavenant2018dynamical}, the mixed finite element discretization \cite{natale2022mixed}, and the finite volume scheme \cite{gladbach2020scaling}. We also mention the work \cite{natale2021computation} by 
Natale and Todeschi, where the authors observed numerical instabilities for the finite volume schemes in \cite{gigli2013gromov,gladbach2020scaling} and introduced a new scheme based on a two-level spatial discretization to overcome this issue. 

We continue the investigation of the convergence of discrete transport models in this work. We aim to design \mb{a unified convergent} discretization scheme for the general model ${\rm WB}_\Lad$ (see \eqref{eq:distance} below). The results apply to the Kantorovich-Bures distance \eqref{def:kan_bure}, the matricial interpolation distance \cite{chen2019interpolation}, and the WFR metric \eqref{def:wfr_metric} directly. To achieve this, we first extend the convergence framework proposed in \cite{lavenant2019unconditional} for the 2-Wasserstein distance $\mc{W}_2$ to ${\rm WB}_\Lad$ from the perspective of Lax equivalence theorem in an abstract manner. In particular, we prove the convergence in Theorem \ref{thm:main_conver} for an abstract scheme under a set of verifiable conditions. Then, we propose a concrete fully discrete scheme for our model \eqref{eq:distance}, fitting into this abstract convergence framework, and thus the convergence can be guaranteed; see Theorem \ref{thmconcrete}. However, it requires an additional absolute continuity assumption on the initial and final distributions. For the WFR distance \eqref{def:wfr_metric}, we manage to remove this assumption by leveraging its equivalent static formulation; see Section \ref{sec:specwfr}.

The rest of this work is organized as follows. In Section \ref{sec:notation}, we review the transport distance ${\rm WB}_{\Lad}$ proposed in \cite{li2020general1} with a regularization result Proposition \ref{prop:regular_ep}, which is necessary for the numerical treatment. Section \ref{sec:absconver} is devoted to an abstract convergence framework, while in Section \ref{sec:conver_example} a concrete convergent discretization scheme was suggested. In Section \ref{sec:specwfr}, sharper results are obtained for the WFR distance.

\section{Transport distances on matrix-valued measures} \label{sec:notation}

In this section, we review the general matrix-valued unbalanced transport model defined in \cite{li2020general1} and its connections with existing models \cite{brenier2020optimal,chen2019interpolation,chizat2018interpolating}. We first fix the notations used throughout the work.

\begin{itemize}
 \setlength\itemsep{-0.1mm}
    \item Let $\R^{n \t m}$ be the space of $n \t m$ real matrices. If $m = n$, we write it as $\MM^{n}$. We denote by $\S^n$, $\S_+^n$, and $\S^n_{++}$ the symmetric matrices, positive semi-definite matrices, and positive definite matrices, respectively. $\AA^n$ denotes the space of $n \t n$ antisymmetric matrices. 
    \item We denote by $|\dd|$ the Euclidean norm on $\R^n$ and equip the matrix space $\R^{n \t m}$ with the Frobenius inner product $A \dd B = \tr(A^{\rm T} B)$ and norm  $\norm{A}_\ff = \sqrt{ A \dd A}$. 
    \item The pseudoinverse of a matrix $A \in \R^{n \t m}$ is denoted by $A^\dag \in \R^{m \t n}$ . If $A \in \S^n$ has the eigendecomposition $A = O \Sigma O^{\rm T}$ with $\Sigma = \text{diag}(\lad_1,\ldots, \lad_s,0, \ldots,0)$, then $A^\dag = O \Sigma^\dag O^{\rm T}$ with $\Sigma^\dag = \text{diag}(\lad_1^{-1}, \ldots, \lad_s^{-1},0, \ldots,0)$, where $O$ is an orthogonal matrix and $\{\lad_i\}$ are nonzero eigenvalues of $A$. 

    \item For a matrix $A \in \MM^n$, its symmetric and antisymmetric parts are given by 
    \begin{equation} \label{eq1}
        A^{\sym} = (A + A^{\rm T})/2\,,\q A^{{\rm ant}} = (A - A^{\rm T})/2\,,
    \end{equation}
    respectively. We write $A \preceq B$ (resp., $A \prec B$) for $A, B \in \S^n$ if $B - A \in \S^n_+$ (resp., $B - A \in \S^n_{++}$). 
    \item The following Powers-St{\o}rmer inequality \cite{powers1970free} is useful: 
\begin{align} \label{est:lip_squre}
    \big \lVert\sqrt{A} - \sqrt{B}\,\big\rVert_\ff^2 \le \sqrt{n} \norm{A - B}_\ff\,, \q \forall A,B \in \S^n_+\,.
\end{align}
    \item Let $\mathcal{X}$ be a compact separable metric space with Borel $\sigma$-algebra $\mathscr{B}(\mathcal{X})$. We denote by $C(\mathcal{X},\R^n)$ the space of $\R^n$-valued continuous functions on $\mathcal{X}$ with the supremum norm  $\norm{\dd}_\infty$. Its dual space, denoted by $\M(\mathcal{X},\R^n)$, is the space of
    $\R^n$-valued Radon measures with total variation norm $\norm{\dd}_{\rm TV}$. 
\item Let $\mc{B}$ be a Banach space with dual space $\mc{B}^*$. We denote by $\l \dd, \dd \r_{\mc{B}}$ the duality pairing between $\mc{B}$ and $\mc{B}^*$. When $\mc{B} = C(\mathcal{X},\R^n)$, we usually write it as $\l \dd, \dd \r_{\mc{X}}$ for simplicity. The weak and weak* convergences on $\mc{B}$ and $\mc{B}^*$ are defined in the standard way. In particular, a sequence of measures $\{\mu_j\} \subset \M(\mathcal{X},\R^n)$ weak* converges to $\mu$ if there holds 
    $\l\mu_j,  \phi\r_{\mathcal{X}} \to \l\mu, \phi\r_{\mathcal{X}}$ for any $\phi \in C(\mathcal{X},\R^n)$, as $j \to +\infty$.  
\item Let $\R_+: = [0,\infty)$, and $\M(\mathcal{X},\R_+)$ be the space of nonnegative finite Radon measures. 
For $\mu \in \M(\mathcal{X},\R^n)$, we have  
an associated variation measure $|\mu|\in \M(\mathcal{X},\R_+)$ such that ${\rm d} \mu = \sigma {\rm d} |\mu|$ with $|\sigma(x)| = 1$ for $|\mu|$-a.e.\,$x \in \mathcal{X}$, where $\sigma: \mathcal{X} \to \R^n$ is 
the density of $\mu$ with respect to $|\mu|$ \cite{evans2015measure,rudin2006real}. 
\item One can identify the space of matrix-valued Radon measures  $\M(\mathcal{X},\R^{n \t m})$ with 
$\M(\mathcal{X},\R^{nm})$ by vectorization. 
Note from \cite[Theorem 3.5]{duran1997lpspace} that both sets of $\S^n$-valued measures $\M(\mathcal{X},\S^n)$ and $\S^n_+$-valued measures $\M(\mathcal{X},\S_+^n)$ are closed in $\M(\mathcal{X},\MM^n)$ with respect to the weak* topology. In addition, we have $(C(\mathcal{X}, \S^n))^* \simeq \M(\mathcal{X}, \S^n)$,
where $\simeq$ denotes the isometric isomorphism.
\item For $\mu \in \M(\mathcal{X}, \S_+^n)$, we define an associated trace measure
$\tr\mu$ by the set function $E \to \tr (\mu (E))$, $E \in \mathscr{B}(\mathcal{X})$. It is clear that $ 0 \preceq \mu(E) \preceq \tr (\mu (E)) I$ and $ \tr\mu$ is equivalent to
$|\mu|$, that is, $|\mu| \ll \tr\mu$ and $\tr\mu \ll |\mu|$. In what follows, we usually use $\tr\mu$ as the dominant measure for $\mu \in \M(\mathcal{X},\S_+^n)$. Also, note that for $\lad \in \M(\mathcal{X}, \R_+)$ with $|\mu| \ll \lad$, there holds \mb{$\frac{\rd \mu}{\rd \lad}(x) \in \S^n_+$} for $\lad$-a.e.\,$x \in \mathcal{X}$, which is an alternative definition of $\M(\mathcal{X}, \S_+^n)$.

\item We use $\ms{sans\ serif}$ letterforms to denote vector-valued or matrix-valued measures, e.g., $\ms{A} \in \M(\mathcal{X}, \MM^n)$, while letters with serifs are reserved for their densities with respect to some reference measure, e.g., $A_\lad: = \frac{\rd \ms{A}}{\rd \lad}$ for $|\ms{A}| \ll \lad$. The symmetric and antisymmetric parts $\ms{A}^\sym$ and $\ms{A}^{\rm ant}$ of $\ms{A} \in \M(\mathcal{X}, \MM^n)$ are defined as in \eqref{eq1}. 


\item Given $\lad \in \M(\mathcal{X}, \R_+)$, we denote by $L^p_\lad (\mathcal{X},\R^n)$ with $p \in [1, +\infty]$ the standard space of $p$-integrable $\R^n$-valued functions. For $\ms{G} \in \M(\mathcal{X}, \S_+^n)$, we consider $\R^{n \times m}$-valued measurable functions endowed with the semi-inner product $ \l P, Q \r_{L^2_{\ms{G}}(\mathcal{X})} := \int_\mathcal{X}  P \dd (\rd \ms{G}\, Q) = \int_\mathcal{X}  P \dd \big(G_\lad Q \big)\, \rd \lad$, where $\lad$ is a reference measure such that $|\ms{G}|\ll \lad$ and $G_\lad$ is the density. The Hilbert space $L^2_{\ms{G}}(\mathcal{X}, \R^{n \times m})$ is then defined as the quotient space by $\ker\big(\norm{\dd}_{L^2_{\ms{G}}(\mathcal{X})}\big)$. 

\item For Banach spaces $X$ and $Y$, we 
denote by $\mc{L}(X,Y)$ the bounded linear operators from $X$ to $Y$ and by $C_c^\infty(\R^d,X)$ the $X$-valued smooth functions with compact support. We also use the space $C^k(\Omega, X)$ of $C^k$-smooth functions with norm 
defined by $\norm{\Phi}_{k,\infty} := \sum_{|\alpha| \le k} \sup_{x \in \Omega} \norm{D^\alpha \Phi(x)}_X$, where the derivatives are assumed to exist in the interior of $\Omega$ with continuous extension to the boundary. 

\item We also recall some concepts from convex analysis. For a function $f: X \to \R \cup \{+ \infty\}$ on a Banach space $X$, we denote by $\dom(f) := f^{-1}(\R)$ its domain and $f$ is proper if $\dom(f) \neq \varnothing$. Moreover,  $f$ is positively homogeneous of degree $k$ if for all $x \in X$ and $\alpha > 0$, $f(\alpha x) = \alpha^k f(x)$.  The conjugate function $f^*$ of $f$ is defined by 
\begin{equation} \label{def:conjugate}
    f^*(x^*) = \sup_{x \in X} \l x^*, x\r_X - f(x)\,, \q \forall x^* \in X^*\,,
\end{equation}
which is convex and weak* lower semicontinuous.
\item We denote the indicator function of a set $A$ by
\begin{equation} \label{def:indicator}
    \chi_A(x) = \begin{cases}
        1, & \text{if}\ x \in A\,,\\
        0, & \text{if}\ x \notin A\,.
    \end{cases} 
\end{equation}
\end{itemize}

We next introduce the generalized dynamic unbalanced transport model on matrix-valued measures $\M(\Omega, \S^n_+)$ defined in \cite{li2020general1}, where $\Omega \subset \R^d$ is a compact set with nonempty interior and 
smooth boundary $\p \Omega$. The definition relies on a weighted action functional and an abstract matricial continuity equation. 

For this, let $\Lad := (\Lad_1,\Lad_2)$ be a pair of weight matrices with $\Lad_1 \in \S^k_+$ and $\Lad_2 \in \S_{++}^n$, where $n, k$ are positive integers. Suppose that $\mc{X}$ is a compact separable metric space. We denote $\xx := \S^n \t \R^{n \t k}  \t \MM^n$ for notational simplicity and then $\M(\mc{X}, \xx) = C(\mc{X}, \xx)^*$. We define a proper function on $\xx$ by
\begin{equation} \label{eq:expre_J}
    J_\Lad(X, Y, Z) = 
         \frac{1}{2}  (Y \Lad_1^\dag) \dd (X^{\dag} Y \Lad_1^\dag) 
          + \frac{1}{2} (Z \Lad_2^\dag) \dd (X^{\dag} Z \Lad_2^\dag)\,,
\end{equation}
for $X \in \S_+^n$, $\ran (Y^{\rm T})  \subset \ran ( \Lad_1)$, $\ran (Z^{\rm T})  \subset \ran (\Lad_2)$ and $\ran ([Y,Z]) \subset \ran(X)$ (otherwise $J_\Lad(X, Y, Z) = +\infty$). It was shown in \cite[Proposition 3.1]{li2020general1} that $ J_\Lad$ is the conjugate function \eqref{def:conjugate} of the characteristic function of the closed convex set $\mathcal{O}_\Lad=  \big\{(A,B,C) \in  \xx\,;\  A + \frac{1}{2} B \Lad^2_1 B^{\rm T} + \frac{1}{2} C \Lad^2_2 C^{\rm T}  \preceq 0 \big\}$ and hence it is convex, positively $1$-homogeneous, and lower semicontinuous. Then, for $\mu: = \ms{(G,q,R)} \in \M(\mc{X}, \xx)$, 
 we define a positive measure on $\mc{X}$ by 
\begin{equation} \label{def:costmeasure}
\jj_{\Lad,E}(\mu): = \int_E J_\Lad \left(\frac{\rd \mu}{\rd \lad}\right) \rd \lad\,,
\end{equation}
for measurable $E \in \mathscr{B}(\mc{X})$, where $\lad \in \M(\mc{X},\R_+)$ is a reference measure with $|\mu| \ll \lad$. Note that this definition \eqref{def:costmeasure} is independent of $\lad$, due to the positive homogeneity of $J_\Lad$. We  define the $\Lad$-weighted \emph{action functional} for $\mu \in \M(\mc{X}, \xx)$ by the total measure $\jj_{\Lad,\mc{X}}(\mu)$. 

To introduce the matricial continuity equation, let $\ms{D}^*:C_c^\infty(\R^d, \S^n) \to C_c^\infty(\R^d, \R^{n \t k})$ be a  first-order constant coefficient linear differential operator with $\ms{D}^*(I) = 0$. By a slight abuse of notation, we define $\ms{D}^*\Phi$ for functions $\Phi (t,x)$ on $\R^{1+d}$ by acting it on the spatial variable $x$ and define the operator $\ms{D}$ by the adjoint of $ - \ms{D}^*$ (e.g., $\ms{D} = \ddiv$ for $\ms{D}^* = \na$). We also denote by $\ms{D}_0$ and $\ms{D}_1$ the homogeneous parts of degree $0$ and $1$ of $\ms{D}$, whose symbols are denoted by $\widehat{\ms{D}_0}$ and $\widehat{\ms{D}_1}$, respectively; see \cite{mitrea2013distributions}, also \cite[(3.11)]{li2020general1}, for the definition of symbol. In addition, let $$Q := [0,1] \t \Omega \subset \R^{1 + d}$$ be the time-space domain, and for functions $\Phi(t,x)$ on $Q$, we often write $\Phi_t(\dd) := \Phi(t,\dd)$. Now, given initial and final distributions $\ms{G}_0, \ms{G}_1 \in \M(\Omega, \S^n_+)$, we define the set $\ce([0,1];\ms{G}_0, \ms{G}_1)$ by the measures $\ms{(G,q, R)} \in \M(Q,\xx)$ that satisfy the following  matrix-valued \emph{continuity equation}:
\begin{equation}\label{eq:weak_ctneq}
        \int_{Q} \p_t\Phi \dd \rd \ms{G} + \ms{D}^* \Phi  \dd  \rd \ms{q}  +  \Phi   \dd  \rd \ms{R} = \int_{\Omega}  \Phi_1  \dd \rd \ms{G}_1 -  \int_{\Omega}  \Phi_0 \dd \rd \ms{G}_0\,,\q \forall \Phi \in C^1(Q,\S^n)\,.
\end{equation} 
If \eqref{eq:weak_ctneq} holds, we say that $\ms{G} \in \M(Q, \S^n)$ connects $\ms{G}_0$ and $\ms{G}_1$.

\mb{It is worth noting that} the distributional equation of \eqref{eq:weak_ctneq} is $\p_t \ms{G} + \ms{D}\ms{q} = \ms{R}^{{\rm sym}}$ with a homogeneous \mb{boundary condition on $\p \Omega$} for the measure $\ms{q}$, which extends the classical no-flux condition $\nu \dd \ms{q} = 0$. Indeed, if the measure $\ms{q}$ has a smooth density $q$ with respect to the Lebesgue measure, this homogeneous boundary condition can be formulated as  $\widehat{\ms{D}_1}(-i \n)(q) = 0$, \mb{thanks to the integration by parts:
\begin{equation*}
    \int_\Omega \ms{D} q \dd \Phi + q \dd \ms{D}^* \Phi \, \rd x = \int_{\p \Omega} q \dd \widehat{\ms{D}_1^*}(-i \n)(\Phi)  \,  \rd x = \int_{\p \Omega} \widehat{\ms{D}_1}(-i \n)(q) \dd \Phi  \, \rd x\,,  \q \forall \Phi \in C^1(\Omega, \S^n)\,,
\end{equation*}  
where $\nu$ denotes the exterior unit normal to $\p \Omega$. Then, it is easy to see that for $\ms{D} = \ms{D}_1 = \ddiv$, the above boundary condition $\widehat{\ms{D}_1}(-i \n)(q) = 0$ reduces to the familiar one $\nu \dd q = 0$.}

We are ready to define the weighted Wasserstein-Bures distance on $\M(\Omega,\S^n_+)$ \cite[Definition 3.8]{li2020general1}:  
  \begin{equation} \label{eq:distance}
      {\rm WB}^2_{\Lad}(\ms{G}_0,\ms{G}_1) = \inf_{\mu \in \ce([0,1];\ms{G}_0,\ms{G}_1)} \jj_{\Lad,Q}(\mu)\,,  \tag{$\mathcal{P}$}
\end{equation}
for $\ms{G}_0, \ms{G}_1 \in \M(\Omega,\S_+^n)$. It is easy to show that $\ce([0,1];\ms{G}_0,\ms{G}_1)$ is nonempty and $\rm{WB}_{\Lad}(\ms{G}_0,\ms{G}_1)$ is  finite and well-defined. Moreover, there holds \cite[Lemma 3.9]{li2020general1}
\begin{equation}  \label{eq:basic_bound}
{\rm WB}^2_{\Lad}(\ms{G}_0,\ms{G}_1) \le 
{\rm WB}^2_{(0,\Lad_2)}(\ms{G}_0,\ms{G}_1) \le 
2  \big\lVert \Lad_2^{-1} \big\lVert_\ff^2 \int_{\Omega} \big\lVert\sqrt {G_{1,\lad}} - \sqrt{G_{0,\lad}} \big\lVert_\ff^2  \ \rd \lad  \,,
\end{equation}
where $G_{0,\lad}$ and $G_{1,\lad}$ are densities of $\ms{G}_0$ and $\ms{G}_1$ for some reference measure $\lad$. By Fenchel-Rockafellar theorem \cite{bouchitte2020convex}, one can show the existence of minimizer to \eqref{eq:distance} and that ${\rm WB}_{\Lad}(\dd,\dd)$ is weak* lower semicontinuous and satisfies, for $\alpha, \beta > 0$ and $\ms{G}_0,\ms{G}_1,\w{\ms{G}}_0,\w{\ms{G}}_1 \in \M(\Omega,\S^n_+)$,
\begin{align} \label{sublienar}
     {\rm WB}^2_{\Lad}\big(\alpha\ms{G}_0 + \beta \w{\ms{G}}_0, \alpha \ms{G}_1 + \beta \w{\ms{G}}_1\big) \le \alpha {\rm WB}^2_{\Lad}\big(\ms{G}_0, \ms{G}_1\big) + \beta {\rm WB}^2_{\Lad}\big(\w{\ms{G}}_0, \w{\ms{G}}_1\big)\,.
\end{align}
In particular, by \cite[Section 5]{li2020general1}, $(\M(\Omega,\S^n_+), {\rm WB}_{\Lad})$ is a complete geodesic metric space with the constant-speed geodesic connecting measures $\ms{G}_0, \ms{G}_1 \in \M(\Omega,\S^n_+)$ given by the minimizer to \eqref{eq:distance}. 

\begin{remark} \label{rem:connect}
We recall how \eqref{eq:distance} connects with the recently proposed models and refer the readers to \cite[Section 7]{li2020general1} for more details. First, setting parameters $n = d$ and $k = 1$ and matrices $\Lad_i = I$ for $i = 1, 2$, and letting the operator $\ms{D}$ in the continuity equation \eqref{eq:weak_ctneq} be the symmetric gradient: $\na_s(q) = \frac{1}{2}(\na q + (\na q)^{{\rm T}})$ for smooth fields $q \in C_c^\infty(\R^d,\R^d)$, we see that 
\eqref{eq:distance} gives the Kantorovich-Bures distance $d_{KB}$ in 
\cite[Definition 2.1]{brenier2020optimal}: 
\begin{multline} 
{\rm WB}^2_{(I,I)}(\ms{G}_0,\ms{G}_1) = \frac{1}{2}d^2_{KB}(\ms{G}_0,\ms{G}_1) = \inf\Big\{\jj_{\Lad,Q}(\mu)  \,;\ \mu = \ms{(G,q,R)} \in \M(Q,\xx) \notag\\ \text{satisfies} \ 
 \p_t \ms{G} = \{ - \na \ms{q}_t + \ms{R}_t\}^{\sym} \ \text{with}\  \ms{G}_t|_{t = 0} = \ms{G}_0\,,\ \ms{G}_t|_{t = 1} = \ms{G}_1
\Big\}\,,\q \forall \ms{G}_0, \ms{G}_1 \in \M(\Omega,\S^d_+)\,.  \tag{$\mathcal{P}_{\rm WB}$}\label{def:kan_bure}
\end{multline} 
In addition, to recover the Wasserstein-Fisher-Rao metric defined in \cite{chizat2018interpolating,kondratyev2016new,liero2016optimal}, we set $n = 1$, $k = d$, and $\Lad_1 = \sqrt{\alpha} I$, $\Lad_2 = \sqrt{\beta}I$ with $\alpha, \beta > 0$, and let $\ms{D} = \ddiv$, then \eqref{eq:distance} gives the WFR distance: for $\rho_0, \rho_1 \in \M(\Omega,\R_+)$, 
\begin{multline} 
{\rm WFR}^2(\rho_0, \rho_1)  = \inf\Big\{\int_0^1 \int_\Omega \rho^{\dag}\Big(\frac{1}{2\alpha}|q|^2 + \frac{1}{2\beta} r^2\Big)\,\rd x\,\rd t \,;\ 
 \p_t \rho + \ddiv \,q = r \ \text{with}\  \rho_t|_{t = 0} = \rho_0\,,\ \rho_t|_{t = 1} = \rho_1 \Big\}\,.\tag{$\mathcal{P}_{\rm WFR}$} \label{def:wfr_metric}
\end{multline}
With some other operator $\ms{D}$ and parameters, the model \eqref{eq:distance} can also give the matricial interpolation distance in \cite{chen2019interpolation}.  
\end{remark}

The infimum \eqref{eq:distance}
can be equivalently taken over the following set:
\begin{equation*}
    \ce_\infty ([0,1];\ms{G}_0,\ms{G}_1): =  \ce([0,1];\ms{G}_0,\ms{G}_1) \cap \{\mu \in \M(Q,\xx)\,; \jj_{\Lad,Q}(\mu) < +\infty\}\,.
\end{equation*}
By disintegration, a measure $\mu = (\ms{G}, \ms{q},\ms{R})\in  \ce_\infty([0,1];\ms{G}_0,\ms{G}_1)$ can be identified with a curve $\{\mu_t = (\ms{G}_t, \ms{q}_t,\ms{R}_t)\}_{t \in [0,1]}$ in $\M(\Omega, \xx)$  with $\ms{G}_t$ weak* continuous; see \cite[Proposition 3.13]{li2020general1}. Thus, we can write $\jj_{\Lad,Q}(\mu) = \int_0^1\jj_{\Lad,\Omega}(\mu_t)\, \rd t$ for $\mu \in \ce_\infty([0,1];\ms{G}_0,\ms{G}_1)$. We now consider the smooth approximation of $\mu \in \ce_\infty([0,1];\ms{G}_0,\ms{G}_1)$, which will be useful for the numerical analysis of \eqref{eq:distance} in the following sections. For this, we need some assumptions, which are in some sense necessary, considering the generality of our model; see Remark \ref{rem:rea_assum} below.

\begin{enumerate}[label=\textbf{C.\arabic*},ref=C.\arabic*]
\setlength\itemsep{-0.1mm}
    \item \label{1} $\Omega$ is a compact domain that is star-shaped with respect to a set of points with a nonempty interior. 
    \item \label{2} $\ms{D}^*$ admits a decomposition: $\ms{D}^*(\Phi) = (\ms{D}_0^*(\Phi),\ms{D}^*_1(\Phi))$ for $\Phi \in C^\infty(\R^d,\S^n)$, where 
    $\ms{D}_0^*$ and $\ms{D}_1^*$ are the homogeneous components of degree $0$ and $1$ of $\ms{D}^*$ satisfying
    \begin{equation*}
      \ms{D}_0^*: \mc{L}(\S^n, \R^{n \t k_1})\,, \q   \ms{D}_1^*:C^\infty(\R^d, \S^n) \to C^\infty(\R^d, \R^{n \t k_2})\,, \q \text{with}\ k = k_1 + k_2\,.
    \end{equation*} 
    \item \label{3} $\ms{G}_0,\ms{G}_1 \in \M(\Omega,\S_+^n)$ are absolutely continuous with respect to Lebesgue measure with densities $G_0$ and $G_1$. 
\end{enumerate}

\begin{proposition}\label{prop:regular_ep}
Let $\mu = \ms{(G,q,R)} \in \ce_\infty([0,1];\ms{G}_0,\ms{G}_1)$ with $\ms{G}_0, \ms{G}_1 \in \M(\Omega,\S_+^n)$. Suppose the assumptions \eqref{1} and \eqref{2} hold. Then, for small $\ep > 0$,  there exist $\mu^\ep = (\ms{G}^\ep,\ms{q}^\ep,\ms{R}^\ep) \in \M(Q,\xx)$ and $\ms{G}_0^\ep,\, \ms{G}_1^\ep \in \M(\Omega,\S^n_+)$ such that $\mu^\ep \in \ce_\infty([0,1];\ms{G}^\ep_0,\ms{G}^\ep_1)$ with the following properties: 
\begin{enumerate}
\item $\ms{G}_0^\ep$, $\ms{G}_1^\ep$, and $\mu^\ep = (\ms{G}^\ep,\ms{q}^\ep,\ms{R}^\ep)$ have $C^\infty$-smooth densities $G_0^\ep$, $G_1^\ep$ and $(G^\ep,q^\ep, R^\ep)$, respectively, with respect to the Lebesgue measure. Moreover, $G^\ep \succeq cI$ on $Q$ for some $c > 0$ depending on $\ep$.
\item $\mu^\ep$ weak* converges to $\mu$ as $\ep \to 0$, and there holds 
    \begin{align} \label{eq:cost_est}
        \lim_{\ep \to 0} \jj_{\Lad,Q}(\mu^\ep) = \jj_{\Lad,Q}(\mu)\,.
    \end{align}
\item For $i = 0,1$, $\ms{G}_i^\ep$ weak* converges to $\ms{G}_i$. If the assumption \eqref{3} holds, we further have
    \begin{equation} \label{est:conver_ini}
            \lim_{\ep \to 0} {\rm WB}_{\Lad} (\ms{G}_i,\ms{G}^\ep_i) = 0 \,, \q  \lim_{\ep \to 0}\, \norm{\ms{G}_i-\ms{G}^\ep_i}_{\rm TV} = 0\,. 
    \end{equation}
    \end{enumerate}
\end{proposition}

\noindent The proof is technical and given in Appendix \ref{appa} for the sake of readability. 
In what follows, we will refer to the measure $\mu^\ep$ constructed above as the \emph{$\ep$-regularization} of $\mu$. 

\begin{remark} \label{rem:rea_assum}
We briefly discuss the necessity of the assumptions \eqref{1}--\eqref{3}. We first observe that 
\eqref{1} is for defining the scaling function $T^\ep$
\eqref{aux_tsscaling}
that compresses the supports of measures smoothed by the mollifier, so that the regularized measures are still supported on $\Omega$. Thus, \eqref{2} is needed for controlling the effect of such scaling; see Lemma \ref{lemma:timescaling}. Clearly, if $\Omega$ is the whole space $\R^d$ or the torus $\T^d = \R^d/\Z^d$, both assumptions could be removed.

For the necessity of \eqref{3}, we show in the case $\Lad = (0,\Lad_2)$ that \eqref{est:conver_ini} could fail  if \eqref{3} does not hold. Without loss of generality, let us consider $\Omega = \R^d$ and estimate ${\rm WB}_{0,\Lad_2}(\theta_d^\ep*\ms{G}_0,\ms{G}_0)$ with $\ms{G}_0 = \d_0 I$ (here $\d_0$ is the Dirac measure and $\theta_d^\ep$ is  the mollifier;
see Appendix \ref{appa}). Note that we can identify ${\rm WB}_{(0,\Lad_2)}$ with the matricial Hellinger distance \cite[Remark 3.11]{li2020general1}, and it was shown that the Hellinger distance is topologically equivalent to the total variation distance \cite[Theorem 3]{monsaingeon2020schr}. It follows that 
\begin{align*}
   {\rm WB}_{0,\Lad_2}(\theta_d^\ep*\ms{G}_0,\ms{G}_0) &\gtrsim  \norm{\theta_d^\ep*\ms{G}_0 - \ms{G}_0}_{\rm TV} 
   \\
     & 
    \gtrsim \sup \big\{ \l \theta_d^\ep(x) \rd x  - \d_0 , \phi \r_{\R^d}\,; \ \phi \in C_c(\R^d, \R) \ \text{with}\ \norm{\phi}_\infty \le 1 \big\} = 1\,, 
\end{align*}
where, in the last step, we consider a sequence of smooth functions, with $k \to + \infty$, 
\begin{equation*}
    \phi_k = \begin{cases}
     e^{\frac{1}{|kx|^2 - 1}+1}\,,   & |x| \le k^{-1}\,,\\
        0, &  |x| \ge k^{-1}\,.
    \end{cases}
\end{equation*}
\mb{However, it is possible to remove \eqref{3} in some cases (e.g., for WFR distance; see Theorem \ref{lem:reg_points})}. 
\end{remark}

The remaining of this work is devoted to the convergence analysis of discrete OT models associated with \eqref{eq:distance}. We end this section with some preparations. For $(\ms{G},u,W) \in \M(\mathcal{X},\S_+^n) \t C(\mathcal{X}, \R^{n \t k} \t \MM^n)$, we define
\begin{equation} \label{def:jalpha_conj}
    \jj_{\Lad,\mathcal{X}}^*(\ms{G},u,W) := \frac{1}{2} \norm{(u \Lad_1 ,W \Lad_2)}^2_{L^2_{\ms{G}}(\mathcal{X})}\,.
\end{equation} 
It is easy to see that the conjugate function of $\jj^*_{\Lad,\mathcal{X}}(\ms{G},u,W)$ with respect to $(u,W)$ gives $\jj_{\Lad,\mathcal{X}}(\ms{G},\ms{q},\ms{R})$:
\begin{equation} \label{eq:fenconj_cost}
    \jj_{\Lad,\mathcal{X}}(\ms{G},\ms{q},\ms{R}) = \sup_{(u,W) \in L^\infty_{|\ms{(G,q,R)}|}(\mathcal{X}, \R^{n \t k} \t \MM^n) } \l (\ms{q},\ms{R}), (u,W) \r_{\mathcal{X}} - \jj_{\Lad,\mathcal{X}}^*(\ms{G},u,W)\,.
\end{equation}
For simplicity, 
 we focus on the weighted matrices $(\Lad_1,\Lad_2) = (I, I)$ and omit the corresponding subscripts. For instance, we will use 
 $\jj_\Omega$ instead of $\jj_{(I,I),\Omega}$.

\section{Abstract convergence framework}  \label{sec:absconver}

In this section, we propose a convergence framework for an 
 abstract discretization of  \eqref{eq:distance} from the perspective of  Lax equivalence theorem,
 which extends the recent work \cite{lavenant2019unconditional} for the classical $2$-Wasserstein distance. 

As a motivation, we briefly review the Lax equivalence theorem following  \cite{teman2012numerical}. Let $X$ and $Y$ be two Banach spaces and $A \in \mc{L}(X,Y)$ be a linear operator with inverse $A^{-1} \in \mc{L}(X,Y)$, so that the linear equation $A x = y$ 
is well-defined for any given $y \in Y$. \mb{We shall introduce a family of discrete problems to approximate $x = A^{-1}y$.} Let $h > 0$ be the discretization parameter. Suppose that $X_h$ and $Y_h$ are finite-dimensional vector spaces approximating $X$ and $Y$, respectively, and $A_h \in \mc{L}(X_h,Y_h)$ is an invertible linear operator. We define interpolation operators $I^x_h \in \mc{L}(X, X_h)$ and $I^y_h \in \mc{L}(Y, Y_h)$, and an embedding operator $P_h \in \mc{L}(X_h, X)$. Then, the finite-dimensional discrete problem can be formulated as $A_h x_h = I^y_h y$ for $y \in Y$. 

The Lax equivalence theorem means a set of sufficient conditions on operators $(A_h, I_h^x, I_h^y, P_h)$ that guarantees the convergence of $P_h x_h$ to $x$. More precisely, under the following conditions:
\vspace{-1.3mm}
\begin{itemize}
\setlength\itemsep{-0.1mm}
    \item consistency: $\norm{ I^y_h A x - A_h I_h^x x}_{Y_h} \to 0$ as $h \to 0$, \mb{measuring} the commutativity between $I^y_h A$ and $A_h I_h^x$,
    \item stability: $\norm{A_h^{-1}} \le C$ and $\norm{P_h} \le C$ for some $C > 0$ independent of $h$,
    \item approximability: $\norm{x - P_h I_h^x x}_X \to 0$ as $h \to 0$, \mb{measuring} how the discrete space $X_h$ approximates $X$,
\end{itemize}
there holds $\norm{x - P_h x_h}_{X} \to 0$ as $h \to 0$. If one is interested in the convergence in terms of $\norm{I_h^x x - x_h}_{X_h}$, then the assumptions related to $P_h$ (i.e., $\norm{P_h} \le C$ and $\norm{x - P_h I_h^x x}_X \to 0$) can be removed. 

It is worth emphasizing that to define a discrete problem, we only need $(A_h, I_h^y)$, while $(I_h^x, P_h)$ are introduced for quantifying the error. Moreover, the Lax equivalence theorem can be interpreted as follows: given the discrete model $(A_h, I_h^y)$, if there exists $(I_h^x, P_h)$ such that the above assumptions hold, then the convergence of $x_h$ follows. In general, the choice of $(I_h^x, P_h)$ is not unique and is actually a key part of the error analysis. 

We shall adopt this perspective and extend the abstract numerical framework in \cite{lavenant2019unconditional} to our case; \mb{see Table \ref{table:analog} for a comparison of related concepts. In Section \ref{subsec:approx_measure}, we introduce some notions for approximating a general Radon measure. Then, Sections \ref{subsec:abstract_conver} and \ref{subsec:convergence} are devoted to an abstract discretization of our transport distance \eqref{eq:distance} and its convergence.} 
The main result of this section is Theorem \ref{thm:main_conver}, where the weak convergence of an abstract discrete matrix-valued OT model is proved.

\begin{table}[!htbp] 
     \centering 
\mb{   \begin{tabular}{|c|c|c|}
\hline 
 & Lax equivalence for $Ax = y$ & Convergence for discrete OT models \tabularnewline 
\hline 
Discrete problems  & $A_h x_h = I^y_h y$ with stable $A_h^{-1}$  & Definition \ref{def:disc_transp} with stability \eqref{eq:bdd_adjoint} \tabularnewline
\hline 
Interpolation & Consistent $I_h^x$ and $I_h^y$  & Consistent $\ss_{\ol{{\rm X}}_\si}$ and $\ss_{{\rm Y}_\si}$ (Definition \ref{def:cons_sampl})
\tabularnewline 
\hline 
Embedding & Stable $P_h$ with approximability & Stable and consistent (Definition \ref{def:discre_appro} and \ref{def:cons_recons})
\tabularnewline 
\hline 
\end{tabular}}
\caption{\mb{Analogy between Lax equivalence theorem and our convergence analysis}} \label{table:analog} 
\end{table}

\subsection{Approximation of Radon measures} \label{subsec:approx_measure}

\mb{To facilitate the discussion on the discretization of the problem \eqref{eq:distance} and its convergence, we first briefly discuss approximating a Radon measure.} Let $\mathcal{X}$ be a compact set in Euclidean space and $h$ be a discretization parameter in a subset $\hh$ of $\R_+^k$ with $0$ being its accumulation point. \mb{For our problem, $\mc{X}$ is either the spatial domain $\Omega$ or the time-space one $Q =[0,1] \t \Omega$, and  $h$ is for the spatial mesh size $\si$ or the time-space one $(\tau,\si)$; see Section \ref{subsec:abstract_conver}.} 


\begin{definition}[Approximation of measure] \label{def:discre_appro}
    We say that $\left\{({\rm M}_h, \rr_{{\rm M}_h})\right\}_{h \in \hh}$ is \emph{a family of discrete approximations} of the measure space $\M(\mathcal{X},\R^n)$ if for any $h \in \hh$, 
    \begin{enumerate}
        \item ${\rm M}_h$ is a finite-dimensional vector space equipped with a norm $\norm{\dd}_{1,\ms{M}_h}$ and an inner product $\l \dd,\dd\r_{\ms{M}_h}$.  The dual norm on ${\rm M}_h$ is defined by 
\begin{equation} \label{def:dualnorm_dis}
    \norm{f^h}_{*,\ms{M}_h} := \max\big\{\left\l f^h, g^h \right\r_{\ms{M}_h}; \ \norm{g^h}_{1,\ms{M}_h} \le 1 \big\}\,,\q  f^h \in {\rm M}_h\,.
\end{equation}    
        \item  $\rr_{{\rm M}_h}: ({\rm M}_h,\norm{\dd}_{1,\ms{M}_h}) \to \M(\mathcal{X},\R^n)$, called the reconstruction operator, is a linear operator with the operator norm
        uniformly bounded with respect to $h$: for some $C > 0$ independent of $h$,
        \begin{align} \label{def:bound_of_r}
            \norm{\rr_{{\rm M}_h}f^h}_{\rm TV} \le C \norm{f^h}_{1,\ms{M}_h}\,, \q  \forall f^h \in {\rm M}_h\,.
        \end{align}
    \end{enumerate}
\end{definition}

\begin{remark} \label{rem:app_ran}
The norm $\norm{\dd}_{\ms{M}_h}$ induced by the inner product $\l \dd,\dd\r_{\ms{M}_h}$ is generally different from \emph{discrete $L^1$ and $L^\infty$-norms} $\norm{\dd}_{1,\ms{M}_h}$ and $\norm{\dd}_{*,\ms{M}_h}$; see \eqref{defnorminner} and \eqref{def:dualnorm_X} below for an example. \mb{We also remark that in our analysis, the reconstruction $\mc{R}_{\ms{M}_h}$ with the stability  \eqref{def:bound_of_r} plays the role of the embedding $P_h$ in the Lax equivalence theorem.} 
\end{remark}

Recall that our main aim is the weak convergence of discrete solutions. A natural question is when two families of reconstructions $\rr_{{\rm M}_h}f^h$ and $\w{\rr}_{{\rm M}_h}f^h$ would give the same limit in the weak* topology as $h \to 0$. \mb{To answer it,} we introduce the following concepts. We denote by $\rr^*_{{\rm M}_h}: C(\mathcal{X},\R^n) \to ({\rm M}_h,\norm{\dd}_{*,\ms{M}_h})$ the adjoint of $\rr_{{\rm M}_h}$: 
\begin{align} \label{def:adj_recons}
    \big\l f^h, \rr^*_{{\rm M}_h}\vp \big\r_{\ms{M}_h}  = \big\l \rr_{{\rm M}_h}f^h, \vp \big\r_{\mathcal{X}}\,,\q   \forall f^h \in {\rm M}_h\,,\  \vp \in C(\mathcal{X},\R^n) \,.
 \end{align}

\begin{definition} \label{def:space_reconst}
Two families of reconstruction operators $\{\rr_{{\rm M}_h}\}_{h \in \hh}$ and $\{\w{\rr}_{{\rm M}_h}\}_{h \in \hh}$ are \emph{equivalent}, denoted by $\rr_{{\rm M}_h} \sim \w{\rr}_{{\rm M}_h}$, 
if there exists $\{\ep_h\}_{h \in \hh}$,  with $\ep_h > 0$ and $\ep_h \to 0$ as $h \to 0$ such that
 \begin{align} \label{def:equirecons}
    &  \norm{(\rr^*_{{\rm M}_h} - \w{\rr}^*_{{\rm M}_h})\vp}_{*,\ms{M}_h} \le \ep_{h} \norm{\vp}_{1,\infty}\,,
     \q \forall \vp \in C^1(\mathcal{X},\R^n)\,.
  \end{align}
A family of reconstructions $\{\h{\rr}_{{\rm M}_h}\}_{h \in \hh}$ is \emph{quasi-isometric} if for some  constant $C > 0$ independent of $h$,
  \begin{align} \label{def:quasirecons}
      C^{-1} \norm{f^h}_{1,\ms{M}_h} \le \norm{\h{\rr}_{{\rm M}_h}f^h}_{\rm TV} \le C \norm{f^h}_{1,\ms{M}_h}\,, \q \forall f^h \in {\rm M}_h\,.
  \end{align}
\end{definition}

One can check that the condition \eqref{def:equirecons} indeed defines an equivalence relation (i.e., a binary relation with reflexivity, symmetry, and transitivity) on the 
set of reconstructions. We remark that this equivalence relation does not preserve the quasi-isometric property. That is, a reconstruction 
$\rr_{{\rm M}_h}$ equivalent to a quasi-isometric reconstruction may not be quasi-isometric. With the above notions, we have the following useful lemma.
 
\begin{lemma}\label{lemma:unifrombdd}
Let $\{({\rm M}_h, \rr_{{\rm M}_h})\}_{h \in \hh}$ be a family of discrete approximations of $\M(\mathcal{X},\R^n)$ and $\{f^h \}_{h \in \hh}$ be a bounded sequence in $({\rm M}_h, \norm{\dd}_{1,\ms{M}_h})$. Then, there exists a sequence $\{h_n\}_{n \in \NN} \subset \hh$ with $h_n \to 0$ and a measure $\mu \in \M(\mathcal{X},\R^n)$ such that for any equivalent reconstruction $\w{\rr}_{{\rm M}_{h_n}}\sim\rr_{{\rm M}_{h_n}}$, $\w{\rr}_{{\rm M}_{h_n}}f^{h_n}$ weak* converges to $\mu$ as $n \to \infty$. 
\end{lemma}
\begin{proof}
First, note that $\{\rr_{{\rm M}_h}f^h\}_{h \in \hh}$ is bounded in $\M(\mathcal{X},\R^n)$ by \eqref{def:bound_of_r}. There exists a subsequence of $\rr_{{\rm M}_h}f^h$, still indexed by $h$, weak* converging to some $\mu \in \M(\mathcal{X},\R^n)$.  We next show that any sequence $\w{\rr}_{{\rm M}_h}f^h$ with $\w{\rr}_{{\rm M}_h} \sim \rr_{{\rm M}_h}$ weak* converges to the same limit. For this, 
by \eqref{def:equirecons}, we have
\begin{align*}
    \big|\big\l \big(\w{\rr}_{{\rm M}_{h}} - \rr_{{\rm M}_{h}}\big)f^{h},  \vp \big\r_{\mathcal{X}} \big| \le \ep_{h} \norm{\vp}_{1,\infty} \norm{f^{h}}_{1,\ms{M}_{h}}\,,  \q \forall  \vp \in C^1(\mathcal{X},\R^n)\,, 
\end{align*}
which gives
$\l \w{\rr}_{{\rm M}_{h}}f^{h}, \vp \r_{\mathcal{X}} \to \l \mu, \vp\r_{\mc{X}}$ as $h \to 0$, by the boundedness of $\norm{f^{h}}_{1,\ms{M}_h}$ and the weak* convergence of $\rr_{{\rm M}_h}f^h$. The proof is completed by the density of $C^1(\mathcal{X},\R^n)$ in $C(\mathcal{X},\R^n)$ and the boundedness of $\w{\rr}_{{\rm M}_h}f^h$.  
\end{proof}

\subsection{Discrete matrix-valued OT problem} \label{subsec:abstract_conver}

We next discretize the matrix-valued transport problem \eqref{eq:distance}, primarily focusing on the spatial discretization, since, as we shall see, the time
discretization is explicit and straightforward. 
For this, we first introduce the following basic modules adapted from \cite[Definition 2.5]{lavenant2019unconditional}. Let $\si \in \Sigma \subset (0,+\infty)$ be the spatial discretization parameter.


\begin{definition}[Spatial discrete modules] \label{def:build_block}
    We define the following abstract objects for the problem \eqref{eq:distance}.
\begin{enumerate}[\textbullet]
    \item  $\big\{\big(\overline{{\rm X}}_\si, \h{\rr}_{\ol{{\rm X}}_\si}\big)\big\}_{\si \in \Sigma}$ and $\big\{\big({\rm Y}_\si, \h{\rr}_{{\rm Y}_\si}\big)\big\}_{\si \in \Sigma}$ are discrete approximations of $\M(\Omega, \MM^n)$ and $\M(\Omega, \R^{n \t k})$, respectively, where $\h{\rr}_{\ol{{\rm X}}_\si}$ and $\h{\rr}_{{\rm Y}_\si}$ are quasi-isometric.  Moreover, let ${\rm X}_\si$ be a subspace of $\ol{{\rm X}}_\si$ and contain a linear cone ${\rm X}_{\si,+}$, and ${\rm X}_\si^\perp$ denote the orthogonal complement of ${\rm X}_\si$ in $\big(\ol{{\rm X}}_\si, \l \dd, \dd \r_{\ol{\ms{X}}_\si}\big)$. For clarity, we denote by $\norm{\dd}_{1,\ms{X}_\si}$ and $\l\dd,\dd \r_{\ms{X}_\si}$
    the norm and inner product on ${\rm X}_\si$ induced from those of $\ol{{\rm X}}_\si$, respectively.  
  
\item \mb{The set of admissible spatial reconstructions on $\ol{{\rm X}}_\si$, denoted by $\big[\h{\rr}_{\ol{{\rm X}}_\si}\big]$, consists of those $\rr_{\ol{{\rm X}}_\si}$ that is equivalent to $\h{\rr}_{\ol{{\rm X}}_\si}$ and maps ${\rm X}_\si$, ${\rm X}_\si^\perp$, and ${\rm X}_{\si,+}$ to $\M(\Omega,\S^n)$, $\M(\Omega, \AA^n)$, and $\M(\Omega,\S^n_+)$, respectively. We assume that $\h{\rr}_{\ol{{\rm X}}_\si}$ itself satisfies the aforementioned mapping property.  The set of admissible reconstructions $\big[\h{\rr}_{{\rm Y}_\si}\big]$ on ${\rm Y}_\si$ is defined by all $\rr_{{\rm Y}_\si}$ equivalent to $\h{\rr}_{{\rm Y}_\si}$, i.e., 
\begin{equation} \label{adeq2}
\Big[\h{\rr}_{{\rm Y}_\si}\Big]: = \Big\{\rr_{{\rm Y}_\si}\,;\ \rr_{{\rm Y}_\si}\sim \h{\rr}_{{\rm Y}_\si} \Big\}\,,   
\end{equation}
while the set of admissible reconstructions on ${\rm X}_\si$ is simply defined as the restriction of $\big[\h{\rr}_{\ol{{\rm X}}_\si}\big]$ on ${\rm X}_\si$:} 
\begin{equation} \label{adeq3}
    \Big[\h{\rr}_{{\rm X}_\si}\Big]: = \Big\{\rr_{{\rm X}_\si}\,;\ \rr_{{\rm X}_\si} : =   \rr_{\ol{{\rm X}}_\si}|_{{\rm X}_\si}, \rr_{\ol{{\rm X}}_\si}\in \big[\h{\rr}_{\ol{{\rm X}}_\si}\big] \Big\}\,.
\end{equation}
        \item 
         $\jj_\sigma: ({\rm X}_{\sigma}, {\rm Y}_\sigma, \ol{{\rm X}}_\sigma) \to [0,+\infty]$ is a proper, lower semicontinuous, convex function with positive homogeneity of degree $-1$ in its first variable and homogeneity of degree $2$ in its second and third variables.  
        \item The discrete derivation $\ms{D}_\si: {\rm Y}_\si \to {\rm X}_\si$ is a linear operator with the adjoint $\ms{D}^*_\si: {\rm X}_\si \to {\rm Y}_\si$ defined by 
        \begin{align} \label{def:adj_disderiv}
           \l y^\si, \sd^*\, x^\si \r_{\ms{Y}_{\si}} := - \l \sd\, y^\si, x^\si \r_{\ms{X}_\si}\,,  \q \forall x^\si \in {\rm X}_\si,\, y^\si \in {\rm Y}_\si\,.
        \end{align}
    \end{enumerate}        
\end{definition}

\begin{remark}
We observe that $\h{\rr}_{{\rm X}_\si} = \h{\rr}_{\ol{{\rm X}}_\si}|_{{\rm X}_\si}$ is quasi-isometric and the elements in the set  $\big[\h{\rr}_{{\rm X}_\si}\big]$ \eqref{adeq3} are equivalent to $\h{\rr}_{{\rm X}_\si}$, in the sense of Definition \ref{def:space_reconst}.
We emphasize that the only assumption for the operator $\sd$ is the linearity.
\end{remark}

We introduce the conjugate of $\jj_\si(G^\si,q^\si,R^\si)$ with respect to $(q^\si,R^\si)$ by 
\begin{equation} \label{def:disjalpha_conj}
   \jj^*_\si (G^\si, u^\si,W^\si) := \sup_{(u^\si,W^\si) \in {\rm Y}_\si \t \ol{{\rm X}}_\si} \l q^\si, u^\si \r_{\ms{Y}_\si} + \l R^\si, W^\si \r_{\ol{\ms{X}}_\si} - \jj_\si(G^\si,q^\si,R^\si)\,.
\end{equation} 
Note that $\jj_\si$ is homogeneous of degree 2 in $(q^\si,R^\si)$, so is $\jj^*_\si$ in $(u^\si,W^\si)$. It follows from \eqref{def:disjalpha_conj} that
\begin{align} \label{eq: ana_Cauchy}
    \l q^\si, u^\si \r_{\ms{Y}_\si} + \l R^\si, W^\si \r_{\ol{\ms{X}}_\si} \le  \gamma^{-2} \jj_\si(G^\si,q^\si,R^\si) + \gamma^2 \jj^*_\si (G^\si, u^\si, W^\si)\,, \q \forall \gamma > 0\,. 
\end{align}
We minimize the right-hand side of \eqref{eq: ana_Cauchy}
in $\gamma$ and have 
\begin{align}\label{eq:disana_Cauchy_2}
    \l q^\si, u^\si \r_{\ms{Y}_\si} + \l R^\si, W^\si \r_{\ol{\ms{X}}_\si} \le 2 \sqrt{\jj_\si(G^\si,q^\si,R^\si)  \jj^*_\si (G^\si, u^\si, W^\si)}\,.
\end{align} 
Let $\P_{{\rm X}_\si}$ be the orthogonal projection from $\big(\ol{{\rm X}}_\si, \l\dd,\dd\r_{\ol{\ms{X}}_\si}\big)$ to its subspace ${\rm X}_\si$. For an admissible reconstruction $\rr_{\ol{{\rm X}}_\si}$ on $\ol{\ms{X}}_\si$, thanks to $\rr_{\ol{\rm X}_\si}: {\rm X}_\si^\perp  \to \M(\Omega, \AA^n)$, we can derive 
\begin{equation} \label{eq:invari_space}
\big\l \rr_{\ol{\rm X}_\si}R^\si,  \Phi \big\r_{\Omega} = \big\l \rr_{\ol{\rm X}_\si} (\P_{{\rm X}_\si} R^\si),  \Phi \big\r_{\Omega} = \big\l  \rr_{{\rm X}_\si}(\P_{{\rm X}_\si} R^\si), \Phi \big\r_{\Omega}\,, \q \forall R^\si \in \ol{\ms{X}}_\si\,,\ \Phi \in C(\Omega,\S_n)\,.
\end{equation}
With the help of the above modules, the semi-discrete model of \eqref{eq:distance} readily follows:
\begin{align*}
    &\min \q \int_0^1 \jj_\sigma(\mu^\si_t)\, \rd t \q \text{over} \q \{\mu^\si_t = (G^\si_t, q^\si_t, R^\si_t)\}_{t \in [0,1]} \subset {\rm X}_{\sigma,+} \t {\rm Y}_\sigma \t \ol{{\rm X}}_\sigma\,, \\
& \text{subject to}\q \p_t G^\si_t + \sd q^\si_t = \P_{{\rm X}_\si} R_t^\si  \q  \text{with}\q  G^\si_t|_{t = 0} = \ms{G}_0^\si\,, \  G^\si_t|_{t = 1} = \ms{G}_1^\si\,,
\end{align*}
where $\ms{G}_i^{\si} \in {\rm X}_{\si,+}$, $i = 0, 1$, is a given approximation of $\ms{G}_i\in \M(\Omega,\S^n_+)$.

We proceed to consider the time discretization. We adopt the staggered grid for the time interval $[0,1]$. Let $N \ge 1$ be the number of time steps and $\tau: = 1/N$ be the temporal step size. We approximate the values of $G_t^\si$ at the time levels $t = i \tau $ for $i = 0,\ldots, N$ and the values of $(q_t^\si,R_t^\si)$ at $t = (i-1/2) \tau$ for $i = 1,\ldots, N$. We denote by 
$$(\ts) \in \Sigma_{\NN}: = \{1/2,\ldots,1/N,\ldots\} \t \Sigma$$ the time-space discretization parameter, and define the product space: for each $(\ts) = (1/N,\si)$ in $\Sigma_{\NN}$,
\begin{equation} \label{eq_product}
\M_{\ts} := {\rm X}_{\ts} \t {\rm Y}_{\ts} \t \ol{{\rm X}}_{\ts} : = ({\rm X}_\sigma)^{N+1} \t ({\rm Y}_\sigma)^N \t (\ol{{\rm X}}_\sigma)^N\,.
\end{equation}
Then, the fully discretized model can be formulated as follows.
 
\begin{definition}[Discrete OT model] \label{def:disc_transp}
Let $({\rm X}_\sigma, {\rm Y}_\sigma, \ol{{\rm X}}_\sigma)$, $\jj_\si$, and $\sd$ be given by Definition \ref{def:build_block}, and $\ms{G}^\si_0, \ms{G}^\si_1 \in  {\rm X}_{\sigma,+}$ be the \emph{discrete initial} and \emph{final distributions} satisfying, \mb{for $i = 0,1$},
\begin{align} \label{eq:convdis_inifal}
    \lim_{\si \to 0}\big\l \h{\rr}_{{\rm X}_\sigma}\ms{G}^\si_i,  \Phi \big\r_{\Omega}  = \l \mathsf{G}_i,  \Phi \r_{\Omega}\,,\q \forall \Phi \in C(\Omega,\S^n)\,.
 \end{align}
We write $\mu^{\ts}$ for $(G^{\ts},q^{\ts},R^{\ts})\in \M_{\ts}$ and define  
the \emph{discrete action functional} $\jj_\ts$ by
    \begin{equation} \label{eq:discost}
        \jj_{\tau,\sigma}(\mu^{\tau,\sigma}) := \tau \sum^N_{k  = 1}  \jj_\sigma \Big(\frac{G^{\ts}_{k-1} + G^{\ts}_k}{2},q^{\ts}_k,R^{\ts}_k \Big)\,.
 \end{equation}
We denote by $\ce_\ts([0,1]; \ms{G}_0^\si, \ms{G}_1^\si)$ the set of $\mu^\ts \in \M_\ts$ satisfying the \emph{discrete continuity equation}: 
\begin{equation} \label{eq:discontieq}
    \tau^{-1}(G^{\ts}_{k} - G^{\ts}_{k-1}) = - \mathsf{D}_\sigma q^{\ts}_k + \P_{{\rm X}_\si} R^{\ts}_k\,, \q  k = 1,2,\ldots, N\,,
\end{equation}
with $G^\ts_{0} = \ms{G}^\si_0$, $G^\ts_N = \ms{G}^\si_1$, and $\{G_i^\ts\}_{i \in \{1,\ldots, N\}} \subset {\rm X}_{\si,+}$. Then, the discrete matrix-valued OT model is defined by
\begin{equation} \tag{$\mathcal{P}_{\ts}$} \label{prob:discre_ot}
    \inf\left\{\jj_{\ts}(\mu^\ts)\,; \ \mu^\ts \in \ce_\ts([0,1]; \ms{G}_0^\si,\ms{G}_1^\si) \right\}\,.
\end{equation}
\end{definition} 

\begin{remark}[Existence of minimizer]\label{rem:existence}
By the direct method in the calculus of variations, we have that \emph{if there exists $\mu^\ts \in \ce_\ts([0,1]; \ms{G}_0^\si, \ms{G}_1^\si)$ such that $\jj_\ts(\mu^\ts) < + \infty$ and $G_i^\ts \in {\rm X}_{\si,+}$, then the discrete OT problem \eqref{prob:discre_ot} admits a minimizer.} Indeed,  it is clear from \eqref{eq:discost} that $\jj_\ts$ is a proper lower semicontinuous  convex functional. By estimates similar to the ones in Theorem \ref{thm:lower_abs} for the boundedness of $\{\mu^\ts\}_{(\ts)}$, we can show the boundedness of a minimizing sequence $\{\mu^\ts_n\}_n$ for \eqref{prob:discre_ot}. Then, thanks to the finite dimensionality of \eqref{prob:discre_ot}, up to a subsequence, $\{\mu^\ts_n\}_n$ converges to some vector $\mu^{\tau,\si}$ satisfying the discrete continuity equation \eqref{eq:discontieq}. The proof is completed by the lower semicontinuity of $\jj_\ts$.  
\end{remark}

Before investigating the convergence of \eqref{prob:discre_ot} to \eqref{eq:distance}, we introduce the time-space reconstruction operators based on the spatial ones in Definition \ref{def:build_block}.

\begin{definition}[Time-space reconstruction]
Let $\rr_{{\rm X}_\si}$, $\rr_{\ol{\rm X}_\si}$, and $\rr_{{\rm Y}_\si}$ be  admissible spatial reconstructions \mb{given in Definition \ref{def:build_block}} and $\mc{M}_\ts$ be the product space \eqref{eq_product}.
We define the piecewise linear (in time) reconstruction $\rr^{l}_{{\rm X}_\ts}: {\rm X}_\ts \to \mc{M}(Q, \S^n)$ by the weak form: for $G^\ts \in {\rm X}_\ts$ and any $\Phi \in C(Q,\S^n)$,
\begin{align}\label{eq:tsrxl}  
  \big\l \rr^{l}_{{\rm X}_\ts}(G^\ts) , \Phi \big\r_{Q} := \sum_{k = 1}^N \int_{(k-1)\tau}^{k \tau}   \Big \l \frac{k \tau - t}{\tau} \rr_{X_\sigma} G^\ts_{k-1} + \frac{t - (k-1)\tau}{\tau} \rr_{{\rm X}_\sigma} G^\ts_{k}, \Phi_t \Big\r_\Omega\, \rd t\,,
\end{align}
and the piecewise constant reconstruction $ \rr^{a}_{{\rm X}_\ts}$ on ${\rm X}_\ts$ by 
\begin{align} \label{eq:tsrxa} 
   \big\l \rr^{a}_{{\rm X}_\ts}(G^\ts), \Phi \big\r_{Q} := \sum_{k = 1}^N  \int_{(k-1)\tau}^{k \tau} \Big\l \frac{1}{2} \rr_{{\rm X}_\sigma}G^\ts_{k-1} + \frac{1}{2}  \rr_{{\rm X}_\sigma}G^\ts_{k}, \Phi_t \Big \r_{\Omega}\, \rd t\,.
\end{align}
We also define the piecewise constant reconstructions $\rr_{{\rm Y}_\ts}$ and $\rr_{\ol{{\rm X}}_\ts}$ on ${\rm Y}_{\ts}$ and $\ol{{\rm X}}_{\ts}$, respectively, by 
\begin{align}
  \big\l \rr_{{\rm Y}_\ts}(q^\ts), \Phi  \big\r_{Q} := \sum_{k = 1}^N \int_{(k-1)\tau}^{k \tau}   \big\l \rr_{{\rm Y}_\si} q^\ts_k, \Phi_t  \big\r_{\Omega}\, \rd t\,, \q \forall \Phi \in C(Q,\R^{n \t k})  \,,  \label{eq:tsry}
\end{align} 
and 
\begin{align}
   \big\l \rr_{\ol{{\rm X}}_\ts}(R^\ts), \Phi  \big\r_{Q} := \sum_{k = 1}^N \int_{(k-1)\tau}^{k \tau}  \big\l \rr_{\ol{{\rm X}}_\si} R^\ts_k, \Phi_t  \big\r_{\Omega}\, \rd t\,, \q 
   \forall \Phi \in C(Q,\MM^n)   \,.  \label{eq:tsrz}
\end{align} 
\end{definition}

We next endow the finite-dimensional spaces ${\rm X}_\ts$, ${\rm Y}_\ts$, and $\ol{{\rm X}}_\ts$ with norms and inner products, so that they, together with the time-space reconstructions defined above, give the families of discrete approximations of $\M(Q,\S^n)$, $\M(Q,\R^{n \t k})$, and $\M(Q,\MM^n)$, respectively, in the sense of Definition \ref{def:discre_appro}. 

\begin{definition} \label{def:tsnorm}
We define the norm on ${\rm X}_\ts$ by $\norm{G^\ts}_{1,\ms{X}_{\ts}} := \sum_{k = 0}^N \tau \norm{G_k^\ts}_{1,\ms{X}_\si}$ for $G^\ts \in {\rm X}_{\ts}$ and the inner product by $\l G^\ts, Q^\ts \r_{\ms{X}_{\ts}} := \sum_{k = 0}^N \tau \l G_k^\ts, Q_k^\ts\r_{\ms{X}_\si}$ for $G^\ts, Q^\ts \in {\rm X}_{\ts}$. Similarly, we equip the space ${\rm Y}_{\ts}$ 
with the norm defined by $\norm{q^\ts}_{1,\ms{Y}_\ts} := \tau \sum_{i = 1}^N \norm{q^\ts_i}_{\ms{Y}_\si}$ for $q^\ts \in {\rm Y}_\ts$ and the inner product defined by $\l p^\ts, q^\ts\r_{\ms{Y}_\ts} := \tau \sum_{i 
= 1}^N  \l p_i^\ts, q^\ts_i\r_{\ms{Y}_\si}$ for $p^\ts, q^\ts \in {\rm Y}_\ts$. The norm $\norm{\dd}_{1,\ol{\ms{X}}_\ts}$ and the inner product $\l\dd,\dd\r_{\ol{\ms{X}}_\ts}$ on $\ol{{\rm X}}_{\ts}$ are defined in the same way as that of ${\rm Y}_{\ts}$. 
\end{definition}

By the above definitions, for $G^\ts \in {\rm X}_{\ts}$, there holds
\begin{equation*} 
\norm{\rr^{l/a}_{{\rm X}_\ts}(G^\ts)}_{\rm TV}   
    \le \tau \sum_{i = 0}^N \norm{\rr_{{\rm X}_\sigma}(G^\ts_i)}_{\rm TV}\,,
\end{equation*}
implying that $\{({\rm X}_{\ts}, \rr^{l/a}_{{\rm X}_{\ts}})\}_{(\ts) \in \Sigma_{\NN}}$ is a family of discrete approximations of $\M(Q,\S^n)$.  Here and in what follows, $\rr^{l/a}_{{\rm X}_\ts}$ denotes two cases: $\rr^{l}_{{\rm X}_\ts}$ and $\rr^{a}_{{\rm X}_\ts}$. Similarly, $\{({\rm Y}_\ts,\rr_{{\rm Y}_\ts})\}_{(\ts) \in \Sigma_\NN}$ and $\{(\ol{{\rm X}}_\ts,\rr_{\ol{{\rm X}}_\ts})\}_{(\ts) \in \Sigma_\NN}$ are discrete approximations of $\M(Q,\R^{n \t k})$ and $\M(Q,\MM^n)$, respectively, with
    \begin{alignat}{2} \label{eq:norm_rots_1}
    \norm{\rr_{{\rm Y}_\ts}(q^{\ts})}_{\rm TV}  & =  \tau \sum_{i = 1}^N \norm{\rr_{{\rm Y}_\sigma}(q^\ts_i)}_{\rm TV}\,, \q \norm{\rr_{\ol{{\rm X}}_\ts}(R^{\ts})}_{\rm TV} & =  \tau \sum_{i = 1}^N \norm{\rr_{\ol{{\rm X}}_\sigma}(R^\ts_i)}_{\rm TV}\,.
    \end{alignat}
For ease of exposition, we define the norm on the product space $\M_{\ts}$: for $\mu^\ts = (G^\ts,q^\ts,R^\ts) \in \M_\ts$,
\begin{equation*}
    \norm{\mu^\ts}_{1,\M_\ts} := \norm{G^\ts}_{1, \ms{X}_\ts} + \norm{q^\ts}_{1, \ms{Y}_\ts} + \norm{R^\ts}_{1, \ol{\ms{X}}_\ts}\,,
\end{equation*}
and the inner product $\l \dd, \dd \r_{\M_\ts}$ can be defined similarly. A direct calculation gives the dual norm $\norm{\dd}_{*, \M_\ts}$ \eqref{def:dualnorm_dis}:  
\begin{align*}
    \norm{\mu^\ts}_{*,\M_\ts} = \max \big\{\norm{G^\ts}_{*,{\rm X}_\ts}, \norm{q^\ts}_{*,{\rm Y}_\ts}, \norm{R^\ts}_{*,\ol{{\rm X}}_\ts}\big\}\,.
\end{align*}
Moreover, by the time-space reconstructions defined above, we introduce 
\begin{equation} \label{def:rotsall}
    \rr^{l/a}_{\ts}(G^\ts,q^\ts,R^\ts) := \big(\rr^{l/a}_{{\rm X}_\ts}(G^\ts), \rr_{{\rm Y}_{\ts}}(q^\ts), \rr_{\ol{\rm{X}}_\ts}(R^\ts)\big): \M_\ts \to \M(Q,\xx)\,,
\end{equation}
and the set of admissible reconstructions on $\M_\ts$: for $(\ts) \in \Sigma_\NN$,
\begin{multline} \label{def:addtsr}
     \rr^{\rm ad}_\ts =  \big\{ \rr^{l/a}_{\ts}\ \text{in \eqref{def:rotsall}}\,;  \ 
     \rr^{l/a}_{{\rm X}_\ts}, \rr_{{\rm Y}_{\ts}},\,\text{and}\ \rr_{\ol{{\rm X}}_{\ts}}\ \text{are defined in \eqref{eq:tsrxl}--\eqref{eq:tsrz} based on} \\ \text{
    admissible spatial reconstructions}\big\}\,.
\end{multline}
In what follows, we denote by $\widehat{\rr}^{l/a}_{{\rm X}_\ts}$, $\h{\rr}_{{\rm Y}_\ts}$, and $\h{\rr}_{\ol{{\rm X}}_\ts}$ the time-space reconstructions \eqref{eq:tsrxl}--\eqref{eq:tsrz} \mb{based on} quasi-isometric spatial ones. Then, we see from \eqref{eq:norm_rots_1} that the operators $\h{\rr}_{{\rm Y}_\ts}$ and $\h{\rr}_{\ol{{\rm X}}_\ts}$ are quasi-isometric as well, while $\widehat{\rr}^{l/a}_{{\rm X}_\ts}$ may not be. One can also check that $\rr^l_{{\rm X}_\ts}$ and $\rr^a_{{\rm X}_\ts}$ are equivalent. 


\subsection{Convergence analysis}\label{subsec:convergence}

The Lax equivalence principle reviewed above suggests that if the discrete model \eqref{prob:discre_ot} is stable and consistent in some sense, then its convergence could be expected. 
This subsection is devoted to the main result Theorem \ref{thm:main_conver}, whose proof follows from asymptotic lower and upper bounds in Theorems  \ref{thm:lower_abs} and \ref{thm:upper_abs}.
We start with the lower bound, which needs the following consistency between $\ms{D}^*$ and $\sd^*$ and between $\jj_\si^*$ and $\jj_\Omega^*$.

\begin{definition}[Consistent reconstructions]\label{def:cons_recons}
Given discrete objects in Definition \ref{def:build_block},
we say that the admissible reconstructions $\big[ \h{\rr}_{\ol{\rm X}_\si} \big]$ and $\big[\h{\rr}_{{\rm Y}_\si}\big]$ are consistent if  
    \begin{enumerate}[\textbullet]
        \item (consistency of  $\ms{D}^*$ and $\sd^*$). There exist
    $\rr_{{\rm Y}_\si} \in [\h{\rr}_{{\rm Y}_\si}]$  
        and $\{\ep_\si\}_{\si \in \Sigma} \subset \R_+$ with $\ep_\si \to 0$ as $\si \to 0$ such that 
        \begin{align} \label{eq:conserr_adjD}
            \big|\big\l  \rr_{{\rm Y}_\si} (q^\si),  \mathsf{D}^* \phi \big\r_{\Omega} + \big \l \h{\rr}_{{\rm X}_\sigma}  (\sd q^\si), \phi \big\r_{\Omega} \big| \le \ep_\si \norm{\phi}_{2,\infty} \norm{q^\sigma}_{1,\ms{Y}_\si}\,, \q \forall q^\si \in {\rm Y}_\si\,, \, \phi \in  C^2(\Omega, \S^n)\,.
        \end{align}
        \item (consistency of $\jj_\Omega^*$ and $\jj_\si^*$). 
        There exist
         $\w{\rr}_{{\rm X}_\si} \in [\h{\rr}_{{\rm X}_\si}]$, $\rr_{{\rm Y}_\si} \in [\h{\rr}_{{\rm Y}_\si}]$, and $\rr_{\ol{{\rm X}}_\si} \in [\h{\rr}_{\ol{{\rm X}}_\si}]$   
        such that
        \begin{align} \label{eq:bdd_adjoint}
            \jj^*_\si \big(G^\si,\h{\rr}^*_{{\rm Y}_\si} u ,\h{\rr}^*_{\ol{\rm X}_\si} W\big) \lesssim (\norm{u}^2_{\infty} + \norm{W}^2_{\infty})\norm{G^\si}_{1,\ms{X}_\si}\,, \q \forall G^\si \in {\rm X}_{\si}\,,\, (u,W) \in C(\Omega, \R^{n \t k} \t \mathbb{M}^n)\,,
        \end{align}
        and 
       for any $G^\si \in {\rm X}_{\si}$ and $(u,W)$ in a bounded set $E$ of
         $C^1(\Omega, \R^{n \t k} \t \mathbb{M}^n)$, 
        \begin{align} \label{est:conjaction}
            \jj^*_\si \big(G^\si,\rr^*_{{\rm Y}_\si} u, \rr^*_{\ol{\rm X}_\si} W  \big) \le \jj_{\Omega}^* \big(\w{\rr}_{{\rm X}_\si} G^\si, u, W \big) + \ep_\si C_E \norm{G^\si}_{1,\ms{X}_\si}\,,
        \end{align}
        for some sequence $\{\ep_\si\}_{\si \in \Sigma} \subset \R_+$ with $\ep_\si \to 0$ as $\si \to 0$ and constant $C_E > 0$ depending on $E$. 
    \end{enumerate}
\end{definition}

\noindent \mb{The condition \eqref{eq:bdd_adjoint} for $\jj_\si^*$ should be regarded as a stability condition. Here, we choose to introduce 
\eqref{eq:bdd_adjoint} and
\eqref{est:conjaction} together for ease of exposition.} 


\mb{For notational simplicity, in the following,} we shall denote by $\fint$ the average integration, e.g., $\fint_a^b = \frac{1}{b-a}\int_a^b\,$, and the generic sequence $\{\ep_\si\}_\si \subset \R_+$ with $\ep_\si \to 0$ as $\si \to 0$ may differ from line to line. We also sometimes use $x \lesssim y$ to denote an inequality $x \le C y$ for some generic $C > 0$.

\begin{theorem}[Asymptotic lower bound]\label{thm:lower_abs}
Given $\ms{G}^\si_0, \ms{G}^\si_1 \in  {\rm X}_{\sigma,+}$ 
as in Definition \ref{def:disc_transp} satisfying \eqref{eq:convdis_inifal}, let $\mu^\ts = (G^\ts, q^\ts, R^\ts) \in \ce_\ts([0,1]; \ms{G}_0^\si, \ms{G}_1^\si)$ for $(\ts) \in \Sigma_\NN$ be a sequence satisfying
    \begin{equation} \label{asp:bound_cost}
      \mathbf{J} := \sup_{(\ts) \in \Sigma_{\NN}} \jj_{\tau,\sigma}(\mu^\ts) < + \infty\,.
    \end{equation} 
Suppose there exist consistent reconstructions 
$\big[ \h{\rr}_{\ol{\rm X}_\si} \big]$ and $\big[\h{\rr}_{{\rm Y}_\si}\big]$ in the sense of Definition \ref{def:cons_recons}. Then, $\norm{\mu^\ts}_{1,\M_\ts}$ is bounded independent of $(\ts)$. Further, there exists a measure $\mu \in \ce_{\infty}([0,1];\ms{G}_0,\ms{G}_1)$ such that
    for any admissible reconstruction $ \rr_\ts \in \rr_\ts^{\rm ad}$ in \eqref{def:addtsr}, up to a subsequence, $\rr_\ts(\mu^\ts)$ weak* converges to $\mu$ with 
    \begin{equation} \label{eq:est_lowerbound}
        \jj_Q(\mu) \le \liminf_{(\ts) \to 0} \jj_{\tau,\sigma}(\mu^\ts)\,.
    \end{equation} 
\end{theorem}

\begin{proof}
We first show the boundedness of $\norm{\mu^\ts}_{1,\M_\ts}$. 
Recalling the time-space reconstructions
 $\widehat{\rr}^{l/a}_{{\rm X}_\ts}$, $\h{\rr}_{{\rm Y}_\ts}$, and $\h{\rr}_{\ol{{\rm X}}_\ts}$ defined above, we have, for any $u \in C(Q,\R^{n\t k})$ and $W \in C(Q,\MM^n)$ with $\norm{u}_{\infty} \le 1$ and $\norm{W}_{\infty} \le 1$,
\begin{align} \label{eq:bounnd}
    & \big\l \h{\rr}_{{\rm Y}_\ts} q^\ts, u \big\r_Q  + \big\l \h{\rr}_{\ol{{\rm X}}_\ts} R^\ts, W \big\r_Q = \sum_{i = 1}^N \tau \big \l q^\ts_i,  \h{\rr}^*_{{\rm Y}_\si} \ol{u}_i \big\r_{\ms{Y}_\si} + \tau  \big\l R^\ts_i ,  \h{\rr}^*_{\ol{\rm X}_\si} \ol{W}_i  \big\r_{\ol{\ms{X}}_\si} \notag \\ 
   \le\ &   2 \sqrt{\sum_{i = 1}^N \tau\jj_\si\left(\frac{G^\ts_{i-1} + G^\ts_i}{2}, q^\ts_i, R^\ts_i\right)} \sqrt{\sum_{i = 1}^N \tau \jj^*_\si \left(\frac{G^\ts_{i-1} + G^\ts_i}{2},\h{\rr}^*_{{\rm Y}_\si} \ol{u}_i, \h{\rr}^*_{{\rm X}_\si} \ol{W}_i \right)} \notag \\
   \lesssim \  & \sqrt{\mathbf{J}\, \sum_{i = 0}^N \tau \norm{G^\ts_i}_{1,\ms{X}_\si}}\,, 
\end{align} 
where functions $\ol{u}_i$ and $\ol{W}_i$ are defined by 
\begin{equation} \label{aux:aver}
\ol{u}_i(\dd): =  \fint_{(i - 1)\tau}^{i \tau} u(s,\dd) \, \rd s \,,\q \ol{W}_i(\dd): =  \fint_{(i - 1)\tau}^{i \tau} W(s,\dd) \, \rd s\,,
\end{equation}
 satisfying $\norm{\ol{u}_i}_\infty \le 1$ and  $\norm{\ol{W}_i}_\infty \le 1$. In proving \eqref{eq:bounnd}, the first inequality is by \eqref{eq:disana_Cauchy_2} and Cauchy's inequality; the second inequality is by assumptions \eqref{eq:bdd_adjoint} and \eqref{asp:bound_cost}. It then follows that 
\begin{align}  \label{eq:auxest_qr}
    \big \lVert \h{\rr}_{{\rm Y}_\ts}(q^\ts) \big \lVert_{\rm TV} +  \big \lVert 
 \h{\rr}_{\ol{\rm X}_\ts}(R^\ts) \big \lVert_{\rm TV}
    \lesssim  \sqrt{\mathbf{J} \normm{G^\ts}_{1,\ms{X}_\ts}}\,. 
\end{align}
We next act the quasi-isometric reconstruction $\h{\rr}_{{\rm X}_\si}$ on 
both sides of \eqref{eq:discontieq} with test function $I$, and obtain 
\begin{align*}
    \left \l \tau^{-1}\h{\rr}_{{\rm X}_\si}\big(G^\ts_{k} - G^\ts_{k-1}\big), I \right\r_{\Omega} = \left\l - \h{\rr}_{{\rm X}_\si}\big(\mathsf{D}_\sigma q^\ts_k\big) + \h{\rr}_{{\rm X}_\si}\big(\mathbb{P}_{\ms{X}_\si} R^\ts_k\big), I \right\r_\Omega \,,\q k = 1,\ldots, N\,,
\end{align*}
which implies, for $i,j \in \{1,\ldots, N\}$ with $i < j$,
\begin{align} \label{eq:auxest_g}
    \left|\tr \h{\rr}_{{\rm X}_\si} \big(G^\ts_{j}\big)(\Omega) - \tr \h{\rr}_{{\rm X}_\si} \big(G^\ts_{i}\big)(\Omega) \right| & \lesssim \sum_{k = i+1}^j \tau \left(\ep_\si \normm{q^\ts_k}_{1,\ms{Y}_\si} + \big \lVert\h{\rr}_{\ol{\rm X}_\si}R^\ts_k \big \lVert_{\rm TV}\right)\,,
\end{align}
by the relation \eqref{eq:invari_space} and the estimate $ |\l \h{\rr}_{{\rm X}_\sigma}  (\sd q^\ts_k), I \r_{\Omega}|\lesssim \ep_\si \norm{q^\ts_k}_{1,\ms{Y}_\si}$ from \eqref{eq:conserr_adjD} with $\phi = I$ (note $\ms{D}^*(I) = 0$). 
Recalling $G_i^\ts \in {\rm X}_{\si,+}$ and $\h{\rr}_{{\rm X}_\si}: {\rm X}_{\si,+} \to \M(\Omega, \S^n_+)$, 
by \eqref{eq:auxest_g}, we have
\begin{equation*}
    \big \lVert G^\ts_j \big \lVert_{1,\ms{X}_\si} \lesssim \norm{\ms{G}^\si_0}_{1,\ms{X}_\si} + \big \lVert \h{\rr}_{{\rm Y}_\ts} q^\ts  \big \lVert_{\rm TV} +  \big \lVert \h{\rr}_{\ol{\rm X}_\ts} R^\ts \big \lVert_{\rm TV}\,, \q  j = 1,\ldots,N\,,
\end{equation*}
thanks to the quasi-isometric property, which, along with \eqref{eq:auxest_qr}, yields 
\begin{equation} \label{auueq}
    \normm{G^\ts}_{1,\ms{X}_\ts} =  \sum_{ i =  0}^N \tau \normm{G^\ts_i}_{1,\ms{X}_\si} \lesssim  \normm{\ms{G}_0^\si}_{1,\ms{X}_\si} + \sqrt{ \mathbf{J} \normm{G^\ts}_{1,\ms{X}_\ts}}\,.
\end{equation}
Note that $\norm{\ms{G}^\si_0}_{1,\ms{X}_\si}$ is bounded by the weak* convergence  \eqref{eq:convdis_inifal}. By arguments similar to \cite[(3.31)]{li2020general1}, the estimate \eqref{auueq} gives that $\norm{G^\ts}_{1,\ms{X}_\ts}$ is bounded by a constant independent of $(\ts)$, and so is $\norm{q^\ts}_{1, \ms{Y}_{\ts}} + \norm{R^\ts}_{1, \ol{\ms{X}}_\ts}$ by \eqref{eq:auxest_qr} and the quasi-isometric reconstruction again. 
We have proved the boundedness of $\norm{\mu^\ts}_{1,\M_\ts}$. Then, by Lemma \ref{lemma:unifrombdd}, up to a subsequence, for any $\rr_\ts \in \rr^{\rm ad}_\ts$, $\{\rr_{\ts}(\mu^\ts)\}_{\ts}$ weak* converges to some $\mu \in \mc{M}(Q,\xx)$. 

Next, we show that the limiting measure $\mu$ satisfies the continuity equation. We consider a special family of reconstructions $\rr_{\ts} := (\h{\rr}^{l}_{{\rm X}_\ts}, \rr_{{\rm Y}_{\ts}}, \h{\rr}_{\ol{\rm X}_\ts}) \in \rr^{\rm ad}_\ts$, where 
 $\rr_{{\rm Y}_\ts}$ is given in \eqref{eq:tsry} with $\rr_{{\rm Y}_\si}$ such that \eqref{eq:conserr_adjD} holds. 
We will show that for any $\Phi\in C^2(Q,\S^n)$, 
\begin{align} \label{eq:desiredlim}
    \lim_{(\tau,\sigma) \to 0} \Big\l \h{\rr}^{l}_{{\rm X}_\ts} G^\ts, \p_t\Phi \Big\r_{Q}  +   \Big\l \rr_{{\rm Y}_\ts}q^\ts, \mathsf{D^*} \Phi \Big\r_{Q} + \Big\l \h{\rr}_{\ol{{\rm X}}_\ts}R^\ts, \Phi \Big\r_Q  = \l \ms{G}_1, \Phi_1\r_\Omega - \l \ms{G}_0, \Phi_0\r_\Omega\,.
 \end{align} 
 For this,  we start with the following calculation:   
\begin{align} 
    & \Big\l \h{\rr}^{l}_{{\rm X}_\ts} G^\ts, \p_t\Phi \Big\r_{Q}  +   \Big\l \rr_{{\rm Y}_\ts}q^\ts, \mathsf{D^*} \Phi \Big\r_{Q} + \Big\l \h{\rr}_{\ol{{\rm X}}_\ts}R^\ts, \Phi \Big\r_Q \notag \\
  = & -\sum_{k = 1}^N \int_{(k-1)\tau}^{k\tau} \Big \l \tau^{-1} \Big(\h{\rr}_{{\rm X}_\sigma}G^\ts_{k} - \h{\rr}_{{\rm X}_\sigma}G^\ts_{k-1}\Big), \Phi_t \Big\r_\Omega + \left\l \h{\rr}_{{\rm X}_\sigma} \ms{G}^\si_1,  \Phi_1 \right\r_{\Omega} -  \left\l \h{\rr}_{{\rm X}_\sigma}\ms{G}^\si_0,  \Phi_0 \right\r_{\Omega} \, \rd t \notag \\
   & + \sum^N_{k = 1} \int_{(k-1)\tau}^{k\tau} \Big\l \rr_{{\rm Y}_\sigma} q^\ts_k, \mathsf{D}^* \Phi_t \Big\r_\Omega \, \rd t + \sum^N_{k = 1} \int_{(k-1)\tau}^{k\tau} \Big\l \h{\rr}_{\ol{\rm X}_\sigma}R^\ts_k, \Phi_t \Big\r_\Omega \, \rd t \notag \\
   = &\   \Big\l \h{\rr}_{{\rm X}_\sigma}\ms{G}^\si_1,  \Phi_1 \Big\r_{\Omega} -  \Big\l \h{\rr}_{{\rm X}_\sigma}\ms{G}^\si_0,  \Phi_0 \Big\r_{\Omega} + O\left(\ep_\si\norm{q^\ts}_{1,\ms{Y}_{\ts}} \norm{\Phi}_{2,\infty}\right)\,, \label{eq:auxest_3}
\end{align}
where the first equality is by definitions \eqref{eq:tsrxl}--\eqref{eq:tsrz} with integration by parts; the second equality is by the following estimate derived by acting $\h{\rr}_{{\rm X}_\si}$ on \eqref{eq:discontieq} with test functions $\Phi \in C^2(\Omega,\S^n)$ and using \eqref{eq:invari_space} and \eqref{eq:conserr_adjD}:
\begin{align*}
    0 & = \big \l \h{\rr}_{{\rm X}_\sigma} (\tau^{-1}(G^\ts_{k} - G^\ts_{k-1})), \Phi_t \big\r_{\Omega} + \big\l \h{\rr}_{{\rm X}_\sigma} \left(\sd q^\ts_k \right), \Phi_t \big\r_{\Omega}  - \big \l \h{\rr}_{\ol{\rm X}_\sigma} R^\ts_k, \Phi_t \big\r_{\Omega} \\
    & = \big\l \h{\rr}_{{\rm X}_\sigma} (\tau^{-1}(G^\ts_{k} - G^\ts_{k-1})), \Phi_t \big\r_{\Omega} - \big \l \rr_{{\rm Y}_\sigma} q^\ts_k, \ms{D}^* \Phi_t \big\r_{\Omega}  -  \big\l \h{\rr}_{\ol{\rm X}_\sigma} R^\ts_k, \Phi_t \big\r_{\Omega} + O\left(\ep_\si\norm{q^\ts_k}_{1,\ms{Y}_\si} \norm{\Phi}_{2,\infty}\right),
\end{align*}  
for $k = 1\,\ldots, N$. Then, by \eqref{eq:convdis_inifal}, the boundedness of $\norm{q^\ts}_{1, \ms{Y}_{\ts}}$, and the weak* convergence (up to a subsequence) of $\rr_{\ts}(\mu^\ts)$,  taking the limit on both sides of \eqref{eq:auxest_3} gives the desired \eqref{eq:desiredlim}.

Finally, we prove the estimate \eqref{eq:est_lowerbound}. Recalling the formula \eqref{eq:fenconj_cost}, for any $\ep > 0$, we can choose $u \in C^1(Q, \R^{n \t k})$ and $W \in C^1(Q, \mathbb{M}^n)$ such that 
\begin{align} \label{eq:auxest_5}
    \jj_Q(\mu) \le  \l \ms{(q,R)}, (u,W)\r_{Q} - \jj_Q^*(\ms{G},u,W) + \ep \,,
\end{align}
where $\mu = \ms{(G,q,R)}$ is the limiting measure defined above. Then, let 
$\w{\rr}^a_{{\rm X}_\ts}, \rr_{{\rm Y}_{\ts}}$, and  $\rr_{\ol{\rm X}_\ts}$ be the reconstructions defined in \eqref{eq:tsrxa}--\eqref{eq:tsrz} based on the spatial ones such that \eqref{est:conjaction} holds.
 By \eqref{eq:auxest_5} and \eqref{def:jalpha_conj}
 and  the weak* convergence of the reconstructed measures,  we have
\begin{equation}\label{eq:auxest_4}
    \jj_Q(\mu) \le  \lim_{(\ts) \to 0}  \Big\l \big(\rr_{{\rm Y}_\ts} q^\ts, \rr_{\ol{{\rm X}}_\ts} R^\ts \big), (u,W) \Big\r_{Q}  + \ep  - \sum_{i = 1}^N  \int_{(i-1)\tau}^{i \tau}\jj_{\Omega}^* \Big(\w{\rr}_{{\rm X}_\si}\Big(\frac{G^\ts_{i-1} + G^\ts_i}{2} \Big), u_t, W_t \Big)\,\rd t \,. 
\end{equation}
Since $\jj_\Omega^*$ is convex with respect to $(u,W)$, by Jensen's inequality and \eqref{est:conjaction}, we have 
\begin{align*}
  &\sum_{i = 1}^N  \int_{(i-1)\tau}^{i \tau}\jj_\Omega^*\Big(\w{\rr}_{{\rm X}_\si}\Big(\frac{G^\ts_{i-1} + G^\ts_i}{2} \Big), u_t, W_t\Big)\, \rd t \ge  \sum_{i = 1}^N   \tau \jj_\Omega^*\Big(\w{\rr}_{{\rm X}_\si}\Big(\frac{G^\ts_{i-1} + G^\ts_i}{2} \Big), \ol{u}_i, \ol{W}_i \Big) \\
\ge  & \sum_{i = 1}^N \tau \jj^*_\si\Big(\frac{G^\ts_{i-1} + G^\ts_i}{2}, \rr^*_{{\rm Y}_\si} \ol{u}_i, \rr^*_{\ol{\rm X}_\si} \ol{W}_i \Big) -  \ep_\sigma C \norm{G^\ts}_{1, \ms{X}_{\ts}}\,,
\end{align*}
where $\ol{u}_i$ and $\ol{W}_i$ are given as in \eqref{aux:aver}.
Combining the above inequality with \eqref{eq:auxest_4}, we arrive at  
\begin{multline}
    \jj_Q(\mu) \le  \liminf_{(\ts) \to 0} \bigg(\sum_{i = 1}^N \tau \Big\l \big(\rr_{{\rm Y}_\si} q^\ts_i, \rr_{\ol{\rm X}_\si} R_i^\ts\big), (\ol{u}_i,\ol{W}_i)\Big\r_{\Omega}  + \ep +  \ep_\sigma C \norm{G^\ts}_{1, \ms{X}_{\ts}}   \\  - \sum_{i = 1}^N \tau \jj^*_\si\Big(\frac{G^\ts_{i-1} + G^\ts_i}{2}, \rr^*_{{\rm Y}_\si} \ol{u}_i, \rr^*_{\ol{\rm X}_\si} \ol{W}_i \Big) \bigg) \le \liminf_{(\ts) \to 0}\jj_\ts(\mu^\ts) + \ep \,,
\end{multline}
by noting that $\jj_\si$ is the conjugate of $\jj_\si^*$ and the definition \eqref{eq:discost} of $\jj_\ts$. The proof is complete as $\ep$ is arbitrary.
\end{proof}

We next consider the asymptotic upper bound. For this, we introduce consistent sampling operators.

\begin{definition}\label{def:cons_sampl}
Let $\ss_{\ol{{\rm X}}_\si}$ and $\ss_{{\rm Y}_\si}$ be the bounded linear operators from $C(\Omega, \MM^n)$ to $(\ol{{\rm X}}_\si, \norm{\dd}_{*,\ol{\ms{X}}_\si})$ and  $C(\Omega, \R^{n \t k})$ to $({\rm Y}_\si,\norm{\dd}_{*,\ms{Y}_\si})$, respectively, such that $\ss_{\ol{{\rm X}}_\si}$ maps $C(\Omega, \S^n)$ to ${\rm X}_\si$ and $C(\Omega, \AA^n)$ to ${\rm X}^\perp_\si$.   
We say that they are families of 
consistent sampling operators if the following conditions hold:  
\begin{enumerate}[\textbullet]
    \item (consistency of $\ms{D}$ and $\sd$). There exists $\{\ep_\si\}_{\si \in \Sigma} \subset \R_+$ with $\ep_\si \to 0$ as $\si \to 0$ such that 
    \begin{align} \label{eq:conserr_D}
        \norm{\sd \ss_{{\rm Y}_\sigma} (\phi) - \ss_{\ol{{\rm X}}_\si} \mathsf{D} (\phi) 
        }_{*,\ms{X}_\si} \le \ep_\sigma  \norm{\phi}_{2,\infty}\,,  
    \end{align}
    for $\phi \in C^\infty(\Omega,\R^{n \t k})$ satisfying $\h{\ms{D}_1}(-i \nu)\phi = 0$ on $\p \Omega$. 
    \item (consistency of $\jj_\Omega$ and $\jj_\si$). For $C^\infty$-smooth $\xx$-valued field $(G,q,R)$ satisfying $\norm{(G,q,R)}_{2,\infty} \le E < + \infty$ and  $G \succeq \eta I$ for some $\eta > 0$ and $\h{\ms{D}_1}(-i\nu)q = 0$ on $\p \Omega$, there exist a constant $C_{E,\eta} > 0$ depending on $E$ and $\eta$ and a sequence $\{\ep_\si\}_{\si \in \Sigma} \subset \R_+$ with $\ep_\si \to 0$ as $\si \to 0$ such that 
    \begin{align} \label{eq:consis_energy}
        \jj_\si (\ss_{\ol{{\rm X}}_\si}(G), \ss_{{\rm Y}_\si}(q), \ss_{\ol{{\rm X}}_\si}(R) + {\rm e}_{\si,q}) \le \jj_\Omega(\mu) + C_{E,\eta}\, \ep_\sigma \,,
    \end{align} 
    where $\mu = (G,q,R)\, \rd x \in \mc{M}(\Omega,\xx)$ and ${\rm e}_{\si,q} = \sd \ss_{{\rm Y}_\sigma} (q) - \ss_{\ol{{\rm X}}_\si} \mathsf{D} (q)$.
\end{enumerate}
\end{definition}

\begin{remark}
The sampling operators play a similar role as interpolation operators in the Lax equivalence theorem. In practice, they are typically defined as the adjoint of reconstructions. It is also worth noting that compared to the consistency conditions \eqref{eq:conserr_adjD}--\eqref{est:conjaction} in Definition \ref{def:cons_recons}, which are defined for the dual objects, the consistency conditions in Definition \ref{def:cons_sampl} above are defined on the primal objects. 
\end{remark}

For ease of exposition, we introduce the following assumption. 

\begin{enumerate}[label=$\textbf{A}*$,ref=$\text{A}*$]
\item \label{asm:end_Point}  Let $\ms{G}_0^\si, \ms{G}_1^\si \in {\rm X}_{\si,+}$ be given in Definition \ref{def:disc_transp}.
For any $\mu \in \ce_\infty([0,1];\ms{G}_0,\ms{G}_1)$, let $\mu^\ep \in \ce_\infty([0,1]; \ms{G}^\ep_0,\ms{G}^\ep_1)$ be its  \emph{$\ep$-regularization} in Proposition \ref{prop:regular_ep} with smooth density $(G^\ep,q^\ep,R^\ep)$. We assume that for $i = 0 , 1$, there exists $\mu^\ts_i \in \ce_\ts\big([0,1]; \ms{G}_i^\si, \ss_{\ol{{\rm X}}_\si} G_i^\ep\big)$
for $(\ts) \in \Sigma_\NN$ such that 
\begin{align}  \label{assumpest}
  \ww_{\ep,i}:= \limsup_{(\ts) \to 0} \jj_\ts(\mu_i^\ts) \to 0\,,\q  \text{as} \ \, \ep \to 0\,.
\end{align}
\end{enumerate}

\begin{theorem}[Asymptotic upper bound]\label{thm:upper_abs}
    Let $\mu \in \ce_\infty([0,1];\ms{G}_0,\ms{G}_1)$ for given  $\ms{G}_0, \ms{G}_1 \in \M(\Omega, \S^n_+)$. Suppose that 
    there exist consistent sampling operators $\ss_{\ol{{\rm X}}_\si}$ and $\ss_{{\rm Y}_\si}$ in the sense of Definition \ref{def:cons_sampl}, and that 
    the assumptions \eqref{1}, \eqref{2}, and \eqref{asm:end_Point} hold.
    Then, for any $\eta > 0$, we can construct a discrete curve $\mu^\ts \in \ce_\ts([0,1];\ms{G}_0^\si, \ms{G}^\si_1)$ such that 
    \begin{equation} \label{est:asyupper}
        \limsup_{(\ts) \to 0} \jj_{\tau,\sigma}(\mu^\ts) \le \jj_Q(\mu) + \eta\,.
    \end{equation}
\end{theorem}

Its proof relies on the following lemma. 

\begin{lemma} \label{lem:regular_curve}
    Let $\mu \in \ce_\infty([0,1];\ms{G}_0,\ms{G}_1)$ be a measure with $C^\infty$-smooth density $(G,q,R)$ with respect to $\rd t \otimes \rd x$ with $G \succeq c_0 I$ on $Q$ for some $c > 0$. Then, there exists 
    $\mu^\ts \in \ce_\ts([0,1];  \ss_{\ol{{\rm X}}_\si}G_0, \ss_{\ol{{\rm X}}_\si}G_1)$ with the estimate: for some constant $C > 0$ depending on $\norm{\mu}_{2,\infty}$ and $c$, and some sequence $\{\ep_\si\}_{\si\in\Sigma} \subset \R_+$ with $\ep_\si \to 0$ as $\si \to 0$, 
    \begin{equation} \label{est:regular_curve}
        \jj_{\tau,\sigma}(\mu^{\tau,\sigma}) \le \jj_Q(\mu) + C_{c_0,\norm{\mu}_{2,\infty}}(\ep_\si + \tau)\,.
    \end{equation}
\end{lemma}
\begin{proof}
By the smoothness of density $(G,q,R)$, we define $\mu^\ts = (G^\ts, q^\ts,R^\ts) \in \mc{M}_\ts$ by  
    \begin{align} \label{cons:discrete_curve}
        \big(G_i^\ts, q_i^\ts,R_i^\ts\big) := \left(\ss_{\ol{{\rm X}}_\si} (G_{i\tau}),  \ss_{{\rm Y}_\si} (\ol{q}_i) , \ss_{\ol{{\rm X}}_\si} (\ol{R}_i) + {\rm e}_{\si,\ol{q}_i} \right)\,, \q  i = 1,\ldots, N\,,
    \end{align}
with $G^\ts_0 := \ss_{\ol{{\rm X}}_\si}(G_t|_{t = 0})$, where smooth functions $\ol{q}_i$, $\ol{R}_i$, and the consistency error ${\rm e}_{\si,\ol{q}_i} \in {\rm X}_\si$ are defined by 
\begin{equation*}
\ol{q}_i = \fint^{i \tau}_{(i - 1)\tau} q_t \, \rd t\,,\q \ol{R}_i = \fint^{i \tau}_{(i - 1)\tau} R_t \, \rd t\,,\q  {\rm e}_{\si,\ol{q}_i} :=  \sd \ss_{{\rm Y}_\sigma} (\ol{q}_i) - \ss_{\ol{{\rm X}}_\si} \mathsf{D} (\ol{q}_i)\,.
\end{equation*}
Recalling from Definition \ref{def:cons_sampl} that $\ss_{\ol{{\rm X}}_\si}^*$ maps from ${\rm X}_\si$ to $\M(\Omega, \S^n)$, we have
\begin{align*} 
    \left\l  M^\si, \P_{{\rm X}_\si} \ss_{\ol{{\rm X}}_\si}\ol{R}_i\right\r_{\ms{X}_\si} =   \left\l \ss^*_{\ol{{\rm X}}_\si} M^\si, \ol{R}_i \right\r_{\Omega}  = \left\l  M^\si, \ss_{\ol{{\rm X}}_\si} (\ol{R}_i)^{\rm sym}  \right\r_{\ms{X}_\si}\,, \q \forall M^\si \in {\rm X}_\si\,,
\end{align*}
and then there holds $\mu^\ts \in \ce_\ts([0,1]; \ss_{\ol{{\rm X}}_\si}G_0, \ss_{\ol{{\rm X}}_\si}G_1)$ by the construction \eqref{cons:discrete_curve}.

We next show \eqref{est:regular_curve}. First, it follows from  \eqref{eq:consis_energy} and \eqref{cons:discrete_curve} that for $k = 1,\ldots, N$, 
    \begin{align*}
         \jj_\sigma \Big(\frac{G^\ts_{k-1} + G^\ts_k}{2},q^\ts_k, R^\ts_k \Big) \le \jj_\Omega \Big(\frac{G_{(k-1)\tau} + G_{k\tau}}{2}, \ol{q}_i, \ol{R}_i \Big)  + C \ep_\sigma  \,.
        \end{align*} 
\mb{Thanks to the smoothness of $(G,q,R)$ and the uniform positivity $G \succeq c_0 I$, recalling definitions  \eqref{eq:expre_J} and \eqref{def:costmeasure},  we estimate, for small enough $\tau$, 
\begin{align*}
    \jj_\Omega \Big(\frac{G_{(k-1)\tau} + G_{k\tau}}{2}, \ol{q}_i, \ol{R}_i \Big) = \jj_\Omega \Big(G_{(k-\frac{1}{2})\tau}, q_{(k-\frac{1}{2})\tau}, R_{(k-\frac{1}{2})\tau} \Big) + O(\tau)\,,
\end{align*}
where we use the expansion \cite{bhatia2013matrix}:  $(A + H)^{-1} = A^{-1} - A^{-1}H A^{-1} + O(\norm{H}^2)$ as $\norm{H} \to 0$ for invertible $A$. Then, by the definition \eqref{eq:discost} of $\jj_{\tau,\sigma}(\mu^{\tau,\sigma})$,  there holds}
\begin{align*}
    \jj_{\tau,\sigma}(\mu^{\tau,\sigma}) & \le \tau \sum^N_{k  = 1} \jj_\Omega \Big(\frac{G_{(k-1)\tau} + G_{k\tau}}{2}, \ol{q}_i, \ol{R}_i \Big) + C \ep_\si \\
    & \le \int_0^1 \jj_\Omega(\mu_t)\, \rd t  +  C (\ep_\si + \tau)\,. \qedhere
\end{align*}
\end{proof}

\begin{proof}[Proof of Theorem \ref{thm:upper_abs}] 
We fix a small $t^*_0 \in (0,1)$ and denote $t_0 = n \tau$ for small enough $\tau$, where $n$ is uniquely determined by $n\tau \le t^*_0 < (n + 1)\tau$. Recall the assumption \eqref{asm:end_Point} and let $G_i^\ep$ ($i = 0,1$) be the density of $\ms{G}_i^\ep$. There exists $\mu^{\frac{\tau}{t_0},\si,(i)} \in \ce_{\frac{\tau}{t_0},\si}\big([0,1]; \ms{G}_i^\si, \ss_{\ol{{\rm X}}_\si} G_i^\ep\big)$
    satisfying,  for some $\{ \ep_{(\ts),i}\}_{(\ts) \in \Sigma_\NN}$ 
    with $\ep_{(\ts),i} \to 0$ as $(\ts) \to 0$, 
    \begin{equation}\label{eq:auxes_end}
        \jj_{\frac{\tau}{t_0},\si}\Big(\mu^{\frac{\tau}{t_0},\si,(i)}\Big) \le \ww_{\ep,i} + \ep_{(\ts),i}\,,
    \end{equation}  
     By Lemma \ref{lem:regular_curve}, we can construct  
    $\mu^{\frac{\tau}{1 -2 t_0},\si,(2)} \in \ce_{\frac{\tau}{1 -2 t_0},\si}\big([0,1];\ss_{\ol{{\rm X}}_\si} G^\ep_0,\ss_{\ol{{\rm X}}_\si} G^\ep_1\big)$  with the estimate:
    \begin{equation} \label{eq:auxes_reg} 
        \jj_{\frac{\tau}{1 -2 t_0},\sigma}\Big(\mu^{\frac{\tau}{1 -2 t_0},\sigma,(2)}\Big) \le \jj_Q(\mu^\ep) + C_{\ep,\norm{\mu^\ep}_{2,\infty}}\Big(\ep_\si + \frac{\tau}{1 - 2t_0}\Big)\,,
    \end{equation}
    where $\mu^\ep$ is the $\ep$-regularization of $\mu$ and $C_{\ep,\norm{\mu^\ep}_{2,\infty}}$ is 
    independent of parameters $\si$, $\tau$, and $t_0$. We glue the above discrete curves and define $\mu^{\ts}$ by
\begin{equation*} 
    \mu^\ts = \left\{ 
    \begin{aligned}
    & \left(G^{\frac{\tau}{t_0},\si,(0)}_i, \frac{1}{t_0} q^{\frac{\tau}{t_0},\si,(0)}_i,\, \frac{1}{t_0}R^{\frac{\tau}{t_0},\,\si,(0)}_i \right)\,, &&  i = 1,\ldots, \frac{t_0}{\tau}\,, \\ 
    & \left(G^{\frac{\tau}{1 -2 t_0},\si,(2)}_{i - \frac{t_0}{\tau}},\, \frac{1}{1- 2 t_0} q^{\frac{\tau}{1 -2 t_0} ,\si,(2)}_{i - \frac{t_0}{\tau}},\, \frac{1}{1- 2 t_0} R^{\frac{\tau}{1 -2 t_0},\si,(2)}_{i - \frac{t_0}{\tau}} \right)\,, && i = \frac{t_0}{\tau} + 1,\ldots, \frac{1 - t_0}{\tau} \,, \\
    & \left(G_{i - \frac{\tau}{1 - t_0}}^{\frac{\tau}{t_0},\si,(1)},\, \frac{1}{t_0}q_{i - \frac{\tau}{1 - t_0}} ^{\frac{\tau}{t_0},\si,(1)},\, \frac{1}{t_0}R_{i - \frac{\tau}{1 - t_0}} ^{\frac{\tau}{t_0},\si,(1)} \right)\,, && i= \frac{1 - t_0}{\tau} +1,\ldots, \frac{1}{\tau} \,.
    \end{aligned}
        \right.
    \end{equation*} 
    It is clear that $\mu^\ts \in \ce_\ts([0,1];\ms{G}_0^\si,\ms{G}_1^\si)$. 
    By definition \eqref{eq:discost} of $\jj_\ts$ and estimates \eqref{eq:auxes_end} and \eqref{eq:auxes_reg}, we have 
    \begin{align*} 
        \jj_{\ts}(\mu^\ts) & = \frac{1}{t_0} \jj_{\frac{\tau}{t_0},\si}\big(\mu^{\frac{\tau}{t_0},\si,(0)}\big)  + \frac{1}{1 - 2t_0} \jj_{\frac{\tau}{1 - 2t_0},\si}\big(\mu^{\frac{\tau}{1 - 2t_0},\si,(2)}\big)  + \frac{1}{t_0} \jj_{\frac{\tau}{t_0},\si}\big(\mu^{\frac{\tau}{t_0},\si,(1)}\big) \\ 
        & \le \frac{2}{t_0}\big(\ww_\ep + \ep_{(\ts)}\big) + \frac{1}{1 - 2t_0} \Big(\jj_Q(\mu^\ep) + C_{\ep,\norm{\mu^\ep}_{2,\infty}}\Big(\ep_\sigma + \frac{\tau}{1 - 2t_0}\Big) \Big)\,,
    \end{align*} 
    where $\ww_\ep = \max\{ \ww_{\ep,0}, \ww_{\ep,1}\}$ and $\ep_{(\ts)} = \max \{\ep_{(\ts),0},\ep_{(\ts),1}\}$. 
    It readily follows that  
    \begin{align} \label{auxest:upper}
      \limsup_{(\ts) \to 0}  \jj_{\ts}(\mu^\ts) \le \frac{2}{t^*_0}\ww_\ep  + \frac{\jj_Q(\mu^\ep) - \jj_Q(\mu) + 2 t^*_0 \jj_Q(\mu)}{1 - 2t^*_0} + \jj_Q(\mu)\,,
    \end{align}
 for arbitrary fixed $\ep$ and $t_0^*$.  Then, for given $0 < \eta < 1$, we first choose $t_0^*$ small enough such that 
    \begin{equation*}
        \frac{2 t_0^*}{1 - 2 t^*_0} \jj_Q(\mu) \le \frac{\eta}{3}\,,
    \end{equation*}
    and then, by \eqref{assumpest} and Proposition \ref{prop:regular_ep}, choose $\ep$ small enough such that 
    \begin{equation*}
        \frac{2}{t^*_0}\ww_\ep  \le \frac{\eta}{3}\,,\q \frac{\jj_Q(\mu^\ep) - \jj_Q(\mu)}{1 - 2t^*_0}  \le \frac{\eta}{3}\,,
    \end{equation*}  
    which, by \eqref{auxest:upper}, implies the 
  desired \eqref{est:asyupper}.
\end{proof}

We are now ready to give an abstract convergence result for the discrete OT model \eqref{prob:discre_ot}.

\begin{theorem} \label{thm:main_conver}
Suppose that \eqref{prob:discre_ot} admits a minimizer $\mu^\ts_*$ for each $(\ts) \in \Sigma_{\NN}$, and there exists consistent spatial reconstructions $\big[\h{\rr}_{\ol{{\rm X}}_\si}\big]$ and $\big[\h{\rr}_{{\rm Y}_\si}\big]$ in the sense of Definition \ref{def:cons_recons} and consistent samplings $\ss_{\ol{{\rm X}}_\si}$ and $\ss_{{\rm Y}_\si}$ in the sense of Definition \ref{def:cons_sampl}, and
    the assumptions \eqref{1}, \eqref{2}, and \eqref{asm:end_Point} hold. Then, there exists a minimizer $\mu_*$ to \eqref{eq:distance} such that
    for any $ \rr_\ts \in \rr_\ts^{\rm ad}$ in \eqref{def:addtsr}, up to a subsequence, $\rr_\ts(\mu_*^\ts)$ 
     weak* converges to $\mu_*$, and there holds 
    \begin{align*}
        \lim_{(\ts) \to 0} \jj_{\tau,\sigma}(\mu^\ts_*) =  \jj_Q(\mu_*)\,.
    \end{align*} 
\end{theorem}

\begin{proof}
For any given $\mu \in \ce_\infty([0,1];\ms{G}_0,\ms{G}_1)$ and $\ep > 0$,  by Theorem \ref{thm:upper_abs}, there exists $\mu^\ts$ satisfying 
    \begin{align*}
        \limsup_{(\ts) \to 0} \jj_{\tau,\sigma}(\mu_*^\ts) \le\limsup_{(\ts) \to 0} \jj_{\tau,\sigma}(\mu^\ts) \le \jj_Q(\mu) + \ep\,,
    \end{align*} 
    which gives
$\sup \{\jj_{\tau,\sigma}(\mu_*^\ts)\,;\ (\ts) \in \Sigma_\NN \}< +\infty$. Then, by Theorem \ref{thm:lower_abs}, 
    for $\rr_\ts \in \rr_\ts^{\rm ad}$, we can find 
     a subsequence of $\rr_\ts (\mu_*^\ts)$ (still indexed by $\ts$)  weak* converging to some $\mu_* \in \ce([0,1];\ms{G}_0,\ms{G}_1)$ with the estimate:
    \begin{align*}
      \jj_Q(\mu_*) \le  \liminf_{(\ts) \to 0} \jj_{\tau,\sigma}(\mu_*^\ts) \le  \limsup_{(\ts) \to 0} \jj_{\tau,\sigma}(\mu_*^\ts)  \le \jj_Q(\mu) + \ep \,.
    \end{align*}
    Since $\ep$ is arbitrary, we readily have
\begin{align*}
    \jj_Q(\mu_*) = \lim_{(\ts) \to 0} \jj_{\tau,\sigma}(\mu_*^\ts) \le \jj_Q(\mu)\,.
\end{align*}
for any $\mu \in \ce_\infty([0,1];\ms{G}_0,\ms{G}_1)$. The proof is complete. 
\end{proof}

Our consistency conditions in Definitions \ref{def:cons_recons} and \ref{def:cons_sampl} are largely inspired by (A3)-(A6) in \cite[Definition 2.9]{lavenant2019unconditional}. Similar conditions were also proposed in \cite[Section 3.1]{carrillo2022primal} for the finite difference approximation for the JKO scheme. 
A key feature of our framework is the introduction of discrete norms and admissible reconstructions in Definitions \ref{def:discre_appro} and \ref{def:build_block} motivated by the Lax equivalence theorem, which emphasizes the role played by the quasi-isometric reconstructions. More precisely, as we can see, some reconstructions involved in Definition \ref{def:cons_recons} have to be quasi-isometric. These facts could help to design a concrete numerical scheme and further verify its convergence. 

We thereby propose the following procedures for constructing a convergent discretization scheme for \eqref{eq:distance}, which will be the main task of the next section. Suppose that a family of meshes for the computation domain with mesh size $\si$ is given. First, we construct discrete approximations $(\ol{{\rm X}}_\si, \h{\rr}_{\ol{{\rm X}}_\si})$ and $({\rm Y}_\si, \h{\rr}_{{\rm Y}_\si})$ of $\M(\Omega,\MM^n)$ and $\M(\Omega, \R^{n \t k})$, respectively, with quasi-isometric $\h{\rr}_{\ol{{\rm X}}_\si}$ and $\h{\rr}_{{\rm Y}_\si}$ in the sense of \eqref{def:quasirecons}.
Second, we construct the discrete derivation $\sd$ and the discrete functional $\jj_\si$, as well as (if necessary) other families of reconstructions equivalent to quasi-isometric ones, such that the conditions in Definition \ref{def:cons_recons} hold. Third, we introduce sampling operators $\ss_{\ol{{\rm X}}_\si}$ and $\ss_{{\rm Y}_\si}$ (usually, the adjoint of reconstructions) such that the conditions in Definition \ref{def:cons_sampl} hold.

\section{Concrete numerical scheme} \label{sec:conver_example}
In this section, we apply the above abstract framework to design a concrete convergent numerical scheme for the problem \eqref{eq:distance}, which is closely related to the finite element method (FEM).
We will follow the basic mesh setup in FEM.
We assume that there exists a family of conforming and quasi-uniform triangulations $(\TT^\si, \V^\si)$ of $\Omega \subset \R^d$ with $\si$ being the mesh size: $\sigma := \max_{K \in \TT^\si} {\rm diam}(K)$, where ${\rm diam}(\dd)$ is the diameter of a set: ${\rm diam}(K): = \sup\{|x-y|\,;\ x,y \in K\}$; 
$\TT^\si$ denotes the set of simplices and $\V^\si$ is the set of vertices; see \cite[Definition 4.4.13]{brenner2008mathematical}. 

We fix more notations for later use. We denote by $|K|$ the volume of a simplex $K$ in $\TT^\si$. Let $\TT^\si_v$ be the set of simplices that contains $v$ as a vertex, and $\V^\si_K$ be the set of vertices of a simplex $K$. Moreover, we denote by $\h{\phi}_v$ the nodal basis function associated with the vertex $v$, that is, $\h{\phi}_v$ is piecewise linear (with respect to the triangulation) such that $\h{\phi}_v(v) = 1$ and $\h{\phi}_v(v') = 0$ for any $v' \in \V^\si$ with $v' \neq v$. We also recall the following useful relations:    
\begin{align} \label{eq:userela_tri}
    |T_v| := \int_{\Omega} \h{\phi}_v\, \rd x = \frac{1}{d + 1} \sum_{K \in \TT^\si_v}|K| \q \text{and} \q \sum_{v \in \V^\si} |T_v| = |\Omega|\,.
\end{align}

By the abstract convergent framework in Section \ref{sec:absconver}, it suffices to construct the discrete modules in Definition \ref{def:build_block} for the discretization of \eqref{eq:distance}. 


\medskip 
\noindent \textbf{Discrete spaces.} 
We define the discrete space $\ol{{\rm X}}_\sigma := \{(R_v)_{v \in \V^\si}\,;\  R_v \in \MM^n\}$ equipped with the inner product $\l \dd,\dd \r_{\V^\si}$ and the discrete $L^1$-norm $\norm{\dd}_{1,\V^\si}$: 
for $A , B \in \ol{{\rm X}}_\si$, 
\begin{equation} \label{defnorminner}
    \l A , B \r_{\V^\si}  = \sum_{v \in \V^\si} |T_v| A_v \dd B_v \,,\q  \norm{A}_{1,\V^\si}  = \sum_{v \in \V^\si} \norm{A_v}_\ff \, |T_v|\,,
\end{equation} 
while the discrete space for $\M(\Omega,\R^{n \t k})$ is given by $  {\rm Y}_\sigma := \{(q_K)_{K \in \TT^\si}\,;\  q_K \in \R^{n \t k}\}$ with the inner product $\l \dd,\dd\r_{\TT^\si}$ and norm $\norm{\dd}_{1,\TT^\si}$: for $p, q \in {\rm Y}_\si$, 
\begin{align*}
    \l p , q \r_{\TT^\si} := \sum_{K \in \TT^\si} |K| p_K \dd q_K\,, \q \norm{p}_{1,\TT^\si} := \sum_{K \in \TT^\si} \norm{p_K}_\ff \, |K| \,.
\end{align*}
The dual norm on $\ol{{\rm X}}_\si$ defined as in \eqref{def:dualnorm_dis} can be computed as: for $B \in \ol{{\rm X}}_\si$,
    \begin{align}\label{def:dualnorm_X}
        \norm{B}_{*,\V^\si} := \sup_{\norm{A}_{1,\V^\si} \le 1} |\l A , B \r_{\V^\si}| 
         = \max_{v\in \V^\si} \norm{B_v}_\ff\,.
    \end{align}
Similarly, the dual norm on ${\rm Y}_\si$ is given by $ \norm{q}_{*,\TT^\si} = \max_{K \in \TT^\si} \norm{q_K}_\ff$. Moreover, the subspace ${\rm X}_\si$ and the cone ${\rm X}_{\si,+}$ of $\ol{{\rm X}}_\si$ are defined by 
\begin{equation*}
    {\rm X}_{\sigma,+} := \{(G_v)_{v \in \V^\si} \,;\  G_v \in \S^n_+\} \subset  {\rm X}_\sigma := \{(S_v)_{v \in \V^\si}\,;\  S_v \in \S^n\}\,.
\end{equation*}

\medskip 
\noindent \textbf{Reconstructions.}
We proceed to introduce the reconstruction operators $\h{\rr}_{\ol{\rm X}_\si}$ on $\ol{{\rm X}}_\si$ and $\h{\rr}_{{\rm Y}_\si}$ on ${\rm Y}_\si$ by  
     \begin{align}\label{eq:repxr}
         \h{\rr}_{\ol{{\rm X}}_\si}(R) =  \sum_{v \in \V^\si} R_v |T_v| \d_{v}\,, \  \forall R \in \ol{{\rm X}}_\si \,,\q \h{\rr}_{{\rm Y}_\si}(q) = \sum_{K \in \TT^\si} q_K \chi_K\,, \ \forall q \in {\rm Y}_\si \,,
    \end{align} 
where $\chi_K$ is the indicator function \eqref{def:indicator}
of $K \in \TT^\si$ identified with the measure $\chi_K\, \rd x$, and $\d_v$ is the Dirac measure supported at the vertex $v$. It is easy to see that
\begin{equation*}
    \norm{\h{\rr}_{\ol{\rm X}_\si}(R)}_{\rm TV} = \norm{R}_{1,\V^\si} \q \text{and} \q \norm{\h{\rr}_{{\rm Y}_\si}(q)}_{\rm TV} = \norm{q}_{1,\TT^\si}\,,
\end{equation*}
which implies that $\h{\rr}_{\ol{{\rm X}}_\si}$ and $\h{\rr}_{{\rm Y}_\si}$ are quasi-isometric by Definition \ref{def:space_reconst}. We also define a piecewise linear reconstruction $\rr^l_{\ol{{\rm X}}_\si} \in \big[\h{\rr}_{\ol{{\rm X}}_\si}\big]$ on $\ol{{\rm X}}_\si$, 
motivated by the nodal interpolation in finite elements,
\begin{align*}
    \rr^l_{\ol{{\rm X}}_\si}(R) = \sum_{v \in \V^\si} R_v \h{\phi}_v \,, \q \forall R \in \ol{{\rm X}}_\si\,.
\end{align*}
Noting $\big(\h{\rr}^*_{\ol{{\rm X}}_\si} \Phi\big)_v  = \Phi(v)$ in \eqref{def:sampling}, the equivalence between $\rr^l_{\ol{{\rm X}}_\si}$ and $\h{\rr}_{\ol{{\rm X}}_\si}$ follows from a Taylor expansion:
\begin{equation} \label{auxerr_3}
\big(\big(\rr^l_{\ol{{\rm X}}_\si}\big)^* \Phi \big)_v =  \frac{1}{|T_v|} \int_{\Omega} \Phi \, \h{\phi}_v\, \rd x = \big(\h{\rr}^*_{\ol{{\rm X}}_\si} \Phi\big)_v  +  O(\norm{\na\Phi}_{\infty} \sigma) \q \text{for}\ \Phi \in C^1(\Omega,\MM^n)\,.
\end{equation}
It is also clear that $ \rr^l_{\ol{{\rm X}}_\si}(R)$ has $H^1$-regularity, and we have the norm estimate:
\begin{equation} \label{auxeq:tvest}
\mb{\norm{\rr^l_{\ol{{\rm X}}_\si} (R)}_{\rm TV}} \le  \sum_{v \in \V^\si} \norm{R_v}_{\rm F}\,|T_v| = \norm{R}_{1,\V^\si}\,,
\end{equation}
where again we identify \mb{$\rr^l_{\ol{{\rm X}}_\si} (R)$} with a measure. 

\medskip 
\noindent \textbf{Discrete derivation.} 
We next introduce the discrete operator $\ms{D}^*_\si: {\rm X}_\si \to {\rm Y}_\si$ by 
\begin{align} \label{defsdd}
        \sd^* G  = \h{\rr}^*_{{\rm Y}_\si}\ms{D}^* \rr_{\ol{{\rm X}}_\si}^l (G)  \q \text{for}\ G \in {\rm X}_\si \,,
\end{align} 
noting that $\ms{D}^* \rr_{\ol{{\rm X}}_\si}^l (G)$ is well-defined by the $H^1$-regularity of $\rr^l_{\ol{{\rm X}}_\si}(G)$. The desired discrete derivation $\ms{D}_\si$ is defined via the duality \eqref{def:adj_disderiv} with $\sd^*$ in \eqref{defsdd}: for $G \in {\rm X}_\si$, $q \in {\rm Y}_\si$, 
\begin{align} \label{def:dis_deriv}
     \l G,  \sd q \r_{\V^\si} := - \l \sd^* G,  q \r_{\TT^\si} = -   \sum_{K \in \TT^\si} \sum_{v \in \V_K^\si} |K| q_K \dd  \fint_K  \ms{D}^* \big(G_v \h{\phi}_v \big)\, \rd x \,.
\end{align}

\medskip 
\noindent \textbf{Discrete action.} 
We finally introduce the discrete functional $\jj_\si$ by,  for $\mu^\si = (G^\si, q^\si, R^\si) \in {\rm X}_\si \t {\rm Y}_\si \t \ol{{\rm X}}_\si$,
 \begin{align} \label{eq:discacf_1}
        \jj_\si(\mu^\si) :=  \frac{1}{2} \sum_{K \in \TT^\si} |K| q^\si_K \dd  \Big(\frac{1}{d+1}\sum_{v \in \V_K^\si} G_v^\si \Big)^\dag q^\si_K + \frac{1}{2} \sum_{v \in \V^\si} |T_v| R^\si_v \dd (G^\si_v)^{\dag} R^\si_v\,,
\end{align} 
which is proper lower semicontinuous and convex by \cite[Proposition 3.1]{li2020general1}. 

To finish the construction of the numerical scheme, we define the sampling operators in Definition \ref{def:cons_sampl} as the adjoint of the reconstructions: for $\Phi \in C(\Omega,\MM^n)$ and $\Psi \in C(\Omega,\R^{n \t k})$,
\begin{equation} \label{def:sampling}
 \ss_{\ol{{\rm X}}_\si} \Phi := (\rr_{\ol{{\rm X}}_\si}^l)^* \Phi = \Big(\frac{1}{|T_v|}\int_\Omega \Phi\,\h{\phi}_v\, \rd x \Big)_{v \in \V^\si}\,, \q \ss_{{\rm Y}_\si} \Psi := \rr_{{\rm Y}_\si}^* \Psi = \Big(\fint_K \Psi \, \rd x\Big)_{K \in \TT^\si}\,.
\end{equation}
The main reason we define $\ss_{\ol{{\rm X}}_\si}$ as $(\rr_{\ol{{\rm X}}_\si}^l)^*$ instead of $\rr_{\ol{{\rm X}}_\si}^*$ is that $(\rr_{\ol{{\rm X}}_\si}^l)^*$ is also well-defined for a measure \mb{(not only for a continuous function)}. We hence define the discrete initial and final distributions by $\ms{G}_0^\si = \ss_{\ol{{\rm X}}_\si} \ms{G}_0$ and $\ms{G}_1^\si = \ss_{\ol{{\rm X}}_\si} \ms{G}_1$, respectively. 
Then, a fully-discrete scheme for \eqref{eq:distance} is readily given by \eqref{prob:discre_ot} in Definition \ref{def:disc_transp}. 
The main result of this section is as follows. 

\begin{theorem} \label{thmconcrete}
Suppose that the assumptions \eqref{1}--\eqref{3} hold. The discrete model \eqref{prob:discre_ot} constructed above admits a minimizer $\mu^\ts_*$, and there exists a minimizer $\mu_*$ to \eqref{eq:distance} such that
    for any $ \rr_\ts \in \rr_\ts^{\rm ad}$ in \eqref{def:addtsr}, up to a subsequence, $\rr_\ts(\mu_*^\ts)$ 
     weak* converges to $\mu_*$ with
    \begin{align*}
        \lim_{(\ts) \to 0} \jj_{\tau,\sigma}(\mu^\ts_*) =  \jj_Q(\mu_*)\,.
    \end{align*} 
\end{theorem}    

We remark that the assumptions \eqref{1} and \eqref{2} are included in Theorem \ref{thm:upper_abs}, while \eqref{3} is needed for the verification of \eqref{asm:end_Point}. 
The proof of Theorem \ref{thmconcrete} follows from Theorem \ref{thm:upper_abs} and the following two lemmas. 

\begin{lemma}\label{lemconsistverfy}
The spatial reconstructions $\big[\h{\rr}_{\ol{{\rm X}}_\si}\big]$ and $\big[\h{\rr}_{{\rm Y}_\si}\big]$ with \eqref{eq:repxr}, defined as in Definition \ref{def:build_block},  are consistent
in the sense of Definition \ref{def:cons_recons}. In addition, the sampling operators $\ss_{\ol{{\rm X}}_\si}$ and $\ss_{{\rm Y}_\si}$ defined in \eqref{def:sampling} are consistent in the sense of Definition \ref{def:cons_sampl}. 
\end{lemma}

\begin{lemma}\label{lemm:con_total}
For any 
$\ms{G}^\si_{0}, \ms{G}^\si_1 \in {\rm X}_{\si,+}$, there exists a discrete curve $\mu^\ts \in \ce_\ts\big([0,1];\ms{G}^\si_0,\ms{G}^\si_1\big)$ satisfying 
\begin{align}  \label{discretepriori}
    \jj_\ts(\mu^\ts) \le 2  \sum_{v \in \V^\sigma} |T_v| \,\Big\lVert\sqrt{\ms{G}^\si_{1,v}} - \sqrt{\ms{G}^\si_{0,v}}\Big \lVert_{\rm F}^2 \,.
\end{align}
It follows that \eqref{prob:discre_ot} admits a minimizer $\mu^\ts_*$ for each $(\ts) \in \Sigma_{\NN}$. Moreover, 
for $\ms{G}_0, \ms{G}_1 \in \mc{M}(\Omega,\S^n_+)$, we have that 
$\ms{G}_0^\si = \ss_{\ol{{\rm X}}_\si} \ms{G}_0$ and $\ms{G}_1^\si = \ss_{\ol{{\rm X}}_\si} \ms{G}_1$ satisfy \eqref{eq:convdis_inifal}, and the assumption \eqref{asm:end_Point} holds if \eqref{3} holds. 
\end{lemma}

\begin{proof}[Proof of Lemma \ref{lemconsistverfy}]
We start with the consistency \eqref{eq:conserr_adjD} with 
quasi-isometric reconstructions $\h{\rr}_{\ol{\rm X}_\si}$ and $\h{\rr}_{{\rm Y}_\si}$. By \eqref{eq:repxr} and \eqref{def:dis_deriv}, we have, for $q^\si \in {\rm Y}_\si$ and $\Phi \in C^2(\Omega,\S^n)$, 
\begin{align} \label{auxerr_1} 
    &\l \h{\rr}_{{\rm Y}_\si} (q^\si),  \ms{D}^* \Phi \r_{\Omega} + \l \h{\rr}_{\ol{{\rm X}}_\sigma} \left(\sd q^\si \right), \Phi \r_{\Omega} \notag \\ = & \sum_{K \in \TT^\si}|K| q^\si_K \dd \fint_K \ms{D}^* \Phi\, \rd x - \sum_{K \in \TT^\si}\sum_{v \in \V_K^\si} |K| q_K^\si \dd \fint_K \ms{D}^* \big(\Phi(v) \h{\phi}_v \big)\, \rd x\,.
\end{align}
Recall $\ms{D}^* = \ms{D}^*_0 + \ms{D}^*_1$ with
 $\ms{D}^*_0 \in \L(\S^n,\R^{n \t k})$ and $\ms{D}^*$ being a homogeneous first order differential operator. Since $\Phi$ is $C^2$-smooth and $\sum_{v \in \V_K^\si} \h{\phi}_v|_K = 1$ on $K$, a simple calculation by Taylor expansion gives
\begin{align*}
  \Big \lVert\fint_K \ms{D}_0^*\, \Phi\, \rd x - \fint_K \sum_{v \in \V_K^\si}   \ms{D}_0^* \big(\Phi(v) \h{\phi}_v \big)\, \rd x \Big  \lVert_\ff \lesssim  \Big  \lVert \fint_K \Phi\, \rd x - \fint_K \sum_{v \in \V_K^\si}   \big(\Phi(v) \h{\phi}_v \big)\, \rd x \Big  \lVert_\ff \lesssim \norm{\na \Phi}_\infty \si\,.
\end{align*}
On the other hand,
we estimate, by the approximation property of linear elements \cite[Theorem 4.4.20]{brenner2008mathematical},
\begin{align*}
     \Big \lVert \fint_K \ms{D}_1^*\, \Phi\, \rd x -  \fint_K \sum_{v \in \V_K^\si} \ms{D}_1^* \big(\Phi(v) \h{\phi}_v \big)\, \rd x \Big \lVert_\ff \lesssim \norm{\Phi}_{2,\infty} \si\,.
\end{align*}
The desired \eqref{eq:conserr_adjD} follows from \eqref{auxerr_1} and above estimates. 

Next, we consider the consistency \eqref{eq:conserr_D}. Recalling the sampling operator \eqref{def:sampling} and the discrete operator \eqref{def:dis_deriv}, we find, for $A^\si \in {\rm X}_\si$ and $\vp \in C^2(\Omega,\R^{n \t k})$ with $\widehat{\ms{D}_1}(-i \nu)\,\vp = 0$, 
\begin{align} \label{auxerr_2}
  \l A^\si, - \sd \ss_{{\rm Y}_\sigma} \vp + \ss_{\ol{{\rm X}}_\si} \mathsf{D} \vp \r_{\V^\si} 
  & = \big\l \h{\rr}_{\ol{\rm X}_\si}^l A^\si, \mathsf{D}\vp \big\r_{\Omega} + \l \sd ^* A^\si, \h{\rr}^*_{{\rm Y}_\sigma} \vp \big\r_{\TT^\si}\,.
\end{align}
 Noting that $\vp$ satisfies the boundary condition $\widehat{\ms{D}_1}(-i \nu)\vp = 0$, we can  calculate  
\begin{align} \label{auxerr_6}
  \big\l \h{\rr}_{\ol{\rm X}_\si}^l A^\si, \mathsf{D}\vp \big\r_{\Omega} + \l \sd ^* A^\si, \h{\rr}^*_{{\rm Y}_\sigma} \vp \big\r_{\TT^\si} = - \big\l \mathsf{D}^* \rr^l_{\ol{\rm X}_\si} (A^\si), \vp \big\r_{\Omega} + \big\l \h{\rr}_{{\rm Y}_\si}^* \ms{D}^* \rr^l_{\ol{\rm X}_\si} (A^\si), \h{\rr}^*_{{\rm Y}_\sigma}\vp \big\r_{\V_\si}\,,
\end{align}
by the integration by parts and \eqref{def:dis_deriv}. By decomposing $\ms{D}^* = \ms{D}^*_0 + \ms{D}^*_1 $ as above, we have 
\begin{equation} \label{auxerr_4}
    - \big\l \mathsf{D}_1^*\, \rr^l_{\ol{\rm X}_\si} (A^\si), \vp \big\r_{\Omega} + \big\l \h{\rr}_{{\rm Y}_\si}^* \ms{D}_1^* \,\rr^l_{\ol{\rm X}_\si} (A^\si), \h{\rr}^*_{{\rm Y}_\sigma}\vp \big\r_{\V_\si}  = 0\,, 
\end{equation}
since  $\ms{D}^*_1\rr_{\ol{\rm X}_\si}^l (A^\si)$ is a piecewise constant function.  Meanwhile, by \eqref{auxeq:tvest}, a direct estimate yields 
\begin{align} \label{auxerr_5}
   \, & \big|- \big\l\, \mathsf{D}_0^*\, \rr^l_{\ol{\rm X}_\si} (A^\si), \vp \big\r_{\Omega} + \big\l \h{\rr}_{{\rm Y}_\si}^* \ms{D}_0^* \, \rr^l_{\ol{\rm X}_\si} (A^\si), \h{\rr}^*_{{\rm Y}_\sigma} \vp  \big\r_{\V_\si} \big| \notag \\
    \le & \sum_{K \in \TT^\si} \int_K    \Big| \rr^l_{\ol{\rm X}_\si} (A^\si) \dd \ms{D}_0 \Big(\vp - \fint_K \vp \,\rd x\Big) \Big|\, \rd x \notag\\
    \lesssim &\  \si  \norm{\na \vp}_{\infty}  \big \lVert \rr^l_{\ol{\rm X}_\si} A^\si \big\lVert_{\rm TV}
     \lesssim \si \norm{\na \vp}_{\infty}  \norm{A^\si}_{1,\V^\si}\,. 
\end{align}
Combining the estimates \eqref{auxerr_4} and \eqref{auxerr_5} with \eqref{auxerr_2} and \eqref{auxerr_6}, then the condition \eqref{eq:conserr_D} follows.

We now consider the estimates  \eqref{eq:bdd_adjoint} and \eqref{est:conjaction} for the conjugate functional $\jj_\si^*$ with 
 $\h{\rr}_{\ol{\rm X}_\si}$ and $\h{\rr}_{{\rm Y}_\si}$. For this, we first compute $\jj_\si^*$ by definitions \eqref{def:conjugate} and  \eqref{eq:discacf_1}: for $(G^\si,u^\si, W^\si) \in {\rm X}_\si \t {\rm Y}_\si \t \ol{{\rm X}}_\si$,  
\begin{align*}
      \jj_\si^*(G^\si, u^\si, W^\si) = \frac{1}{2} \sum_{K \in \TT^\si} |K| u^\si_K \dd  \Big(\frac{1}{d+1} \sum_{v \in \V_K^\si} G_v^\si \Big)\, u^\si_K + \frac{1}{2} \sum_{v \in \V^\si} |T_v| W^\si_v \dd G^\si_v W^\si_v \,.
\end{align*}
To estimate, for $G^\si \in {\rm X}_\si$ and $(u,W) \in C(\Omega, \R^{n \t k} \t \mathbb{M}^n)$, 
 \begin{align}
     \jj^*_\si \big(G^\si,\h{\rr}^*_{{\rm Y}_\si} u, \h{\rr}^*_{\ol{{\rm X}}_\si} W \big) & = \frac{1}{2} \sum_{K \in \TT^\si} |K| \fint_K u\, \rd x \dd  \Big(\frac{1}{d+1}\sum_{v \in \V_K^\si} G_v^\si \Big) \fint_K u\, \rd x + \frac{1}{2} \sum_{v \in \V^\si} |T_v| W(v) \dd G_v^\si W(v)\,, \label{eq:jjgyx}
 \end{align}
we note from \eqref{eq:userela_tri} that for some $C > 0$,
 \begin{align*}
    &  \sum_{K \in \TT^\si} |K|\Big( \frac{1}{d+1} \sum_{v \in \V_K^\si} G_v^\si \Big)\dd \Big(\fint_K u\, \rd x \fint_K u^{{\rm T}}\, \rd x \Big)\\
=  & \sum_{v \in \V^\si} \frac{1}{d+1} \sum_{K \in \TT_v^\si} |K|  G_v^\si \dd  \Big(\fint_K u\, \rd x \fint_K u^{{\rm T}}\, \rd x \Big) \\
\le  &  \sum_{v \in \V^\si}  |T_v| G_v^\si \dd \big( u(v)\,u(v)^{{\rm T}}\big) + C \norm{u}^2_{1,\infty} \norm{G^\si}_{1,\V_\si} \si \,,\q \forall u \in C^1(\Omega, \R^{n \t k})\,.
\end{align*}
Then, recalling \eqref{def:jalpha_conj}, it readily follows that, for $C^1$-smooth $(u,W)$,
 \begin{align*}
  \jj^*_\si \Big(G^\si,\h{\rr}^*_{{\rm Y}_\si}u,\h{\rr}^*_{\ol{{\rm X}}_\si} W \Big)  & \le  \frac{1}{2} \sum_{v \in \V^\si} |T_v| G_v^\si \dd \big(u(v)\, u(v)^{\rm T} + W(v)\, W(v)^{\rm T}\big)  + C \norm{u}^2_{1,\infty} \norm{G^\si}_{1,\V_\si} \si \\
  & = \jj_{\Omega}^*\big(\h{\rr}_{\ol{\rm X}_\si}G^\si,u,W\big)  + C \norm{u}^2_{1,\infty} \norm{G^\si}_{1,\V_\si} \si\,,
 \end{align*}
that is, \eqref{est:conjaction} holds. The estimate \eqref{eq:bdd_adjoint} is a simple consequence of the expression \eqref{eq:jjgyx}. 

We finally verify the condition \eqref{eq:consis_energy}. For notational simplicity, we write $\w{\Phi}_v = \ss_{\ol{{\rm X}}_\si}\Phi$ for a continuous function $\Phi$. Then, for smooth 
$\mu = (G,q,R)$ with $G \succeq \eta I$ for some $\eta > 0$, we compute 
\begin{align} \label{auxerr_7}
    &\jj_\si \big(\ss_{\ol{{\rm X}}_\si}(G), \ss_{{\rm Y}_\si}(q), \ss_{\ol{\rm X}_\si}(R) + {\rm e}_{\si,q}\big) \notag\\
=  & \frac{1}{2} \sum_{K \in \TT^\si} |K| \fint_K q\, \rd x \dd  \Big(\frac{1}{d+1}\sum_{v \in \V_K^\si} \w{G}_v \Big)^\dag \fint_K q\, \rd x + \frac{1}{2} \sum_{v \in \V^\si} |T_v| \big(\w{R}_v + {\rm e}_{\si,q}\big) \dd \w{G}_v^{\dag} \big(\w{R}_v + {\rm e}_{\si,q}\big) \notag\\
=  & \frac{1}{2} \sum_{K \in \TT^\si} |K| \Big(\frac{1}{d+1}\sum_{v \in \V_K^\si} \w{G}_v \Big)^\dag \dd \big(q(v)\,q(v)^{{\rm T}} + O(\norm{q}^2_{1,\infty}\si)\big) \notag \\& + \frac{1}{2} \sum_{v \in \V^\si} |T_v| \big(\w{R}_v + {\rm e}_{\si,q}\big) \dd \w{G}_v^{\dag} \big(\w{R}_v + {\rm e}_{\si,q}\big)\,, 
\end{align}
by using $\fint_K q\, \rd x = q(v) + O(\norm{q}_{1,\infty} \si)$. Noting $G \succeq \eta I$ on $Q$ and $G^\dag = G^{-1}$, we have, for small enough $\si > 0$,
\begin{align*}
    \Big(\frac{1}{d+1}\sum_{v \in \V_K^\si} \w{G}_v \Big)^{-1} = \frac{1}{d+1} \sum_{v \in \V_K^\si} G(v)^{-1}  + O(\eta^{-2}\norm{G}_{1,\infty} \si)\,.
\end{align*}
Then, it follows from \eqref{eq:userela_tri} that, for some $C > 0$,
\begin{align}\label{auxerr_8}
    & \frac{1}{2} \sum_{K \in \TT^\si} |K| \Big(\frac{1}{d+1}\sum_{v \in \V_K^\si} \w{G}_v \Big)^\dag \dd \big(q(v)\,q(v)^{{\rm T}} + O(\norm{q}^2_{1,\infty}\si)\big) 
    \notag \\
\le \, & \frac{1}{2} \sum_{v \in \V^\si} |T_v| q(v) \dd G(v)^{-1} q(v) + C \eta^{-2}(\norm{G}_{1,\infty} + 1)\norm{q}_{1,\infty}^2 \si \,.
\end{align}
Similarly, by definition of ${\rm e}_{\si,q}$ and \eqref{eq:conserr_D} verified above, we have,  for small enough $\si > 0$,
\begin{align} \label{auxerr_9}
    \frac{1}{2} \sum_{v \in \V^\si} |T_v| \big(\w{R}_v + {\rm e}_{\si,q}\big) \dd \w{G}_v^{\dag} \big(\w{R}_v + {\rm e}_{\si,q}\big) \le \frac{1}{2} \sum_{v \in \V^\si} |T_v| R(v) \dd G(v)^{\dag} R(v) + C_{\norm{\mu}_{2,\infty},\eta}\, \si \,,
\end{align}
where the generic constant $C_{\norm{\mu}_{2,\infty},\eta}$ depends on $\norm{\mu}_{2,\infty}$ and $\eta$. Then, by
 \eqref{auxerr_8} and \eqref{auxerr_9}, \eqref{auxerr_7} implies 
\begin{align}\label{est:auxconen}
    \jj_\si\big(\ss_{\ol{{\rm X}}_\si}(G), \ss_{{\rm Y}_\si}(q), \ss_{\ol{\rm X}_\si}(R) + {\rm e}_{\si,\mu}\big) \le \frac{1}{2} \sum_{v \in \V^\si} |T_v|  G(v)^{-1}\dd \left(q(v)\,q(v)^{{\rm T}} + R(v)\,R(v)^{{\rm T}}\right) + C_{\norm{\mu}_{2,\infty},\eta}\, \si\,.
\end{align}
Since the term $G(x)^{-1}\dd \left(q(x)\,q(x)^{{\rm T}} + R(x)\,R(x)^{{\rm T}}\right)$ is  $C^{1}$-smooth on $\Omega$ with uniformly positive definite $G \succeq \eta I$, its gradient can be bounded by a constant depending on $\norm{\mu}_{1,\infty}$ and $\eta$. Thus, we can conclude  \eqref{eq:consis_energy} by \eqref{est:auxconen}. 
\end{proof}

\begin{proof}[Proof of Lemma \ref{lemm:con_total}]
The estimate \eqref{discretepriori} is a discrete analog of \eqref{eq:basic_bound}. 
We define  $(G^\ts_{i,v})_{i,v} \in {\rm X}_\ts$ by
\begin{equation} \label{eq:giv}
    G^\ts_{i,v} := \left(\sqrt{\ms{G}^\si_{0,v}} + i \tau \Big(\sqrt {\ms{G}^\si_{1,v}} - \sqrt{\ms{G}^\si_{0,v}} \Big) \right)^2 \,, \q i = 0,1,\ldots, N\,,
\end{equation}
and  
\begin{equation} \label{eq:rriv}
     R^\ts_{i,v} := \Big(( 2 i - 1 ) \tau \Big(\sqrt {\ms{G}^\si_{1,v}} - \sqrt {\ms{G}^\si_{0,v}}\Big) + 2 \sqrt {\ms{G}^\si_{0,v}} \Big) \Big(\sqrt {\ms{G}^\si_{1,v}}-\sqrt {\ms{G}^\si_{0,v}}\Big)\,, \q i = 1,\ldots, N\,,
\end{equation}
with $\ms{G}_0^\si = (\ms{G}^\si_{0,v})_{v \in \V^\si} \in {\rm X}_{\si,+}$ and $ \ms{G}_1^\si = (\ms{G}^\si_{1,v})_{v \in \V^\si} \in {\rm X}_{\si,+}$ being given. It is easy to check that $\mu^\ts := (G^\ts,0,R^\ts)$ satisfies the discrete continuity equation \eqref{eq:discontieq}, by noting, \mb{for $A,B \in \mathbb{S}^n$,}
\begin{align*}
     (A + i \tau (B - A))^2 - (A + (i-1) \tau (B - A))^2 & = (2 i - 1) \tau^2 (B-A)^2 + \tau \big( A(B-A) + (B-A)A\big) \\
     & = \mb{\tau \{((2i-1)\tau(B - A) + 2 A)(B - A)\}^{\rm sym}\,.}
\end{align*}
We next estimate the discrete action functional defined by \eqref{eq:discost} and \eqref{eq:discacf_1}:
\begin{align}  \label{est:auxdef_11}
    \jj_{\ts}(\mu^\ts) = \frac{\tau}{2}  \sum_{i = 1}^N  \sum_{v \in \V^\sigma} |T_v| R^\ts_{i,v} \dd \Big(\frac{G^\ts_{i,v} + G^\ts_{i-1,v}}{2}\Big)^{\dag} R^\ts_{i,v}\,.
\end{align}
We first compute
\begin{align}  \label{est:auxest_11}
 \frac{1}{2}\big(G^\ts_{i,v} + G^\ts_{i-1,v}\big)
   = & \Big(\Big(i - \frac{1}{2}\Big)\tau \Big(\sqrt{\ms{G}^\si_{1,v}} - \sqrt{\ms{G}^\si_{0,v}}\Big) + \sqrt{\ms{G}^\si_{0,v}}\Big)^2 + \frac{1}{4} \tau^2 \Big(\sqrt{\ms{G}^\si_{1,v}} - \sqrt{\ms{G}^\si_{0,v}}\Big)^2\,,
\end{align}
by definition \eqref{eq:giv} and \mb{the simple identity, for $A,B \in \mathbb{M}^n$,}
\begin{equation*}
   \mb{ A^2 + B^2 = \frac{(A-B)^2}{2} + \frac{(A + B)^2}{2}\,.}
\end{equation*}
\mb{Note from \eqref{est:auxest_11} that for $i = 1,\ldots,N$, there hold
\begin{align} \label{auxeqq1}
   \ran\Big(\Big(\sqrt{\ms{G}^\si_{1,v}} - \sqrt{\ms{G}^\si_{0,v}}\Big)^2\Big) \subset \ran \left( G^\ts_{i,v} + G^\ts_{i-1,v}\right)\,,
\end{align}
and }
\begin{align} \label{auxeqq2}
    \mb{\ran \left(G^\ts_{i,v} + G^\ts_{i-1,v}\right) = \ran\Big(\Big(i - \frac{1}{2}\Big)\tau \Big(\sqrt{\ms{G}^\si_{1,v}} - \sqrt{\ms{G}^\si_{0,v}}\Big) + \sqrt{\ms{G}^\si_{0,v}}\Big) = \left(\ker\left(\ms{G}^\si_{0,v} \right)\bigcap \ker\left(\ms{G}^\si_{1,v} \right)  \right)^\perp\,.}
\end{align}
Then, by
\eqref{est:auxdef_11} and \eqref{est:auxest_11}, as well as the definition \eqref{eq:rriv} of $R^\ts_{i,v}$, we have the following estimates
\begin{align} \label{auxqqq3}
    \jj_\ts(\mu^\ts)
        &\le \frac{\tau}{2}\sum_{i = 1}^N  \sum_{v \in \V^\sigma} |T_v|\, R^\ts_{i,v} \dd \Big(\Big(\Big(i - \frac{1}{2}\Big)\tau \Big(\sqrt{\ms{G}^\si_{1,v}} - \sqrt{\ms{G}^\si_{0,v}}\Big) + \sqrt{\ms{G}^\si_{0,v}}\Big)^2\Big)^{\dag}  \, R^\ts_{i,v}  \notag\\
        & \le 2 \tau \sum_{i = 1}^N  \sum_{v \in \V^\sigma} |T_v| \,\Big(\sqrt{\ms{G}^\si_{1,v}} - \sqrt{\ms{G}^\si_{0,v}}\Big)\dd \Big(\sqrt{\ms{G}^\si_{1,v}} - \sqrt{\ms{G}^\si_{0,v}}\Big) \notag\\
        & = 2  \sum_{v \in \V^\sigma} |T_v|\, \left\lVert\sqrt{\ms{G}^\si_{1,v}} - \sqrt{\ms{G}^\si_{0,v}}\right\lVert_{\rm F}^2 \,,
\end{align}
\mb{where the second inequality follows from the relations \eqref{auxeqq1} and \eqref{auxeqq2} and the fact that $A^\dag A$ for $A \in \S^n$ is the projection to the range of $A$.} Then, the existence of the minimizer to \eqref{prob:discre_ot} follows from Remark \ref{rem:existence}. 

The weak* convergence \eqref{eq:convdis_inifal} simply follows from the uniform continuity of $\Phi$: 
\begin{align*}
  \big\l \h{\rr}_{{\rm X}_\sigma}\ms{G}^\si_i,  \Phi \big\r_{\Omega}  =  \big\l \ms{G}_i,  
\rr_{{\rm X}_\sigma}^l \h{\rr}_{{\rm X}_\sigma}^* \Phi \big\r_{\Omega} \to \l \mathsf{G}_i,  \Phi \r_{\Omega}\,,\q \text{as}\ \si \to 0\,.
\end{align*}
We finally check the assumption \eqref{asm:end_Point}. We only consider the case $i = 0$. \mb{By the estimate \eqref{auxqqq3} with $\ms{G}_0^\si = \ss_{\ol{\rm X}_\si}\ms{G}_0$ and $\ms{G}^\si_1 = \ss_{\ol{{\rm X}}_\si}G_0^\ep $,  where $G_0^\ep$ is the $\ep$-regularization of $\ms{G}_0$, 
there exists $\mu^\ts \in \ce_\ts([0,1];\ms{G}_0^\si, \ms{G}^\si_1)$ such that 
\begin{align*}
     \jj_\ts(\mu^\ts) & \overset{\eqref{est:lip_squre}}{\lesssim}  \sum_{v \in \V^\sigma} |T_v|\, \left\lVert\ms{G}^\si_{1,v} - \ms{G}^\si_{0,v}\right\lVert_{\rm F}\\
    & \ \lesssim\ \, \sum_{v \in \V^\sigma}  \int_\Omega  \h{\phi}_v \norm{G_0 - G_0^\ep}_\ff \, \rd x \lesssim  \norm{G_0 - G_0^\ep}_{\rm TV}\,, 
\end{align*}
where the second inequality is by the definition \eqref{def:sampling} of sampling operators with the absolute continuity of $\ms{G}_i = G_i \, \rd x$.} This further yields \eqref{assumpest}: $\limsup_{(\ts) \to 0}\jj_\ts(\mu^\ts) \to 0$
as $\ep \to 0$ by 
Proposition \ref{prop:regular_ep}. The proof is complete. 
\end{proof}

\section{Convergence of discrete WFR metric}
\label{sec:specwfr}

In this section, we consider the numerical analysis of the  Wasserstein–Fisher–Rao metric \eqref{def:wfr_metric} with $\alpha = \beta = 1$. 
We adopt the discrete scheme in Section \ref{sec:conver_example}. 
It is clear that Theorem \ref{thmconcrete} directly applies to this case, where the assumption \eqref{2} holds automatically. We will show that the assumption \eqref{3} can also be removed by a more subtle analysis based on the static formulation. In particular, we will prove the following result. 

\begin{theorem} \label{lem:reg_points}
For a given $\rho \in \M(\Omega,\R_+)$, let $\rho^\ep$ be its $\ep$-regularization in Proposition \ref{prop:regular_ep}. Then, it holds that   
\begin{align} \label{eq:endpoints}
    {\rm WFR}^2(\rho,\rho^\ep) = O(\ep)\,,\q \text{as} \ \ep \to 0\,. 
\end{align}
Moreover, there exists $\mu^\ts = (\rho^\ts,q^\ts,r^\ts)  \in \ce_\ts([0,1];\ss_{\ol{{\rm X}}_\si}(\rho),\ss_{\ol{{\rm X}}_\si}(\rho^\ep))$ for $(\ts)\in \Sigma_{\NN}$ such that
\begin{align} \label{eq:regendpoints}
    \jj_\ts(\mu^\ts) = O(\ep + \si)\,,\q \text{as} \ \ep, \si \to 0\,. 
\end{align}
\end{theorem}
Note that \eqref{eq:endpoints} is an improvement of the third statement in Proposition \ref{prop:regular_ep}, while the estimate
\eqref{eq:regendpoints} verifies the assumption \eqref{asm:end_Point}. The rest of this section is devoted to the proof of Theorem \ref{lem:reg_points}. 

\medskip

\noindent \textbf{Static formulation of WFR metric.}  
For the proof of \eqref{eq:endpoints}, we recall the 
static formulation of \eqref{def:wfr_metric} with $\alpha = \beta = 1$ from 
\cite[Theorem 5.6]{chizat2018unbalanced}. We define the truncated cosine $\cos_{b}(z): = \cos(\min\{|z|,b\})$ with $b \in \R$. Then, for $m_0, m_1 \ge 0$ and $x, y \in \Omega$, we introduce the cost function:
\begin{equation} \label{costwfr}
    c(x, m_0, y, m_1) := {\rm WFR}^2(m_0 \d_{x}, m_1 \d_{y}) = 2\big(m_0 + m_1 - 2 \sqrt{m_0 m _1}\,  \cos_{\frac{\pi}{2}}(|x-y|/2)\big)\,.
\end{equation}
We also need the so-called semi-couplings of two marginals $\rho_0,\rho_1 \in \M(\Omega,\R_+)$: 
\begin{equation*}
    \Gamma(\rho_0,\rho_1) := \big\{ (\gamma_0,\gamma_1) \in \M(\Omega^2, \R_+)^2\,;\ (\pi^1_{\#} \gamma_0, \pi^2_{\#}\gamma_1) = (\rho_0, \rho_1) \big\}\,,
\end{equation*}
where $\pi^1:(x,y) \to x$ and $\pi^2:(x,y) \to y$ are projections from $\Omega^2 = \Omega \times \Omega$ to its first and second factors, respectively. Then,
it holds that the WFR metric can be formulated as
\begin{align} \label{def:wfr_metric_static_2}
    {\rm WFR}^2(\rho_0,\rho_1) & = \inf \Big\{\int_{\Omega^2} c \Big(x,\frac{\rd \gamma_0}{\rd \gamma},y,\frac{\rd \gamma_1}{\rd \gamma}\Big)\, \rd \gamma\,;\   (\gamma_0,\gamma_1) \in \Gamma(\rho_0, \rho_1) \Big\}\,,\notag  \\
    & = 2(\rho_0(\Omega) + \rho_1(\Omega)) - 4 \sup\Big\{\int_{|x-y| < \pi} \cos(|x-y|/2) \, \rd \sqrt{\gamma_0 \gamma_1}\,;\ (\gamma_0,\gamma_1) \in \Gamma(\rho_0, \rho_1) \Big\}\,,
\end{align}
where $\gamma$ is a reference measure such that $\gamma_0, \gamma_1 \ll \gamma$, and the measure $\rd \sqrt{\gamma_0\gamma_1}$ is defined by $\sqrt{\frac{\rd \gamma_0}{\rd \gamma}\frac{\rd \gamma_1}{\rd \gamma}}\,\rd \gamma$ which is independent of the choice of $\gamma$. 

\medskip

\noindent \textbf{Proof of \eqref{eq:endpoints}.} 
We recall the $\ep$-regularization \eqref{def:epregend_points} and have 
\begin{align} \label{reveq:add1}
    \rho^\ep = (1 - \ep)\h{\rho}^\ep + \ep \ol{\rho}^\ep \q \text{with}\q \h{\rho}^\ep =  T^\ep_{\#}(\theta^\ep_d * \rho)\q \text{and}\q \ol{\rho}^\ep = T_{\#}^\ep((\theta_d^\ep* \ol{\rho})|_{\Omega_\ep})\,,
\end{align}
where \mb{the $\ep$-neighborhood $\Omega_\ep$, the convolution kernel $\theta^\ep_d$, and the space scaling function $T^\ep = (1 + L \ep)^{-1}$ are defined in \eqref{auxdef1}, \eqref{auxdef2}, and \eqref{aux_tsscaling}, respectively,} and $\ol{\rho}$ is a smooth density supported on $\Omega_1$ with $\ol{\rho} \ge 1$ on $\Omega_{1/2}$. To estimate ${\rm WFR}^2(\rho,\rho^\ep)$, we first note, by \eqref{sublienar},
\begin{align} \label{auxerr_13}
    {\rm WFR}^2(\rho,\rho^\ep) \le (1 - \ep) {\rm WFR}^2(\rho, \h{\rho}^\ep) + \ep {\rm WFR}^2(\rho, \ol{\rho}^\ep)\,,
\end{align}
where ${\rm WFR}^2(\rho, \ol{\rho}^\ep)$ can be bounded by a constant $C$ independent of $\ep$ by \eqref{eq:basic_bound}. We next estimate ${\rm WFR}^2(\rho, \h{\rho}^\ep)$, using the static formulation \eqref{def:wfr_metric_static_2}. By construction, we have 
\begin{align}  \label{equalmassauxeq} 
    \h{\rho}^\ep(\Omega)  = \int_{(T^\ep)^{-1}\Omega} \int_{\R^d} \theta_d(y - x)\,\rd \rho (x)\rd y = 
    \int_{\R^d} \int_{\R^d} \theta_d(y-x)\,\rd \rho (x)\rd y = \rho(\Omega)\,, 
\end{align}
 where we have viewed $\rho$ as a measure on $\R^d$ by zero extension. Then, we define the semi-coupling measures by
\begin{equation} \label{def:semi_coupling}
    \rd \gamma_0 (x,y)  = \rd \gamma_1 (x,y) =  \theta_d^\ep((T^\ep)^{-1} y - x)\, \rd\rho(x)\, \rd ( (T^\ep)^{-1} y)\,,
\end{equation}
and a reference measure $\rd \gamma(x,y) = \rd \rho(x) \rd y$, which satisfy $\pi^1_{\#} \gamma_0 = \rho$, $\pi_{\#}^2 \gamma_1 = \h{\rho}^\ep$, and $\gamma_0, \gamma_1 \ll \gamma$. It follows that 
$
   \rd \gamma_0 = \rd \gamma_1 = (1 + L \ep)^d \theta_d^\ep(x- (1 + L \ep) y)\, \rd \gamma$, and by \eqref{def:wfr_metric_static_2} and \eqref{equalmassauxeq}, there holds 
\begin{align} \label{estwft}
  \frac{1}{4}  {\rm WFR}^2(\rho,\h{\rho}^\ep) & \le \rho(\Omega) - \int_{|y - x| < \pi} \cos\left(|x-y|/2\right)  \theta_d^\ep(x- (1 + L \ep) y)\, \mb{\rd \rho(x)}\, \rd (1 + L \ep) y\,.
\end{align}
Recalling that the support of $\rho$ is $\Omega$, we can find, for some $C > 0$ depending on $\Omega$,  
\begin{equation} \label{supp_est}
\{ (x,y) \in \R^d \times \R^d\,;\    \theta_d^\ep(x- (1 + L \ep) y)\chi_{\Omega}(x) \neq 0\} \subset \{(x,y)\in \Omega \times \Omega \,;\  |x-y| \le C \ep\}\,.
\end{equation}
Indeed, let $\Omega^c$ (resp., $\Omega^c_\ep$) be the complement of $\Omega$ (resp., $\Omega_\ep$). 
By $\Omega_\ep \subset (T^\ep)^{-1} \Omega$, we have $(T^\ep)^{-1} \Omega^c \subset \Omega_\ep^c$ and hence $|x - (T^{\ep})^{-1} y| > \ep$ with $\theta_d^\ep(x- (T^{\ep})^{-1} y) = 0$ for $x \in \Omega$, $y \in \Omega^c$. Meanwhile, it is easy to see that $|x - (T^{\ep})^{-1} y| \le \ep$ implies $|x-y| \le \ep + L \ep |y| \le C \ep$ for $x,y \in \Omega$. Thanks to \eqref{supp_est}, we can further estimate, from \eqref{estwft}, 
\begin{align} \label{auxerr_14}
  \frac{1}{4}  {\rm WFR}^2(\rho,\h{\rho}^\ep) &  \le
     \rho(\Omega) -    \int_{|x-y| \le C \ep}  \cos\left(|x-y|/2\right)  \theta_d^\ep(x- (1 + L \ep) y)\, \mb{\rd \rho(x)} \, \rd (1 + L \ep) y \notag \\
     & \le \rho(\Omega) -  (1 - C \ep) \int_\Omega \int_{\R^d} \theta_d^\ep(x- (1 + L \ep) y)\, \rd (1 + L \ep) y\, \mb{\rd \rho(x)} \notag \\
    & = \rho(\Omega) -  (1 - C \ep) \rho(\Omega) =  C \ep \rho(\Omega)\,.
\end{align}
Combining \eqref{auxerr_13} and \eqref{auxerr_14}, we achieve the desired \eqref{eq:endpoints}.  

\medskip

\noindent \textbf{Geodesics between Dirac measures.} For the proof of \eqref{eq:regendpoints}, we recall the characterization of the geodesic $\rho(s)$ connecting $\rho_0 = m_0 \d_{x}$ and $\rho_1 = m_1 \d_{y}$ in the WFR metric space; see \cite[Theorem 4.1]{chizat2018interpolating} and also \cite{liero2016optimal}. Let $m_0, m_1 \ge 0$ and $x, y \in \Omega$. We consider the following three cases. 

\smallskip
\noindent \emph{Case I}: $|x - y| < \pi$. The geodesic is unique and of the form $\rho(s) = m(s) \d_{\gamma(s)}$ for $0 \le s \le 1$, where $m$ is a quadratic function given by
\begin{align} \label{eq_geo_1}
    m(s) = (1 - s)^2 m_0 + s^2 m_1 + 2 s(1 - s)\sqrt{m_0 m_1} \cos(|x - y|/2)\,, 
\end{align}
and $\gamma(s)$ is a straight line connecting $x$ and $y$ with 
\begin{align} \label{eq_geo_2}
    |\dot{\gamma}(s)| = \frac{|x - y| H(m)}{m(s)}\,, \q  H(m) = \Big(\int_0^1 \frac{1}{m(s)} \rd s \Big)^{-1}\,.
\end{align}

\noindent \emph{Case II}: $|x - y| > \pi$.  The unique geodesic is given by 
\begin{equation} \label{eq_geo_3}
\rho(s) = (1 - s)^2 m_0 \d_{x} + s^2 m_1 \d_y\,.    
\end{equation}

\noindent \emph{Case III}: $|x - y| = \pi$. We have infinitely many geodesics including linear combinations of those in the first two cases.

\medskip

\noindent \textbf{Proof of \eqref{eq:regendpoints}.} 
It suffices to prove that we can construct $\mu^\ts \in \ce_\ts([0,1]; \ss_{\ol{{\rm X}}_\si}(\rho),  \ss_{\ol{{\rm X}}_\si}(\h{\rho}^\ep))$ satisfying
\begin{equation}\label{aimwfr}
     \jj_\ts(\mu^\ts) = O(\ep + \si)\,,
\end{equation}
\mb{where $\h{\rho}^\ep$ is given in \eqref{reveq:add1}.}
Indeed, by Lemma \ref{lemm:con_total}, there exists $\w{\mu}^\ts \in \ce_\ts([0,1]; \ss_{\ol{{\rm X}}_\si}(\rho),  \ss_{\ol{{\rm X}}_\si}(\ol{\rho}^\ep))$ with $\jj_\ts(\w{\mu}^\ts)$ bounded by a constant independent of $\ts$. Then, 
$(1 - \ep)\mu^\ts + \ep \w{\mu}^\ts$ gives the desired curve connecting $\ss_{\ol{{\rm X}}_\si}(\rho)$ and $\ss_{\ol{{\rm X}}_\si}(\rho^\ep)$ with the estimate \eqref{eq:regendpoints}, by linearity of the discrete continuity equation
and sublinearity of $\jj_\ts(\dd)$. 

Recall the semi-couplings $\gamma_0$,  $\gamma_1$ in \eqref{def:semi_coupling} and the reference measure $\rd\gamma(x,y) = \rd \rho(x) \rd y$ from the proof of \eqref{eq:endpoints}. We first show that our aim \eqref{aimwfr} can be implied by the following claim: \mb{for $x,y \in \Omega$ with 
\begin{align} \label{auxerr_16}
|x- (1 + L \ep) y| \le \ep \q  \text{and}\q   m(x,y) := \frac{\rd \gamma_0}{\rd \gamma} = \frac{\rd \gamma_1}{\rd \gamma} = (1 + L \ep)^d\, \theta_d^\ep(x- (1 + L \ep) y)\,, 
\end{align}
 there exists $\mu^\ts_{(x,y)} \in \ce_\ts\big([0,1]; \ss_{\ol{{\rm X}}_\si}(m \d_x), \ss_{\ol{{\rm X}}_\si}(m \d_y)\big)$ 
such that 
\begin{align} \label{est:end_point_buildblk}
   \jj_\ts\big(\mu_{(x,y)}^\ts \big) \le C (c(x, m, y, m) + m \si) \,,
\end{align}
for some constant $C > 0$ independent of $(x,y, m)$.}

For this, we define 
\begin{align*}
    \mu^\ts: = \int_{\Omega^2} \mu^\ts_{(x,y)} \, \rd \gamma(x,y)\,,
\end{align*}
which, by linearity, satisfies the discrete continuity equation with endpoints given by, for $v \in \V^\si$, 
\begin{align*}
    \int_{\Omega^2} \ss_{\ol{{\rm X}}_\si}\Big(\frac{\rd \gamma_0}{\rd \gamma}(x,y) \d_x\Big)_v \, \rd \gamma (x,y) & = \frac{1}{|T_v|} \int_{\Omega^2} \frac{\rd \gamma_0}{\rd \gamma}(x,y) \h{\phi}_v(x)\,  \rd \gamma (x,y) \\
    & = \frac{1}{|T_v|} \int_{\Omega} \h{\phi}_v(x)\, \rd \rho (x) = \ss_{\ol{{\rm X}}_\si}(\rho)_v\,, 
\end{align*}
and 
\begin{align*}
     \int_{\Omega^2} \ss_{\ol{{\rm X}}_\si}\Big(\frac{\rd \gamma_1}{\rd \gamma}(x,y) \d_y\Big)_v\, \rd \gamma (x,y) = \ss_{\ol{{\rm X}}_\si}(\h{\rho}^\ep)_v \,.
\end{align*}
Then, by sublinearity of $\jj_\ts$ and Jensen's inequality, we have, from the claim \eqref{est:end_point_buildblk},
\begin{align} \label{claimprfoof}
    \jj_\ts(\mu^\ts) \le \int_{\Omega^2} \jj_\ts(\mu^\ts_{(x,y)})\, \rd \gamma(x,y) \le \mb{C  \int_{\Omega^2} c(x, m, y, m)}\, \rd \gamma(x,y) + C \rho(\Omega) \si\,.
\end{align}
Recalling \eqref{estwft} and \eqref{auxerr_14} above, we find, by the definition \eqref{auxerr_16} \mb{of $m$},  
\begin{equation*}
 \int_{\Omega^2} \mb{c(x, m, y, m)}\, \rd \gamma(x,y) = O(\ep)\,.   
\end{equation*}
Then, the desired \eqref{aimwfr} follows from \eqref{claimprfoof}.

We next show that to prove the claim \eqref{est:end_point_buildblk}, it suffices to construct the mid-states $\mu_i^\si = (\w{\rho}^\si, q^\si_i, r^\si_i)$ for $i = 0,1$ such that with $x_0 = x$ and $x_1 = y$, 
\begin{align} \label{aim_a}
     - \ddiv_\si q_i^\si  + r_i^\si  = \w{\rho}^\si - \mb{\ss_{\ol{{\rm X}}_\si}(m \d_{x_i})}\,,
\end{align}
 and for some $C > 0$ independent of \mb{$(x,y, m)$}, 
\begin{align} \label{aim_b}
    \jj_\si(\mu^\si_i) \le \mb{C(c(x,m,y,m) + m \si)}\,.
\end{align}
 Indeed, similarly to \cite[Proposition 3.3]{lavenant2019unconditional}, we let $\chi(t) := 4t^2$ if $t \le 1/2$; $4(1-t)^2$ if $t \ge 1/2$,
satisfying $\chi(0) = \chi(1) = 0$, $\chi(1/2) = 1$, and $\int_0^1 \dot{\chi}^2/\chi\, \rd s < \infty$. Then, with $\mu_i^\si$ in \eqref{aim_a}, we define $\mu^\ts = (\rho^\ts,q^\ts,r^\ts)$ by 
\begin{align*}
    \rho_k^\ts = \begin{cases}
        (1- \chi(k\tau))\ss_{\ol{{\rm X}}_\si}(\mb{m \d_{x}}) + \chi(k \tau) \w{\rho}^\si\,, & k \le N/2\,, \\
        (1- \chi(k\tau))\ss_{\ol{{\rm X}}_\si}(\mb{m \d_{y}}) + \chi(k \tau) \w{\rho}^\si\,, & k \le N/2 \,, 
    \end{cases} \q k = 0,\ldots,N\,,
\end{align*}
and for $k= 1, \ldots, N$,
\begin{align*}
    q^\ts_k = \frac{\chi(k\tau)-\chi((k-1)\tau)}{\tau} \begin{cases}
    q_0^\si\,, & k \le N/2\,,\\
    q_1^\si\,, & k  > N/2\,,
    \end{cases}\q  r^\ts_k = \frac{\chi(k\tau)-\chi((k-1)\tau)}{\tau} \begin{cases}
    r_0^\si\,, & k \le N/2\,,\\
    r_1^\si\,, & k  > N/2\,,
    \end{cases} 
\end{align*}
where, without loss of generality, we let $N$ be even (otherwise, we can modify the function $\chi(t)$). One can show that by \eqref{aim_a}, $\mu^\ts$ satisfies the discrete continuity equation with endpoints \mb{$\ss_{\ol{{\rm X}}_\si}(m \d_x)$ and $\ss_{\ol{{\rm X}}_\si}(m \d_y)\big)$}:
\begin{align*}
     \tau^{-1}(\rho^{\ts}_{k} - \rho^{\ts}_{k-1}) = - \ddiv_\sigma q^{\ts}_k + r^{\ts}_k\,, \q  k = 1,2,\ldots, N\,,
\end{align*}
 and 
 the desired estimate $\eqref{est:end_point_buildblk}$ by \eqref{aim_b} and the same calculation as in the proof of \cite[Proposition 3.3]{lavenant2019unconditional}.

We only prove \eqref{aim_a} and \eqref{aim_b} for the case $i = 0$. We start with the construction at the continuous level. \mb{Noting $|x-y| < \pi$ by \eqref{auxerr_16}, let $\rho(t) = \w{m}(t)\d_{\gamma(t)}$ for $t \in [0,1]$ be the unique geodesic between $m \d_{x}$ and $m \d_{y}$} defined in \eqref{eq_geo_1} and \eqref{eq_geo_2}. 
We define the measure $\w{\rho} = \int_0^1 \w{m}(s) \d_{\gamma(s)}\,\rd s \in \M(\Omega,\R_+)$, which is supported on the line $\gamma([0,1])$. By the absolute continuity of $\w{m}(t)$ and $\gamma(t)$, a direct calculation gives 
\begin{align} \label{auxerr_10}
    \l\w{\rho} - m \d_{x},  \phi\r_{\Omega} & = \int_0^1 \w{m}(t)\phi(\gamma(t)) - \w{m}(0) \phi(x)\, \rd t \notag\\
    & = \int_0^1 \int_0^t  \frac{\rd}{\rd s} (\w{m}(s)\phi(\gamma(s)))\, \rd s \rd t = \int_0^1 (1 - s)  \frac{\rd}{\rd s}  (\w{m}(s)\phi(\gamma(s)))\, \rd s \notag \\
    & = \int_0^1 (1 - s)  \big(\dot{\w{m}}(s)\phi(\gamma(s)) + \w{m}(s) \na \phi(\gamma(s)) \dd \dot{\gamma}(s)\big) \, \rd s \,.
\end{align}
We further define $q_0 = \int_0^1 (1-s) \w{m}(s) \dot{\gamma}(s) \d_{\gamma(s)}\, \rd s \in \mc{M}(\Omega, \R^d)$ and $r_0 = \int_0^1 (1-s) \dot{\w{m}}(s) \d_{\gamma(s)}\, \rd s \in \mc{M}(\Omega, \R)$. Then, the formula \eqref{auxerr_10} gives $\l\w{\rho} - m\d_{x},  \phi\r_{\Omega} = \l r_0, \phi\r_{\Omega} + \l q_0, \na \phi\r_{\Omega}$, which is equivalent to 
\begin{align*}
    - \ddiv q_0 + r_0  = \w{\rho} - m \d_{x}\,.
\end{align*} 
Noting the densities of $q_0$ and $r_0$ with respect to 
$\w{\rho}$ are given by $(1 - s) \dot{\gamma}(s)$ and $(1 - s) \frac{\dot{\w{m}}(s)}{\w{m}(s)}$ at the point $\gamma(s)$, respectively, we can compute the associated action $\jj_\Omega(\h{\rho},q_0,r)$ as
\begin{align} \label{continuous_est}
    \jj_\Omega(\h{\rho},q_0,r_0) = \frac{1}{2} \int_0^1  (1 - s)^2 \Big( |\dot{\gamma}(s)|^2 + \Big|\frac{\dot{\w{m}}(s)}{\w{m}(s)}\Big|^2 \Big) \w{m}(s)\, \rd s \le \mb{c(x, m, y, m)}\,,
\end{align}
\mb{thanks to $\int_0^1  ( |\dot{\gamma}(s)|^2 + |\dot{\w{m}}(s)/\w{m}(s)|^2 ) \w{m}(s)\, \rd s = 2  c(x, m, y, m)$} by \eqref{costwfr} and recalling that
 $\w{m}(t) \d_{\gamma(t)}$ is the geodesic. We shall prove that $$\mu^\si := (\w{\rho}^\si, q^\si, r^\si) := (\ss_{\ol{{\rm X}}_\si}(\w{\rho}), \ss_{{\rm Y}_\si}(q_0), \ss_{\ol{{\rm X}}_\si}(r_0))$$ is the desired discrete mid-state. One can check from definitions \eqref{defsdd} and \eqref{def:dis_deriv} that $\ddiv_\si \ss_{{\rm Y}_\si} (q_0) = \ss_{\ol{{\rm X}}_\si}(\ddiv q_0)$, 
which yields that $\mu^\si$ satisfies the discrete equation \eqref{aim_a}. 
We next estimate the discrete action:
 \begin{align} \label{eq:actwfr}
        \jj_\si(\mu^\si) =  \frac{1}{2} \sum_{K \in \TT^\si} |K|  \Big( \frac{1}{d+1} \sum_{v \in \V_K^\si} {\w{\rho}}_v^\si \Big)^\dag |q^\si_K|^2 + \frac{1}{2} \sum_{v \in \V^\si} |T_v|  \big(\w{\rho}^\si_v\big)^{\dag} |r^\si_v|^2\,.
\end{align}

Before proceeding, for the reader's convenience, we recall
\begin{subequations}
\begin{align}
 & \w{\rho}^\si_K := \frac{1}{|K|} \int_{\{s\,;\ \gamma(s) \in K\}}  \w{m}(s)\, \rd s \,, \label{eq_a}  \\ 
 & 
 \w{\rho}^\si_v := \frac{1}{|T_v|} \int_0^1  \w{m}(s) \h{\phi}_v(\gamma(s))\, \rd s \,, \label{eq_b}  \\ 
& q^\si_K := \frac{1}{|K|} \int_{\{s\,;\ \gamma(s) \in K\}}  (1 - s) \w{m}(s) \dot{\gamma}(s)\, \rd s \,, \label{eq_c} \\
& r^\si_v := \frac{1}{|T_v|} \int_0^1 (1-s) \dot{\w{m}}(s) \h{\phi}_v(\gamma(s))\, \rd s  \,, \label{eq_d}
\end{align}
\end{subequations}
and for ease of exposition, introduce the following auxiliary quantities: 
\begin{subequations}
\begin{align}  
& \TT^\si_{\sim} := \{K \in \TT^\si\,; \  |\gamma([0,1]) \cap K| > 0\}\,, \label{eqq_a}\\
 &  s_K : = \arg\max\{|\dot{\gamma}(s)|\,;\  \gamma(s) \in K \}\q \text{for} \ K \in \TT^\si_\sim\,, \label{eqq_b} \\ 
 &  s_v : = \arg\max\{|\dot{\w{m}}/\w{m}|\,;\  \gamma(s) \in \cup_{K \in \TT_v^\si} K  \}\q  \text{if} \ \big|\gamma([0,1]) \cap \big(\cup_{K \in \TT_v^\si} K \big)\big| > 0  \,,  \label{eqq_c} \\
 & c_{{\rm min}} := \min_{s \in [0,1]}|\dot{\gamma}(s)|\,,\q c_{{\rm max}} := \max_{s \in [0,1]}|\dot{\gamma}(s)| \,,\q C_{{\rm Max}} := \max_{s \in [0,1]} |\dot{\w{m}}(s)/\w{m}(s)|\,, \label{eqq_d}  
\end{align}
\end{subequations}
where we view the simplex $K \in \TT^\si$ as a closed set, and in \eqref{eqq_a} and \eqref{eqq_c}, we consider the 1-dimensional Hausdorff measure on the line $\gamma([0,1])$ (still denoted by $|\dd|$). In what follows, we denote by ${\rm Lip}(\dd)$ the Lipschitz constant of a function and let $C > 0$ be a generic constant depending on $\Omega$ but independent of $\si$ and \mb{$(x,y,m)$}. We now prepare some estimates for later use. First, recall from \eqref{auxerr_16} that $|x - y| \le C \ep$ with $C > 0$.  The traveling mass $\w{m}(s)$ in \eqref{eq_geo_1} can be written as $\w{m}(s) = m f(s)$ with 
\begin{align*}
   f(s) = 1 + (2 s^2 - 2 s)(1 - \cos(|x - y|/2))\,.
\end{align*}
Moreover, we can estimate, thanks to $|x-y| = O(\ep)$,
\begin{equation} \label{limnorm}
   \mb{ \Lip(\w{m}) = m \norm{f'}_\infty \le C m |x-y|}\,,
\end{equation}
and 
\begin{equation}  \label{limnorm22}
  C_{{\rm Max}} \le C |x - y|\,, \q   {\rm Lip}\Big(\frac{\dot{\w{m}}}{\w{m}}\Big) = {\rm Lip}\Big( \frac{f''f - (f')^2}{f^2} \Big) \le C |x-y|\,,
\end{equation}
where $C_{{\rm Max}}$ is given in \eqref{eqq_d}. Similarly, we have, for $\gamma(s)$ connecting $x$ and $y$ \eqref{eq_geo_2},
\begin{align} \label{estminmax}
C^{-1} |x-y| \le c_{{\rm min}} \le c_{{\rm max}} \le C|x-y|\,.
\end{align}
We also note that
$\{s\,;\ \gamma(s) \in K\}$ is a closed interval with the length
\begin{align} \label{eqqq}
    l_K: =  |\{s\,;\ \gamma(s) \in K\}| \le c_{{\rm min}}^{-1} \,|\gamma([0,1]) \cap K| \le c_{{\rm min}}^{-1}\, \si  \,.
\end{align}

We are ready to finish the proof. We start with the estimation of the first term in \eqref{eq:actwfr}. By the mesh regularity and \eqref{eq:userela_tri}, we see that there exists $C > 0$ such that $ C^{-1} |K| \le  |T_v| \le C |K|$ for any $v \in \V_K^\si$, which gives
\begin{align*}
    \frac{1}{d + 1} \sum_{v \in \V_K^\si} \w{\rho}^\si_v \ge \frac{1}{C |K|} \sum_{v \in \V_K^\si} \l \w{\rho}, \h{\phi}_v\r_{\Omega} \ge  C^{-1} \w{\rho}^\si_K\,.
\end{align*}
Then, it follows that, noting that we can replace the sum over $K \in \TT^\si$ by $K \in \TT^\si_\sim$,
\begin{align}\label{auxerr_11}
   \sum_{K \in \TT^\si} |K|  \Big( \frac{1}{d+1} \sum_{v \in \V_K^\si} {\w{\rho}}_v^\si \Big)^\dag |q^\si_K|^2  \le C \sum_{K \in \TT_\sim^\si} |K| \left(\w{\rho}_K^\si\right)^{\dag} |q^\si_K|^2\,.
\end{align}
Recalling \eqref{eq_a}, \eqref{eq_c}, and \eqref{eqq_b}, we have 
\begin{equation} \label{aux_a}
    |q^\si_K| \le \frac{1}{|K|}\, |\dot{\gamma}(s_K)| \int_{\{s\,;\ \gamma(s) \in K\}}  \w{m}(s)\, \rd s =  |\dot{\gamma}(s_K)| \, \w{\rho}^\si_K \,, \q K \in \TT_\sim^\si\,,
\end{equation}
while, by \eqref{eq_a}, there holds
\begin{align} \label{aux_b}
  |K|\, \w{\rho}^\si_K  &\le \int_{\{s\,;\ \gamma(s) \in K\}} \left(\w{m}(s_K) + {\rm Lip}(\w{m}) |s - s_K|\right) \rd s \notag \\
  & \le \w{m}(s_K) l_K +  C {\rm Lip}(\w{m}) l_K^2 \,.
\end{align}
Substituting \eqref{aux_a} into \eqref{auxerr_11} and then using \eqref{aux_b} with \eqref{eqqq} imply
\begin{align} 
    \sum_{K \in \TT^\si} |K|  \Big( \frac{1}{d+1} \sum_{v \in \V_K^\si} {\w{\rho}}_v^\si \Big)^\dag |q^\si_K|^2 & \le C  \sum_{K \in \TT_\sim^\si} \left|K\right|  \left|\dot{\gamma}(s_K)\right|^2 \w{\rho}^\si_K \notag
     \\
     & \le C  \sum_{K \in \TT_\sim^\si}  \left|\dot{\gamma}(s_K)\right|^2 ( \w{m}(s_K) l_K +  C {\rm Lip}(\w{m}) l_K^2) \notag
     \\
& \le C \sum_{K \in \TT_\sim^\si} |\dot{\gamma}(s_K)|^2\, \w{m}(s_K) l_K + C c_{{\rm max}}^2 \, {\rm Lip}(\w{m}) \sum_{K \in \TT_\sim^\si}  l_K^2\,,\label{aux_c}
\end{align} 
where the remainder term can be estimated by \eqref{limnorm}, \eqref{estminmax}, and \eqref{eqqq} with $|x-y| \le {\rm diam}(\Omega)$: 
\begin{align}  \label{aux_cc}
    c_{{\rm max}}^2 \, {\rm Lip}(\w{m}) \sum_{K \in \TT_\sim^\si}  l_K^2  \le C (c_{{\rm max}}/c_{\min})^2  {\rm Lip}(\w{m}) |x-y| \si \le \mb{C m \si}\,.
\end{align}

We next consider the second term  in \eqref{eq:actwfr}. By \eqref{eq_d} and  \eqref{eqq_c}, we have 
\begin{align*}
    |r^\si_v| \le \frac{1}{|T_v|} \Big| \frac{\dot{\w{m}}(s_v)}{\w{m}(s_v)} \Big| \int   \w{m}(s) \h{\phi}_v(\gamma(s)) \rd s = \Big| \frac{\dot{\w{m}}(s_v)}{\w{m}(s_v)} \Big|\, \w{\rho}_v^\si\,.
\end{align*}
Then, similarly to the estimates above, it follows that, by \eqref{limnorm22} and \eqref{eqqq},
\begin{align}\label{est2nd}
     \sum_{v \in \V^\si} |T_v|  (\w{\rho}^\si_v)^{\dag} |r^\si_v|^2
\le &  \sum_{v \in \V^\si} |T_v| \Big| \frac{\dot{\w{m}}(s_v)}{\w{m}(s_v)} \Big|^2\, \w{\rho}_v^\si \notag \\
\le & \sum_{v \in \V^\si} \sum_{K \in \TT_v^\si} \big\l \w{\rho}, \h{\phi}_v \big\r_K \Big(\frac{\dot{\w{m}}(s_K)}{\w{m}(s_K)} + {\rm Lip}\Big(\frac{\dot{\w{m}}}{\w{m}}\Big)\,|s_v - s_K| \Big)^2  \notag  \\
\le & C \sum_{K \in \TT^\si} \sum_{v \in \V_K^\si} \l \w{\rho}, \h{\phi}_v\r_K \Big(\frac{\dot{\w{m}}(s_K)}{\w{m}(s_K)}\Big)^2 + C \norm{\w{\rho}}_{\rm TV} \si^2  \,,
\end{align}
with the total variation of $\w{\rho}$:
\begin{equation*}
 \norm{\w{\rho}}_{\rm TV} = \l \w{\rho}, I\r_{\Omega} =  \mb{m \int_0^1 f(s) \rd s  \le C m}\,. 
\end{equation*}
Noting $\sum_{v \in \V_K^\si} \l \w{\rho}, \h{\phi}_v\r_K = \left|K\right| \w{\rho}^\si_K$, by \eqref{aux_b} with \eqref{eqq_d}, we arrive at, from \eqref{est2nd}, 
\begin{align} \label{auxeq_d}
       \sum_{v \in \V^\si} |T_v|  (\w{\rho}^\si_v)^{\dag} |r^\si_v|^2  & \le C \sum_{K \in \TT^\si_\sim} \left|K\right| \w{\rho}^\si_K \Big(\frac{\dot{\w{m}}(s_K)}{\w{m}(s_K)}\Big)^2 + C m \si^2   \notag \\
       & \le C \sum_{K \in \TT^\si_\sim}  \left(\w{m}(s_K) l_K +  C {\rm Lip}(\w{m}) l_K^2\right) \Big(\frac{\dot{\w{m}}(s_K)}{\w{m}(s_K)}\Big)^2 + C m \si^2 \notag \\
       & \le C \sum_{K \in \TT^\si_\sim} \frac{\dot{\w{m}}(s_K)^2}{\w{m}(s_K)} l_K  + C m  \si^2 \,,
\end{align}
where we have used, by \eqref{eqq_d}, \eqref{limnorm22}, and \eqref{eqqq},
\begin{equation*}
   \sum_{K \in \TT^\si_\sim}   {\rm Lip}(\w{m}) l_K^2 \Big(\frac{\dot{\w{m}}(s_K)}{\w{m}(s_K)}\Big)^2 \le C_{\rm Max}^2 \sum_{K \in \TT^\si_\sim} {\rm Lip}(\w{m}) \, l_K^2 \le C m \si\,.
\end{equation*}
Since $\rho = \w{m}(t)\d_{\gamma(t)}$ is the geodesic which is constant-speed, we have 
\begin{equation*}
    |\dot{\gamma}(s_K)|^2 \w{m}(s_K) + \frac{\dot{\w{m}}(s_K)^2}{\w{m}(s_K)} = 2 c(x,m,y,m)\,,
\end{equation*}
which, combining with \eqref{eq:actwfr}, \eqref{aux_c}, and 
\eqref{auxeq_d}, gives the desired \eqref{aim_b}:
$
    \jj_\si(\mu^\si) \le \mb{C(c(x,m,y,m) + m \si)}
$.
We have proved the claim \eqref{est:end_point_buildblk} and hence Theorem \ref{lem:reg_points}.

\section{Concluding remarks}

This is a follow-up work of \cite{li2020general1}, where we delve into the convergence analysis of the discretization of the general matrix-valued OT model ${\rm WB}_\Lad$ \eqref{eq:distance}. We have established an abstract convergence framework in the spirit of \cite{lavenant2019unconditional} and suggested a specific discretization scheme inspired by the FEM theory. These findings could directly apply to the Kantorovich-Bures distance, the WFR distance, and the matricial interpolation distance proposed in \cite{brenier2020optimal,chen2019interpolation,chizat2018interpolating}.
A promising future research direction lies in the development of efficient algorithms for solving the resulting discrete optimization problem \eqref{prob:discre_ot} and its practical applications in imaging processing and learning problems. 

\hspace*{2 cm}

\noindent 
\mb{{\bf Acknowledgements:} The authors 
would like to thank the anonymous referees and editors for their careful reading and constructive comments and suggestions, which have helped us improve this work.}

\titleformat{\section}{\bfseries}{\appendixname~\thesection .}{0.5em}{}
\titleformat{\subsection}{\normalfont\itshape}{\thesubsection.}{0.5em}{}

\appendices

\section{Proof of Proposition  \ref{prop:regular_ep}}\label{appa}

We first recall \cite[Lemma 3.15]{li2020general1}.

\begin{lemma}\label{lemma:timescaling} 
For $\mu \in \ce_\infty([0,1];\ms{G}_0,\ms{G}_1)$, let $\ms{s}(t):[0,1] \to [a,b]$ be strictly increasing absolutely continuous with an absolutely continuous inverse $\ms{t} = \ms{s}^{-1}$. Then,
     $\w{\mu} := \int_a^b \d_s \otimes (\ms{G}_{\ms{t}(s)}, \ms{t}'(s) \ms{q}_{\ms{t}(s)},  \ms{t}'(s) \ms{R}_{\ms{t}(s)})\, \rd s \in \ce([a,b]; \ms{G}_0,\ms{G}_1)$ and
    \begin{align}  \label{eq:timescinlemma}
        \int_0^1 \ms{t}'(\ms{s}(t)) \jj_{\Lad,\Omega}(\mu_t)\, \rd t = \int_a^b \jj_{\Lad,\Omega}(\w{\mu}_s) \,\rd s\,.
    \end{align} 
In addition, for a diffeomorphism $T$ on $\R^d$, suppose that there exists $\TT_{\ms{D^*}}(x): \Omega \to \L(\R^{n \t k})$ such that
\begin{align} \label{def:tdlinear}
    \TT_{\ms{D^*}}[(\ms{D^*} \Phi)\circ T] := \ms{D^*} (\Phi\circ T)\,,\q \forall\Phi \in C_c^\infty(\R^d, \S^n)\,.
\end{align}
Then, $\w{\mu} := \int_0^1 \d_t \otimes T_{\#} (\ms{G}_{t}, \TT_{\ms{D}} \ms{q}_{t},  \ms{R}_{t})\, \rd t \in \ce([0,1]; T_{\#}\ms{G}_0, T_{\#}\ms{G}_1)$ on $T(\Omega)$, where $T_{\#}(\dd)$ is the pushforward by $T$  and $\TT_{\ms{D}}$ is the transpose of $\TT_{\ms{D}^*}$, i.e., $ (\TT_{\ms{D}} q) \dd p = q \dd (\TT_{\ms{D}^*} p)$ for any $p,q \in \R^{n \t k}$. 
\end{lemma}

We now prove Proposition \ref{prop:regular_ep}.
We extend the measure $\mu = \{\mu_t\}_{t \in [0,1]} \in \ce_\infty([0,1],\ms{G}_0,\ms{G}_1)$ by, for $t \in \R$,
\begin{align*}
    \mu_t: = \mu_t\chi_{(0,1)}(t) + (\ms{G}_0,0,0)\chi_{(-\infty,0]}(t) + (\ms{G}_1,0,0)\chi_{[1,+\infty)}(t)\,,
\end{align*} 
and regard $\mu_t$ as a measure on $\R^d$ by zero extension.
We hence have $\mu \in \mc{M}(\R^{1+d},\xx)$. 
We define the $\ep$--neighborhood of $\Omega$ by
\begin{equation}\label{auxdef1}
\Omega_\ep = \{x \in \R^d\,; \ {\rm dist}(x, \Omega) \le \ep\}\,.    
\end{equation}
The proof proceeds in four steps.

 \smallskip 
\noindent \emph{Step 1 (lift)}. Let $\ol{G} \in C^\infty_c(\R^{d},\S_+^n)$ be a smooth function supported in $\Omega_1$ and satisfying $\ol{G} \succeq  I$ on $\Omega_{1/2}$. We identify the density $\ol{G}$ and the measure $\ol{G}\, \rd t \otimes \rd x$. For small $\d > 0$, we define   
    \begin{equation} \label{def:liftmu}
      \mu^\d = \ms{(G^\d,q^\d,R^\d)} = ((1- \d)\ms{G} + \d \ol{G}, (1 - \d)\ms{q}, (1- \d) \ms{R})\in \M(\R^{1 + d},\xx)\,,
    \end{equation}
which satisfies the continuity equation
   $\p_t \ms{G}^\d + \ms{D}\ms{q}^\d =  (\ms{R}^\d)^{{\rm sym}}$.
    Note $\ms{G}^\d(E) \succeq \d |E| I$ for any $E \in \mathscr{B}(\R \t \Omega_{1/2})$, where $|E|$ denotes the Lebesgue measure of $E$. 
    Moreover, by the convexity of $\jj_{\Lad, \mathcal{X}}$, we have 
    \begin{equation} \label{est:cost_lift}
        \jj_{\Lad, \R^{1 + d}}(\mu^\d) \le (1 - \d)\jj_{\Lad,\R^{1 + d}}(\mu) + \d \jj_{\Lad,\R^{1 + d}}(\ol{G},0,0) = (1 - \d)\jj_{\Lad,Q}(\mu) < +\infty\,.
    \end{equation}

    \noindent \emph{Step 2 (time-space regularization)}. We introduce a 
family of convolution kernels 
\begin{equation}\label{auxdef2}
\theta_d^\ep(x) := \ep^{-d}\theta_d(x/\ep)\,,    
\end{equation}
where $\theta_d : \R^{d} \to \R_+$ is a $C^\infty$-smooth radial nonnegative function supported in $B_d(0,1)$ with $\int_{\R^{d}} \theta_d(x) \rd x = 1$.  Then, the convolution kernel on $\R^{1 + d}$ is defined by $\rho_\ep (t,x): = \theta^\ep_1(t)\theta^\ep_d(x)$. 
To alleviate notations, in the rest of this proof, we still denote by $\mu$ the measure $\mu^\d$ in \eqref{def:liftmu}. We now define  
     $\w{\mu}^\ep := (\rho_\ep * \mu)|_{[-\ep,1+\ep] \t \Omega_\ep}$, which
 has a $C^\infty$-smooth density $(\w{G}^\ep,\w{q}^\ep,\w{R}^\ep)$ on $[-\ep,1+\ep] \t \Omega_\ep$ with respect to
$\rd t \otimes \rd x$ with $\w{G}^\ep \succeq \d I$. Then, by the linearity of the continuity equation and basic properties of the convolution, there holds $\w{\mu}^\ep = \big\{(\w{\ms{G}}^\ep_t, \w{\ms{q}}^\ep_t, \w{\ms{R}}^\ep_t)\big\}_{t \in [-\ep,1+\ep]} \in \ce\big([a,b]; \w{\ms{G}}^\ep_t|_{t = a}, \w{\ms{G}}^\ep_t|_{t = b}\big)$ on $\Omega_\ep$ for any $-\ep\le a < b \le 1+\ep$.

\smallskip 
\noindent  \emph{Step 3 (compression)}. By assumption \eqref{1} for $\Omega$, without loss of generality, we let $0$ be in the interior of the set of points with respect to which $\Omega$ is star-shaped. By \cite[Theorem  5.3]{rubinov2013abstract}, the associated  Minkowski functional 
\begin{equation*}
p_{\Omega}(x) : = \inf\{\lad > 0\,; \ x\in \lad \Omega\}     
\end{equation*}
is Lipschitz with the Lipschitz constant denoted by $L$. Thus,  $p_{\Omega}(x) \le 1 + L \ep$ holds for any $x \in \Omega_\ep$. Then, we can define the time-space scaling: 
\begin{equation} \label{aux_tsscaling}
     \ms{s}^\ep(t) = (1 + 2 \ep)^{-1}(t + \ep):[-\ep,1+\ep] \to [0,1]\,,\q 
    T^\ep(x) = (1 + L \ep)^{-1}x: \Omega_\ep \to \Omega\,. 
\end{equation}
By assumption \eqref{2} for $\ms{D}^*$, the linear map $\TT^\ep_{\ms{D}^*}$ for $T^\ep$ satisfying \eqref{def:tdlinear} is given by $\TT^\ep_{\ms{D}^*}(p) = (p_1, (1 + L \ep)^{-1}p_2) $ for $p = (p_1,p_2)$ with $p\in \R^{n \t k_1}$, $p\in \R^{n \t k_2}$, and $k = k_1 + k_2$. Recalling Lemma \ref{lemma:timescaling}, we let  
    \begin{equation} \label{def:timereg}
      \mu^\ep_s :=  T^\ep_{\#} \Big(\w{\ms{G}}^\ep_{\ms{t}^\ep(s)}, (1 + 2 \ep) \TT^\ep_{\ms{D}}\big(\w{\ms{q}}^\ep_{\ms{t}^\ep(s)}\big), (1 + 2 \ep) \w{\ms{R}}^\ep_{\ms{t}^\ep(s)}\Big) \,, \q s \in [0,1]\,,
    \end{equation} 
which satisfies the continuity equation,
where $\ms{t}^\ep = (\ms{s}^\ep)^{-1}$ and $\TT_{\ms{D}}^\ep$ is the transpose of $\TT_{\ms{D}^*}^\ep$. 

\medskip 

\noindent \emph{Step 4 (estimates)}.   We proceed to estimate the action functional: for some $C > 0$,
    \begin{align} \label{est:cost_regul}
        \jj_{\Lad, Q}(\mu^\ep) & =   \int_0^1 \jj_{\Lad, \Omega}(\mu_t^\ep)\, \rd t = (1 + 2 \ep) \int_{-\ep}^{1+\ep} \jj_{\Lad,\Omega} \left(T^\ep_{\#} \left(\w{\ms{G}}^\ep_{s},  \TT^\ep_{\ms{D}}\big(\w{\ms{q}}^\ep_{s}\big), \w{\ms{R}}^\ep_{s}\right)\right)\rd s \notag\\
    & \le (1 + 2 \ep) \int_{-\ep}^{1+\ep}  \jj_{\Lad, \Omega^\ep} \Big(\Big(\w{\ms{G}}^\ep_{s},  \TT^\ep_{\ms{D}}\big(\w{\ms{q}}^\ep_{s}\big), \w{\ms{R}}^\ep_{s}\Big)\Big)\, \rd s \notag\\
    & \le (1 + C \ep) \int_{-\ep}^{1+\ep}  \jj_{\Lad, \Omega^\ep} \left(\left(\w{\ms{G}}^\ep_{s},  \w{\ms{q}}^\ep_{s}, \w{\ms{R}}^\ep_{s}\right)\right) \rd s \notag \\
    & \le (1 + C \ep)  \jj_{\Lad,\R^{1+d}}(\rho_\ep*\mu)\,,
    \end{align}
where the first line is from \eqref{eq:timescinlemma}; the second line is from the change of variables $x = T^\ep(y)$;
the third line is by definition of $\TT_{\ms{D}}^\ep$ and $\jj_{\Lad,\Omega^\ep}$; the fourth inequality is from the fact that $\w{\mu}^\ep$ is the restriction of $\rho_\ep * \mu$ on $[-\ep,1+\ep] \t \Omega_\ep$. It is easy to show that the action functional $\jj_{\Lad,\R^{1 + d}}$ is monotone with respect to the smoothing by convolution; see \cite[Lemma 8.1.10]{ambrosio2008gradient}. 
Then, by \eqref{est:cost_lift} and \eqref{est:cost_regul}, we arrive at 
\begin{align} \label{est:one-side_action}
    \jj_{\Lad, Q}(\mu^\ep) \le (1 + C \ep)  (1 - \d)\jj_{\Lad,Q}(\mu)\,.
\end{align}

We next set $\d = \ep$ and show that $\mu^\ep$ constructed above weak* converges to the measure $\mu$, which, by \eqref{est:one-side_action} and the lower semicontinuity of $\jj_{\Lad,Q}(\mu)$,  readily implies 
\begin{equation*}
    \lim_{\ep \to 0} \jj_{\Lad, Q}(\mu^{\ep}) = \jj_{\Lad, Q}(\mu)\,.
\end{equation*}
It suffices to consider the weak* convergence of $\ms{G}^\ep$ since the same analysis applies to $\ms{q}^\ep$ and $\ms{R}^\ep$. We write 
$\T^\ep(t,x) = (\ms{s}^\ep(t),T^\ep(x))$ with $\ms{s}^\ep$ and $T^\ep$ given in \eqref{aux_tsscaling}. By the construction of $\mu^\ep$, a direct computation gives
\begin{align} \label{auxest_convol}
 |\l \ms{G}^\ep, \Phi \r_{Q} - \l \ms{G}, \Phi \r_{Q}| 
 & = \left|\frac{1}{1 + 2 \ep} \int_{[-\ep,1+\ep] \t \Omega_\ep} \Phi \circ \T^\ep \, \rd \left(\rho_\ep * \left((1- \ep)\ms{G} + \ep \ol{G}\right)\right)   - \l \ms{G}, \Phi \r_{Q} \right| \notag \\
 & \le  \left|\frac{1 - \ep}{1 + 2 \ep} \int_{\R^{1+d}} \rho_\ep *\left(\chi_{[-\ep,1+\ep] \t \Omega_\ep} \Phi\circ \T^\ep  \right) \rd \ms{G} - \l \ms{G}, \Phi \r_{Q} \right| + C \ep\,, 
\end{align} 
for $\Phi \in C(Q,\S^n)$, where the constant $C$ depends on $\norm{\Phi}_\infty$ and $\norm{\ol{G}}_\infty$. Moreover, we note that
for any $(t,x) \in Q$ and $(s,y) \in B_1(t,\ep) \t B_d(x,\ep)$, there hold $\T^\ep(s,y) \in Q$ and 
\begin{align*}
    |\T^\ep(s,y) - (t,x)| \le \left|\frac{s + \ep - t - 2 \ep t}{1 + 2 \ep}\right| + \left|\frac{y - x - L \ep x}{1 + L \ep}\right|  \le C \ep \,.
\end{align*}
Then, by uniform continuity of $\Phi$ on $Q$, we have the pointwise convergence: 
\begin{align*}
    & \limsup_{\ep \to 0} \normm{\rho_\ep *(\chi_{[-\ep,1+\ep] \t \Omega_\ep} \Phi \circ \T^\ep )(t,x) - \Phi(t,x)}_{\ff} \\
\le & \limsup_{\ep \to 0} \int_{t - \ep}^{t + \ep} \int_{B_d(x,\ep)} \rho_\ep(t-s,x-y) \normm{\Phi \circ \T^\ep(s,y) - \Phi(t,x)}_\ff\, \rd y \rd s = 0\,,
\end{align*}
which, along with \eqref{auxest_convol}, implies $\lim_{\ep \to 0} \l \ms{G}^\ep, \Phi \r_{Q} = \l \ms{G}, \Phi \r_{Q}$. 

For the third statement in the proposition, we only consider the convergence of the initial distribution at $t = 0$. By the construction, we have
\begin{align} \label{def:epregend_points}
    \ms{G}^\ep|_{t = 0} =  T_{\#}^\ep(\theta_d^\ep * ((1 - \ep)\ms{G}_0 + \ep \ol{G})|_{\Omega_\ep})\,.
\end{align}
A similar estimate as above shows the weak* convergence of $\ms{G}^\ep|_{t = 0}$. We next estimate ${\rm WB}_\Lad(\ms{G}^\ep|_{t = 0},\ms{G}_0)$ under the assumption \eqref{3}. For this, by estimates \eqref{est:lip_squre} and \eqref{eq:basic_bound},   we have 
\begin{align} \label{est:intial_pre}
    {\rm WB}_{\Lad}(\ms{G}^\ep|_{t = 0}, \ms{G}_0) &\lesssim \int_{\Omega}  \normm{(\theta^\ep_d * G_0) \circ (T^\ep)^{-1}  - G_0}_\ff \rd x + \ep \notag \\
    & \lesssim \int_{\Omega_\ep} \normm{ \theta^\ep_d * G_0 - G_0 \circ T^\ep }_\ff \rd x +  \ep  \notag \\
 & \lesssim \norm{\theta^\ep_d * G_0 - G_0}_{L^1(\R^d)} + \norm{G_0 -   G_0 \circ T^\ep}_{L^1(\Omega_\ep)} + \ep \,. 
\end{align}
Then, the property of \emph{approximations to the identity} gives $\norm{\theta^\ep_d * G_0 - G_0}_{L^1(\R^d)} \to 0$ as $\ep \to 0$, while by the density of $C^\infty_c(\R^d,\S^n)$ in  
$L^1(\R^d,\S^n)$, it is easy to see $\norm{G_0 - G_0 \circ T^\ep}_{L^1(\Omega_\ep)} \to 0$ as $\ep \to 0$. Hence, \eqref{est:intial_pre} gives \eqref{est:conver_ini} as desired and completes the proof.

\end{document}